\documentclass[sn-mathphys-num]{sn-jnl}

\usepackage[export]{adjustbox}
\usepackage{siunitx}
\usepackage{amsmath,amsfonts,cancel,mathtools,centernot}
\usepackage{fix-cm,mathpazo} 
\usepackage{framed,array,amssymb,tabulary,bm}
\usepackage{aligned-overset}
\usepackage{diagbox}
\usepackage{geometry}
\usepackage{todonotes}

\usepackage{hyperref} 
\usepackage[ruled,vlined,noend]{algorithm2e} 
\SetCommentSty{mycommfont}
\usepackage{algpseudocode}
\usepackage{sepfootnotes}
\usepackage[english]{cleveref}
\usepackage{amsthm} 
\newtheorem{theorem}{Theorem}[section]
\newtheorem{lemma}{Lemma}[section]
\newtheorem{proposition}{Proposition}[section]
\newtheorem{remark}{Remark}[section]
\newtheorem{definition}{Definition}[section]
\newtheorem{corollary}{Corollary}[section]
\theoremstyle{definition}
\newtheorem{example}{Example}[section]

\usepackage{babel} 
\usepackage{enumitem} 
\usepackage{tikz}
\usepackage{booktabs}
\usepackage{multirow}

\usepackage[utf8]{inputenc}
\usepackage[T1]{fontenc}
\usepackage{lmodern} 
\usepackage{mathabx}
\usepackage{float}
\usepackage{cases}

\usepackage{parskip}

\usepackage{graphicx}

\DeclareMathOperator{\F}{\mathcal F}
\usepackage{dsfont}
\DeclareMathOperator{\1}{\mathds{1}}
\DeclareMathOperator{\D}{\mathcal D}
\DeclareMathOperator{\R}{\mathbb R}
\DeclareMathOperator{\N}{\mathbb N}
\DeclareMathOperator{\C}{\mathcal C}
\newcommand*\diff{\mathop{}\!\mathrm{d}}
\DeclareMathOperator{\KL}{KL}
\DeclareMathOperator{\RR}{R}
\DeclareMathOperator{\TT}{T}
\DeclareMathOperator{\SK}{SK}
\DeclareMathOperator{\II}{I}

\DeclareMathOperator{\M}{\mathcal M}
\DeclareMathOperator{\intt}{int}
\DeclareMathOperator{\relint}{relint}
\DeclareMathOperator{\Lip}{Lip}
\DeclareMathOperator{\B}{\mathcal B}

\DeclareMathOperator{\supp}{supp}
\let\P\relax
\DeclareMathOperator{\P}{\mathcal P}

\DeclareMathOperator{\T}{\mathcal T}
\let\eps\varepsilon
\let\phi\varphi
\DeclareMathOperator*{\argmin}{argmin}
\DeclareMathOperator{\tT}{\textrm T}
\DeclareMathOperator{\tF}{\textrm F}
\DeclareMathOperator{\dom}{dom}
\DeclareMathOperator{\OT}{OT}
\newcommand{\FL}{\mathcal L}

\newcommand{\Mat}[1]{\bm{#1}}
\newcommand{\MP}{\Mat P}
\newcommand{\MQ}{\Mat Q}
\newcommand{\MM}{\Mat M}
\newcommand{\MX}{\Mat X}
\newcommand{\MY}{\Mat Y}
\newcommand{\Vf}{\Mat f}
\newcommand{\Vg}{\Mat g}
\newcommand{\Vr}{\Mat r}
\newcommand{\Vs}{\Mat s}
\newcommand{\Vt}{\Mat t}
\newcommand{\Vc}{\Mat c}
\newcommand{\Vq}{\Mat q}
\newcommand{\Vnu}{\Mat \nu}
\newcommand{\Vxi}{\Mat \xi}

\usepackage{titletoc}

\hypersetup{
    colorlinks = false,
    citecolor = black,
    filecolor = black,
    linkcolor = black,
    urlcolor = black,
    linkbordercolor = red,
    breaklinks = true,
    pdfhighlight = /O,
    linktocpage=true 
}

\begin{document}

\title[Interpolating between Optimal Transport 
and KL regularized Optimal Transport
using Rényi Divergences]{\centering Interpolating between Optimal Transport \\
and KL regularized Optimal Transport \\
using Rényi Divergences}


\author*[1]{\fnm{Jonas} \sur{Bresch}}\email{\color{black}bresch@math.tu-berlin.de}

\author*[1]{\fnm{Viktor} \sur{Stein}}\email{\color{black}stein@math.tu-berlin.de}

\affil[1]{\orgdiv{Institute of Mathematics}, \orgname{Technical University Berlin}, \orgaddress{\street{Straße des 17. Juni}, \city{Berlin}, \postcode{10623}, \state{Berlin}, \country{Germany}}}


\abstract{
    Regularized optimal transport (OT) 
    has received much attention in recent years
    starting from Cuturi's introduction of Kullback-Leibler (KL) divergence regularized OT.
    In this paper, we propose regularizing the OT problem
    using the family of $\alpha$-Rényi divergences for $\alpha \in (0, 1)$.
    Rényi divergences are neither $f$-divergences 
    nor Bregman distances, 
    but they recover the KL divergence in the limit $\alpha \nearrow 1$.
    The advantage of introducing the additional parameter $\alpha$ is 
    that for $\alpha \searrow 0$ we obtain convergence to the unregularized OT problem.
    For the KL regularized OT problem, this was achieved by letting the regularization parameter $\eps$ tend to zero, 
    which causes numerical instabilities.
    We present two different ways to obtain premetrics on probability measures, namely
    by Rényi divergence constraints and by penalization.
    The latter premetric interpolates
    between the unregularized and the KL regularized OT problem 
    with weak convergence of the unique minimizer, 
    generalizing the interpolation property of KL regularized OT.
    We use a nested mirror descent algorithm to solve the primal formulation.
    Both on real and synthetic data sets 
    Rényi regularized OT plans outperform their KL and Tsallis counterparts 
    in terms of being closer to the unregularized transport plans 
    and recovering the ground truth in inference tasks better.
}

\maketitle

\section{Introduction}

Optimal transport (OT) is a physically 
and geometrically well-motivated procedure for efficiently transporting mass between probability measures 
$\mu, \nu \in \P(X)$ on a metric space $X$.
However, any OT plan between $\mu$ and $\nu$ 
is only supported on (a subset of) the graph 
of the $c$-subdifferential of a $c$-convex function, 
in particular, on a Lebesgue-null set. 
The solution of the regularized OT problem 
is supported on a non-null set.
Entropic regularization using the negative Shannon entropy 
as the regularizer
was first addressed by Wilson~\cite{W1969} in 1969.
Building upon those ideas,
Cuturi~\cite{C13} introduced the discrete KL regularized OT problem 
by adding the regularizer $\eps \KL(\;\cdot\mid \mu \otimes \nu)$ 
to the OT objective,
speeding up the computation compared to the original OT problem by utilizing the Sinkhorn algorithm.
He showed that for $\eps \searrow 0$ the regularized OT plan 
converges to the unregularized plan.
Unfortunately, 
in practice the regularized plan is only structurally similar 
to the unregularized plan for very small $\eps$,
see \Cref{fig:KL_reg_comp}.
Additionally, the very fast Sinkhorn algorithm suffers from numerical instability 
for those small $\eps$.
Chizat et al.~\cite{CPSV17}
and Schmitzer~\cite{S19}, 
introduced stabilized versions of the vanilla one,
partially addressing these numerical instabilities.

Building upon this point of view,
many subsequent works address other regularizers and algorithms.
For instance,
Muzellec et al.~\cite{MNPN17} introduce a generalization of the discrete KL regularized OT
by regularizing with the $q$-Tsallis entropy but not with the $q$-Tsallis divergence, 
which recover the Shannon entropy and the KL divergence, respectively, for $q = 1$.
To solve the $q$-Tsallis regularized OT problem 
they propose two algorithms.
However, 
they only prove convergence of one algorithm for $q > 1$ 
and only convergence of the other to some plan that fulfills only one marginal constraint.
The case $q = 2$, 
referred to as Gini regularizer, was explored in \cite{LMM21,RRSW17}.
Furthermore,
for $f$-divergences, where $f$ is Legendre -- fulfilled, for instance, 
for the (reverse) KL, (reverse) $\chi^2$, and squared Hellinger divergence, 
but not the total variation%
---Terjék et al.~\cite{TS22} consider the regularization of the continuous OT problem
and propose a generalized Sinkhorn algorithm.
These methods do not apply to our setting,
since the $\alpha$-Rényi divergence does not belong to the class of $f$-divergences, 
for any $\alpha \in (0,1)$.

Another regularizer for the discrete setting,
which is not an $f$-divergence,
but instead is itself a Legendre function
is discussed in \cite{KTM20}, 
with an emphasis on the statistical properties of the regularized OT problem.
Their methods do not apply to Rényi divergences,
since the domain of the Fenchel conjugate of their regularizer $f$ has to fulfill
$\R^{d\times d}_{< 0} \subsetneq \dom{f^*}$~\cite[Def.~2.2]{KTM20},
which is not true for the Rényi divergence 
$f = \RR_{\alpha}(\cdot \mid \Vr \Vc^{\tT})$, 
since for this choice, we have equality.
Here $\Vr$ and $\Vc$ denote the discrete counterparts of $\mu$ and $\nu$, 
respectively.
Regularizers that are of Legendre-type are considered
by Dessein et al.~\cite{DPR18}. 
However, 
their definition of the restricted transport polytope~\cite[Eq.~(46), Prop.~5]{DPR18} 
differs from ours and their assumptions (A5) and (B4) are not fulfilled.
Lastly, 
the authors of \cite{BSR18,PC20} considered 
general (strongly) convex regularizers for the discrete OT problem, 
but most investigations are carried out for regularizers 
that are separable with respect to
the columns of the matrix they are applied to,
a property not fulfilled by the Rényi divergence.

\begin{figure}[t]
    \resizebox{\linewidth}{!}{
    \begin{tabular}{@{\hspace{1.4cm}} c @{\hspace{2.3cm}} c @{\hspace{2.1cm}} c @{\hspace{1.9cm}} c @{\quad} c}
        $\OT_{1}$ & $\OT_{10^{-1}}$ & $\OT_{10^{-2}}$ & $\OT_{10^{-3}}$ & $\OT$ \\
        \multicolumn{5}{c}{\includegraphics[width=1.2\linewidth, clip=true, trim=0pt 0pt 0pt 0pt]{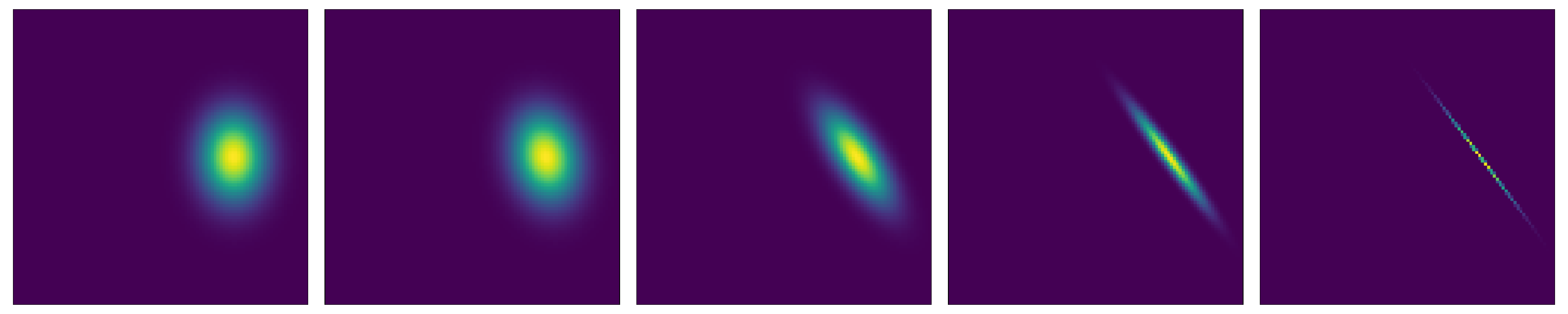}}
    \end{tabular}}
    \caption{
    Comparison of the unregularized plan (right) with the KL regularized plans for different regularization parameters, 
    between two Gaussians,
    for a (scaled) squared Euclidean distance matrix~\cite{FLA2021}.
    The two plans on the right had to be computed using improved versions of the Sinkhorn algorithm~\cite{CPSV17,S19}, 
    because the vanilla version diverged.
    }
    \label{fig:KL_reg_comp}
\end{figure}

In this paper,
we regularize the OT problem between $\mu$ and $\nu$ 
with the $\alpha$-Rényi divergence $\RR_{\alpha}(\cdot \mid \mu \otimes \nu)$
to the trivial coupling $\mu \otimes \nu$, for $\alpha \in (0, 1)$.
This generalizes Cuturi's KL regularized OT
because $\RR_{\alpha}$ converges pointwise to the KL divergence for $\alpha \nearrow 1$.
The $q$-Tsallis entropy, used in \cite{MNPN17}, 
and $q$-Tsallis divergence can be interpreted as a first-order approximation 
of the $\alpha$-Rényi entropy and divergence of the same order $\alpha = q$, respectively.
Hence, the Rényi divergence includes more information 
to distinguish two probability distributions through its higher order terms, 
such that small differences are better detected by the Rényi divergence.
Generally, 
for $\alpha \in (0, 1)$, 
the $\alpha$-Rényi divergence is always lower bounded 
by the $\alpha$-Tsallis divergence.

Rényi regularization is advantageous for applications 
since for $\alpha \searrow 0$, the regularized OT plans converge weakly to the unregularized ones. 
For the KL regularized OT problem,
convergence to the unregularized problem was achieved 
by instead taking the regularization parameter $\eps$ to zero, 
which makes the Sinkhorn algorithm instable.
Thanks to this new asymptotic behavior with respect to $\alpha$,
we derive a workaround for these instabilities.
We investigate the asymptotic regimes $\alpha \to \{ 0, 1 \}$ 
and $\eps \to \{ 0, \infty \}$
and prove convergence to the unregularized OT problem for $\alpha \searrow 0$ and $\eps \searrow 0$, respectively,
and to the KL regularized OT for $\alpha \nearrow 1$
for the objective value as well as (weakly) for the minimizer.
In contrast to KL regularized OT, 
the dual constraint is not penalized softly, 
but needs to be explicitly enforced,
so that---unlike for KL---in the case of Rényi regularization
the dual problem is not more amenable to optimization than the primal problem.

Numerically we show that for discretized measures 
the Rényi regularized transport plans structurally resemble the unregularized ones 
much better than their KL regularized and Tsallis regularized counterparts, 
even for moderate choices of regularization parameter $\eps$.
In the real-world application of predicting voter migration, 
the Rényi plans recover the ground truth better 
than other regularized plans and even than the unregularized plan, 
probably because $\alpha$-Rényi divergences put more emphasis on smaller probabilities depending on the order $\alpha$.

\textbf{Outline of the Paper}\quad
In \Cref{sec:RenyiOT}, we define the Rényi regularized OT distance 
in \Cref{def:Rényi_dist},
where the feasible region of the OT problem is restricted 
to all transport plans, whose Rényi divergence to the trivial coupling 
is bounded above by some threshold.
The Rényi regularized OT formulation we use for experiments 
is the Lagrangian reformulation \eqref{eq:dualRényiDivSinkhornDist} of that distance.
We prove that both distances are metrics
on the set of probability measures,
when we set them to zero for equal inputs.
In \Cref{sec:DualFormulationRenyiOT}, 
we derive the dual formulation 
for compact ground spaces $X$,
by calculating the convex preconjugate of the Rényi divergence 
with respect to the first component in \Cref{thm:dual_formulation},
a result that may be of independent interest.
The dual formulation enables
the explicit representation of the regularized OT plan using the dual potentials.
We show that, 
similarly to the KL regularized and Tsallis regularized OT, 
those optimal dual potentials are unique almost everywhere up to additive constants.
The main result of that section is
\Cref{theorem:ConvToOTKL}, which contains
the interpolation properties of the Rényi regularized OT problem 
in $\alpha \to \{0,1\}$ and $\eps \to \{0,+\infty\}$ separately. 

In \Cref{sec:algo} we prove the convergence of a nested mirror descent scheme 
to solve the Rényi regularized OT problem ($\RR_\alpha$ROT), 
where the outer mirror descent uses a modified Polyak step size 
and converges due to the objective being locally Lipschitz continuous.
This mirror descent reduces to performing Sinkhorn updates in every iteration
since we use the negentropy as the divergence generator.

\Cref{sec:Num} contains numerical experiments.
We show that regularizing with the Rényi divergence results in tighter transport plans,
that is, the Rényi regularized OT plans are closer 
to the unregularized plans with respect to the mean squared error for much smaller values 
of the regularization parameter $\eps$, similarly to the Tsallis regularized OT plans, 
but our plans reconstruct the ground truth better than the latter in inference tasks 
on real data sets.

Conclusions are drawn in \Cref{sec:Conclusions}.
The appendix contains an auxiliary proof on log-convex functions, 
the generalization of the preconjugate computation to non-compact spaces, 
and a subgradient method to solve the dual problem (D-$\RR_\alpha$ROT).

\textbf{Notation} \quad
By $\R_+ \coloneqq [0, \infty)$ we denote the nonnegative and by $\overline{\R} \coloneqq \R \cup \{ \pm \infty \}$ 
the extended real numbers.
Throughout this paper we will use following conventions: $0 \cdot \infty \coloneqq 0$, $\frac{1}{0} \coloneqq \infty$, $\frac{1}{\infty} \coloneqq 0$, as well as $0 \ln(0) \coloneqq 0$ and $\ln(0) \coloneqq -\infty$.
For a normed space $X$ with dual space $X^*$, 
we denote by $\Gamma_0(X)$ the set 
of convex and lower semicontinuous functions 
$f \colon X \to (- \infty, \infty]$, excluding the function $f \equiv \infty$.
The (effective) domain 
of $f$ is $\dom(f) \coloneqq \{ x \in X: f(x) < \infty \}$.
For $x \in \dom(f)$, 
the subdifferential of $f$ at $x$ is 
$\partial f(x) \coloneqq \{ x^* \in X^*: f(y) - f(x) \ge \langle x^*, y - x \rangle \ \forall y \in X \}$.

Let $X$ be a Polish space and $\B$ its Borel $\sigma$-algebra.
By $\M(X)$, we denote the space of finite Borel measures on $(X, \B)$, 
equipped with the total variation norm,
by $\M_+(X)$ is subset of non-negative measures, and
and by $\P(X)$ its subset of probability measures.
Further, 
$C(X)$ denotes the space of continuous functions on $X$, equipped with the supremum norm.
The support of $\mu \in \P(X)$ is defined as the
set of all points in $X$ whose every open neighborhood has positive measure
and denoted by $\supp(\mu)$.

\section{Rényi regularized optimal transport} \label{sec:RenyiOT}

In this section we introduce two ways to regularize the optimal transport (OT) problem 
using the $\alpha$-Rényi divergence \cite{GCE2024,GCE2024b} for $\alpha \in (0,1)$.
First,
we regularize by restricting the minimization problem to all transport plans
whose Rényi divergence to the trivial coupling is upper bounded by a regularization parameter $\gamma$.
Second,
we regularize the objective function by adding the Rényi divergence as a penalty term. 
This second regularized problem will be denoted by $\OT_{\eps, \alpha}$
and is the main object of study in this paper.
We derive the convex preconjugate of the Rényi divergence,
enabling the derivation of the dual formulation 
of $\OT_{\eps, \alpha}$ for compact ground spaces $X$.
The last subsection is concluded with a theorem
stating the convergence properties of $\OT_{\eps, \alpha}$
for $\eps \to \{ 0, \infty \}$ and $\alpha \to \{ 0, 1 \}$, respectively.

We begin with some definitions related to transportation problems.
The \emph{transport polytope} associated to marginals $\mu, \nu \in \P(X)$ 
is given by 
\begin{equation} \label{eq:transportPolytope}
    \Pi(\mu, \nu) 
    \coloneqq 
    \left\{ \pi \in \P(X \times X) \mid \pi(A \times X) = \mu(A), \, \pi(X \times B) = \nu(B) \;\forall A,B \in \B \right\} .
\end{equation}
Using the projections $P^{(j)} \colon X \times X \to X$, $(x_1, x_2) \mapsto x_j$, $j \in \{ 1, 2 \}$ 
onto the first and second component, 
respectively, 
the marginal constraints can be rewritten as $(P^{(1)})_{\#} \pi = \mu$ 
and $(P^{(2)})_{\#} \pi = \nu$.
Furthermore, for $p \in [1, \infty)$ the set of cost functions is defined as
\begin{equation} \label{eq:distanceFunctions}
    \D_p
    \coloneqq
    \left\{ c = d^p \mid d \colon X \times X \to [0, \infty), d \text{ is weakly lower semicontinuous and a metric} \right\}.
\end{equation}

Let $p \in [1, \infty)$ and $c \in \D_p$.
On the (convex) set of probability measures with finite $p$-th moment,
\begin{equation*}
    \P_p(X) 
    \coloneqq 
    \left\{ \mu \in \P(X) : \int_{X} c(x, x_0) \diff{\mu}(x) < \infty \ \text{for some } x_0 \in X \right\},
\end{equation*}
the (unregularized) optimal transport problem (or: Wasserstein-$p$ distance) is defined as
\begin{equation} \label{eq:OT}
    \OT \colon \P_p(X)\times\P_p(X) 
    \mapsto [0,\infty),
    \qquad
    (\mu,\nu)
    \mapsto
    \min_{\pi \in \Pi(\mu,\nu)}\;
    \langle c,\pi\rangle,
\end{equation}
where $\langle c, \pi \rangle \coloneqq \int_{X \times X} c(x, y) \diff{\pi}(x, y)$. 
Note that if $X$ is compact, 
then $\P_p(X) = \P(X)$.

The \emph{Rényi divergence} for discrete distributions was introduced 
by Alfred Rényi in 1961~\cite{R1961}.
The continuous counterpart for probability measures is defined below.

\begin{definition}[$\alpha$-Rényi divergence {\cite[Def.~2]{EH14}}]
\label{def:RenyiDivDef}
    The $\alpha$-Rényi divergence of \textit{order} $\alpha \in (0, 1)$ is defined as 
    \begin{equation} \label{eq:Renyi_div}
        \RR_\alpha: \P(X) \times \P(X) \to [0, \infty], 
        \quad
        (\mu \mid \nu) 
        \mapsto \frac{1}{\alpha - 1}\ln\left(\int_{X} \left(\frac{\rho_\mu(x)}{\rho_\nu(x)}\right)^\alpha \diff {\nu}(x) \right),
    \end{equation}
    where $\rho_{\mu}$ and $\rho_{\nu}$ are the densities of $\mu$ and $\nu$ 
    with respect to any non-negative, 
    $\sigma$-finite reference measure $\tau$ chosen 
    such that $\mu$ and $\nu$ are absolutely continuous
    with respect to $\tau$, respectively, 
    e.g. $\tau = \frac{1}{2} \mu + \frac{1}{2} \nu$.
    This representation is independent of the choice of $\tau$.

    We define the $0$- and $1$-Rényi divergence via limits~\cite[Thms.~4,~5]{EH14}, 
    that is, for $\mu, \nu \in \P(X)$ we let
    \begin{align*}
        \RR_1(\mu\mid\nu) 
        & \coloneqq \lim_{\alpha \nearrow 1} \RR_\alpha(\mu\mid\nu) 
        = \KL(\mu\mid\nu)
        \coloneqq \int_X \ln\left(\frac{\rho_\mu(x)}{\rho_\nu(x)}\right)\diff \mu(x). \\
        \RR_0(\mu\mid\nu) 
        & \coloneqq \lim_{\alpha \searrow 0} \RR_\alpha(\mu\mid\nu)
        = -\ln\bigg(\nu\big(\supp(\mu)\big)\bigg).
    \end{align*}
\end{definition}

\begin{figure}[t]
    \centering
    \includegraphics[width=0.42\linewidth, clip=true, trim=0pt 0pt 670pt 0pt]{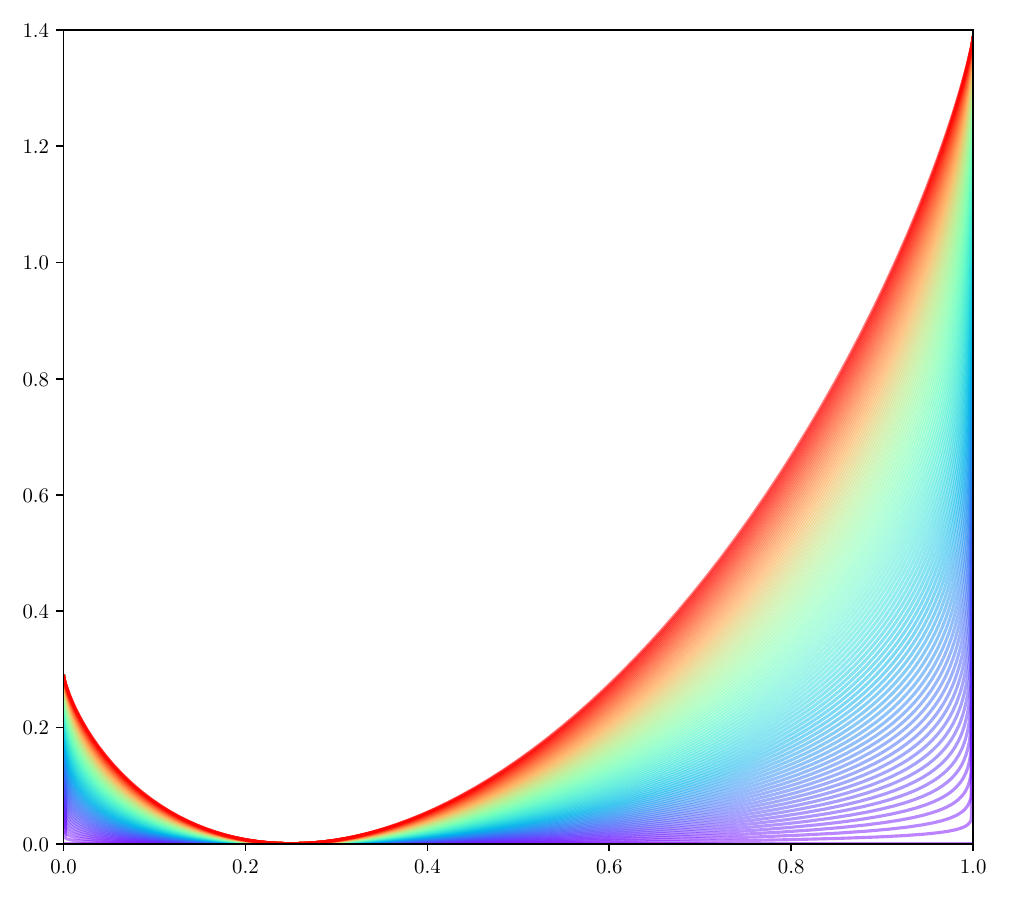}
    \includegraphics[width=0.48\linewidth, clip=true, trim=550pt 0pt 50pt 0pt]{pics/renyi_tsallis_p_q.pdf}
    \caption{
    Plot of the discrete Rényi divergence (left), defined in \eqref{eq:discreteRenyi}, 
    and the discrete Tsallis divergence (right), 
    both of $(p,1-p)$ to $(\frac{1}{4},\frac{3}{4})$ 
    over $p \in [0, 1]$ for different values of $\alpha \in [0, 1]$ as indicated by the color bar.
    Note that in the left plot, the green band extends much higher and the red band is much thinner than in the right plot.}
    \label{fig:Renyi-Comparison}
\end{figure}

Another way of measuring the discrepancy of two probability measures 
is given by $f$-divergences.
A $f$-divergence~\cite{LMS2017} between $\mu,\nu \in \P(X)$,
where $f \colon [0, \infty) \to [0, \infty]$ is a so-called entropy function,
is defined by
\begin{align*}
    D_f(\mu \mid \nu) \coloneqq \int_X f\left(\rho(x)\right) \diff \nu(x)
    + f_{\infty}' \cdot \mu_s(X),
\end{align*}
where $\mu = \rho \nu + \mu_s$ is the Lebesgue decomposition of $\mu$ 
with respect to $\nu$ and $f_{\infty}' \coloneqq \lim_{x \to \infty} \frac{f(x)}{x}$ 
is the \textit{recession constant} of $f$, 
with the convention $\infty \cdot 0 \coloneqq 0$.
Hence, 
the Rényi divergence is not a $f$-divergence due to the log-term.

The $q$-Tsallis divergences for $q \in (0, 1)$
are a family of $f$-divergences,
using the entropy function 
$f_q(t) \coloneqq \frac{1}{q-1}(t^q - qt + q - 1)$ for $t \ge 0$, namely,
\begin{align*}
    D_{f_q}(\mu \mid\nu)
    = \frac{1}{q-1}\left[
    \int_X \rho(x)^q\diff\nu(x) -1\right] + \frac{1}{1-q} \mu_s(X).
\end{align*}
The $q$-Tsallis entropy of a measure $\mu$ 
with density $\rho_\mu$ with respect to $\tau$
is defined by 
\begin{equation*}
    \TT_q(\mu) \coloneqq \frac{1}{q - 1}\left[1 - \int_X \rho_\mu(x)^q \diff \tau(x)\right].
\end{equation*}
Note that
the Tsallis divergence can be seen as the first-order approximation 
of the Rényi divergence of the same order and we have
\begin{equation*}
    R_{\alpha} = m_{\alpha} \circ D_{f_{\alpha}}, 
    \qquad 
    \text{where}
    \qquad
    m_{\alpha} \colon \left[0, \tfrac{1}{1 - \alpha}\right) \to [0, \infty), 
    \; x \mapsto \tfrac{1}{\alpha - 1} \ln\big( (\alpha - 1) x + 1\big).
\end{equation*}

Before we list important properties of the Rényi divergences (for a complete list we refer to~\cite{EH14})
used throughout the paper, we first define Markov kernels.

\begin{definition}[Markov kernel] \label{definition:MarkovKernel}
    Let $(X, \B)$ and $(T, \T)$ be measurable spaces.
    A Markov (transition) kernel between them 
    is a map $\kappa \colon X \times \T \to [0, 1]$ 
    such that $k(\cdot, B)$ is $\B$-$\B([0, 1])$ 
    measurable for all $B \in \T$ and $k(x, \cdot) \in \P(T, \T)$ for all $x \in X$.
    The signal $\mu \in \P(X)$ pushed through $\kappa$
    is defined as $(\kappa \circ \mu)(B) \coloneqq \int_{X} \kappa(x, B) \diff{\mu}(x)$ for $B \in \T$.    
\end{definition}

\begin{example}[Deterministic Markov kernel]
    Every measurable map $\phi \colon (X, \T) \to (X, \T)$ induces a deterministic Markov kernel $\kappa_{\phi} \colon X \times \T \to \{ 0, 1 \}$, $(x, B) \mapsto \1_{\phi^{-1}(B)}(x)$ with the property that $\kappa_{\phi} \circ \mu = \phi_{\#} \mu$.
\end{example}

\begin{proposition}[Properties of the Rényi divergence] \label{prop:RenyiDivProps}\leavevmode
    \begin{enumerate}
        \item 
        \textbf{Divergence property \cite[Thm.~8]{EH14}.}
        We have $\RR_{\alpha}(\mu \mid \nu)\ge 0$ for all $\alpha \in [0, 1]$ and for $\alpha > 0$ and $\mu, \nu \in \P(X)$ we have $\RR_{\alpha}(\mu \mid \nu) = 0$ if and only if $\mu = \nu$.

        \item 
        \textbf{Varying the order $\alpha$~\cite[Thms.~3,~7]{EH14}.}
        $\RR_{\alpha}$ is pointwise nondecreasing and continuous in $\alpha \in [0, 1]$.

        \item 
        \textbf{Strict convexity~\cite[Thm.~11]{EH14}.}
        $\RR_{\alpha}$ is jointly convex for $\alpha \in [0, 1]$.
        For $\mu_1,\mu_2, \nu \in \P(X)$ and $\eta \in (0, 1)$, it holds that
        \begin{equation*}
            \RR_\alpha(\eta\mu_1 + (1 - \eta)\mu_2\mid \nu)
            = \eta \RR_\alpha(\mu_1\mid \nu) + (1 - \eta)\RR_\alpha(\mu_2 \mid \nu),
        \end{equation*}
        if and only if $\mu_1 = \mu_2$ $\nu$-a.s.
        
        \item 
        \textbf{Continuity ~\cite[Thms.~15,~17]{EH14}.}
        $\RR_{\alpha}$ is weakly jointly lower semicontinuous for every $\alpha \in (0, 1]$ 
        and strongly jointly uniformly continuous for every $\alpha \in (0, 1)$.

        \item 
        \textbf{Data processing inequality~\cite[Eq.~(57/59)]{PS10}.}
        For any $P, Q \in \P(X)$, $\alpha \in (0,1]$, and Markov kernel $\kappa$, it 
        holds that $\RR_\alpha(\kappa \circ P \mid \kappa \circ Q) \leq \RR_\alpha(P \mid Q)$.
    \end{enumerate}
\end{proposition}

\subsection{The Rényi Optimal Transport Premetric}

If we replace the KL divergence with the Rényi divergence in Cuturi's Sinkhorn distance~\cite[Def.~1]{C13},
then we obtain what we call the Rényi-OT premetric,
which relies on the following definition.

\begin{definition}[$\alpha$-Rényi ball of level $\gamma$ centered around $\mu\otimes\nu$] \label{definition:RenyiBall}
    For $\mu, \nu \in \P(X)$, $\gamma \in [0, \infty]$ and $\alpha \in (0, 1)$,
    the $\alpha$-Rényi ball of level $\gamma$ centered around $\mu\otimes\nu$ is
    \begin{equation}
        \Pi_\gamma^\alpha(\mu, \nu) 
        \coloneqq 
        \left\{\pi \in \Pi(\mu, \nu) : \RR_\alpha(\pi \mid \mu\otimes\nu) \le \gamma \right\}.
    \end{equation}
\end{definition}

By the divergence property from \Cref{prop:RenyiDivProps} 
we have $\Pi_0^{\alpha}(\mu, \nu) = \{ \mu \otimes \nu \}$.

\begin{lemma} \label{lemma:restricted_polytope_compact}
    The set $\Pi_{\gamma}^{\alpha}(\mu, \nu)$ is convex and weakly compact.
\end{lemma}

\begin{proof}
    See Appendix~\ref{sec:app_add_proofs}.
\end{proof}

\begin{definition}[Rényi-OT premetric]
\label{def:Rényi_dist}
    For an order $\alpha \in (0, 1)$, a regularization parameter $\gamma \in [0, \infty]$ and a cost $c \in \D_p$, the $(c, \gamma, \alpha)$-Rényi-OT premetric is defined as 
    \begin{equation}
    \begin{aligned} 
    \label{eq:Renyi-Sinkhorn_metric}
        d_{c,\gamma, \alpha} \colon \P_p(X)\times\P_p(X) & \to [0, \infty), \\
        (\mu, \nu) & \mapsto \min_{\pi \in \Pi_\gamma^\alpha(\mu, \nu)} \left(\langle c, \pi \rangle\right)^{\frac{1}{p}}
        = 
        \left(\min_{\pi \in \Pi_\gamma^\alpha(\mu, \nu)} \langle c, \pi \rangle\right)^{\frac{1}{p}}.
    \end{aligned}
    \end{equation}
\end{definition}

We have $d_{c, 0, \alpha}(\mu,\nu)^p = \langle c, \mu \otimes \nu \rangle$.

The next proposition states the uniqueness of the minimizer 
and relates the value of $d_{c,\gamma, \alpha}$ 
to the unregularized problem in dependence on the threshold $\gamma > 0$.

\begin{proposition}
\label{remark:ExandUniqueOfOT}
    Let $\mu, \nu \in \P_p(X)$, $\gamma \in [0, \infty)$,
    $c \in \D_p$,
    and $\alpha \in (0, 1)$.
    Let 
    \begin{equation*}
        \overline\gamma 
        \coloneqq 
        \inf_{\substack{\pi \in \Pi(\mu, \nu), \\ \OT(\mu,\nu) = \langle c, \pi\rangle}} 
        \RR_\alpha(\pi \mid \mu\otimes \nu) < \infty.
    \end{equation*}
    If $\gamma < \overline\gamma$, then
    the problem \eqref{eq:Renyi-Sinkhorn_metric} has a unique solution.
    If $\gamma \ge \overline\gamma$, then
    $d_{c, \gamma, \alpha}(\mu, \nu) = \OT(\mu, \nu)$.
\end{proposition}

\begin{proof}
    The function in \eqref{eq:Renyi-Sinkhorn_metric}
    is proper, since $\langle c, \mu\otimes\nu\rangle < \infty$,
    see~\cite[Def.~6.4]{V08}.
    By \cite[Lemma~4.3]{V08}, the objective function $\langle c, \cdot \rangle$ 
    is weakly lower semicontinuous on $\P(X \times X)$, and by \Cref{lemma:restricted_polytope_compact},
    $\Pi_\gamma^\alpha(\mu,\nu)$ is weakly compact,
    so there exists a minimizer in \eqref{eq:Renyi-Sinkhorn_metric}.

    Let $\pi_c \in \Pi(\mu,\nu)$ be an unregularized OT plan
    solving \eqref{eq:OT}. We distinguish two cases.
    
    If $\gamma \ge \RR_\alpha(\pi_c \mid \mu\otimes\nu)$
    (this holds for example if $\pi_c = \mu\otimes\nu$),
    then we have $\pi_c \in \Pi_{\gamma}^{\alpha}(\mu, \nu)$ and thus
    $d_{c,\gamma,\alpha}(\mu,\nu)^p = \langle c,\pi_c \rangle$
    and the minimizer is not unique in general.

    If instead $\gamma < \overline{\gamma}$,
    assume towards contradiction that $\pi_1,\pi_2 \in \Pi_\gamma^\alpha(\mu,\nu)$ 
    are distinct minimizers of \eqref{eq:Renyi-Sinkhorn_metric}.
    Since $\pi_1,\pi_2 \in \Pi_\gamma^\alpha(\mu,\nu)$ and $\pi_c \not\in\Pi_\gamma^\alpha(\mu,\nu)$, we have
    $\langle c,\pi_c\rangle < \langle c,\pi_1\rangle = \langle c,\pi_2\rangle$.
    Then for any $t \in (0,1)$ we have by the linearity of the function $\langle c, \,\cdot\, \rangle$ 
    that $\pi_t \coloneqq t\pi_1 + (1 - t)\pi_2 \in \Pi(\mu,\nu)$ has the same functional value: $\langle c, \pi_t \rangle = \langle c, \pi_1 \rangle$.
    By the strict convexity of the Rényi divergence 
    we have 
    \begin{equation*}
        \RR_\alpha(\pi_t \mid \mu\otimes\nu) < t\RR_\alpha(\pi_1\mid \mu\otimes\nu) 
        + (1 - t) \RR_\alpha(\pi_2 \mid \mu\otimes\nu) \le \gamma \qquad \forall t \in (0, 1).
    \end{equation*}
    Hence, 
    $\pi_t \in \Pi_\gamma^\alpha(\mu,\nu)$ and thus $\pi_t \ne \pi_c$.
    Furthermore,
    we have $\pi_{t, \kappa} \coloneqq \kappa\pi_c + (1-\kappa)\pi_t \in \Pi(\mu,\nu)$
    for all $\kappa \in (0,1)$ by the convexity of the transport polytope 
   and thus 
    \begin{equation} \label{eq:pi_t_kappa}
        \langle c,\pi_{t, \kappa}\rangle 
        = \kappa \langle c,\pi_c\rangle + (1-\kappa)\langle c, \pi_t\rangle 
        < \langle c,\pi_1\rangle = \langle c,\pi_2\rangle
    \end{equation}
    and
    \begin{align*}
        \RR_\alpha(\pi_{t, \kappa}\mid\mu\otimes\nu) 
        & < \kappa \RR_\alpha(\pi_c\mid\mu\otimes\nu) + (1-\kappa)\RR_\alpha(\pi_t\mid\mu\otimes\nu) \\
        & = \kappa \RR_\alpha(\pi_c\mid\mu\otimes\nu) + (1-\kappa)(\gamma - \eps),
    \end{align*}
    where $\eps \coloneqq \gamma - \RR_{\alpha}(\pi_t \mid \mu \otimes \nu) > 0$.
    Hence $\RR_\alpha(\pi_{t, \kappa}\mid\mu\otimes\nu) \le \gamma$ holds 
    if $\kappa(\RR_\alpha(\pi_c\mid\mu\otimes\nu) - (\gamma - \eps))  \le \eps$.
    We now choose any $\kappa$ with $0 < \kappa  \le \frac{\eps}{\RR_\alpha(\pi_c\mid\mu\otimes\nu) - \gamma + \eps} < 1$, which is well-defined since $\RR_\alpha(\pi_c\mid\mu\otimes\nu) - \gamma > 0$.
    For this choice of $\kappa$, 
    we have $\pi_{t, \kappa} \in \Pi_\gamma^\alpha(\mu \otimes \nu)$,
    contradicting that $\pi_1$ and $\pi_2$ are minimizers by \eqref{eq:pi_t_kappa}.
    We conclude that the minimizer is unique and fulfills the ball constraints with equality.
\end{proof}

We can now prove a generalization of \cite[Thm.~1]{C13}.

\begin{theorem}[$\Mat 1_{\mu \ne \nu} d_{c, \gamma, \alpha}$ is a metric] \label{theorem:RenyiDistance}
    For $\alpha \in (0,1)$, $\gamma \in [0, \infty)$ and $c \in \D_p$ we have that
    \begin{equation*}
        \tilde{d}_{c, \gamma, \alpha} \colon \P_p(X) \times \P_p(X) \to [0, \infty), \qquad
        (\mu, \nu) \mapsto
        \Mat 1_{\mu \ne \nu}(\mu, \nu)
        d_{c, \gamma, \alpha}(\mu, \nu)
    \end{equation*}
    is a metric on $\P_p(X) \times \P_p(X)$.
\end{theorem}

The proof is essentially the same 
as in the case for the KL-regularizer,
the only small difference is that the data processing inequality 
also holds for the Rényi divergence.

\begin{proof}
    See Appendix~\ref{sec:app_add_proofs}.
\end{proof}

\subsection{Rényi regularized Optimal Transport}

In Cuturi's paper, in \cite[Eq.~(2)]{C13}, the KL divergence constraint is instead handled 
by penalization and the feasible set of the regularized OT problem 
is expanded to the whole $\Pi(\mu, \nu)$. 
We can also handle the Rényi divergence constraint 
in \eqref{eq:Renyi-Sinkhorn_metric} by penalization.
For the Tsallis entropy instead of the Shannon entropy, 
an analogous construction is considered in~\cite[Def.~1]{MNPN17}.
However, for the continuous setting, in~\cite[Rem.~3]{NS21} it is explained
why regularizing the OT problem 
with the KL divergence instead of the Shannon entropy is more sensible.
In the discrete case, both regularizers yield the same minimization problem.

\begin{definition}[Dual Rényi-OT distance] \label{definition:dualRSdist}
    For $\alpha \in (0, 1)$, $c \in \D_p$
    and a regularization parameter $\eps \in [0, +\infty)$ 
    the \emph{dual Rényi-OT distance} is given by
    \begin{equation} \label{eq:dualRényiDivSinkhornDist}
        d_{c}^{\alpha, \eps} 
        \colon \P_p(X) \times \P_p(X) \to [0, \infty), 
        \qquad
        (\mu, \nu) 
        \mapsto 
        \left(\langle c, \pi^{\alpha, \eps}_{c}(\mu, \nu) \rangle\right)^{\frac{1}{p}}, 
    \end{equation}
    where
    \begin{equation} \label{eq:dualSinkhornProblem}
        \pi^{\alpha, \eps}_{c}(\mu, \nu) 
        \in
        \argmin_{\pi \in \Pi(\mu, \nu)}
        \quad 
        \langle c, \pi \rangle
        + \eps \RR_{\alpha}(\pi \mid \mu\otimes\nu).
    \end{equation}
\end{definition}

\begin{theorem}[Rényi divergence regularized OT plan] \label{theorem:PiDivReg}
    Let $\mu, \nu \in \P_p(X)$, $\eps \in (0, \infty)$ and $\alpha \in (0, 1)$.
    The problem \eqref{eq:dualSinkhornProblem} has a unique solution $\pi_{c}^{\alpha, \eps}(\mu,\nu)$.
\end{theorem}

\begin{proof}
    The existence of a minimizer 
    follows from the weak lower semicontinuity 
    of $\RR_{\alpha}(\;\cdot \mid \mu \otimes \nu)$ from \Cref{prop:RenyiDivProps} 
    and the weak compactness of $\Pi(\mu, \nu)$, see \cite[p.~45]{V08}.
    The uniqueness follows from the strict convexity 
    of $\RR_\alpha(\;\cdot\mid \mu\otimes\nu)$ on $\Pi(\mu,\nu)$ 
    from \Cref{prop:RenyiDivProps}.
\end{proof}

As in the case of KL regularization~\cite[Eq.~(2)]{C13}, 
we can show a certain type of equivalence between the two different Rényi regularized OT problems 
from \Cref{def:Rényi_dist} and \Cref{definition:dualRSdist}.

\begin{remark}[Lagrangian reformulation] \label{remark:LagrangeReformulation}
    The Lagrangian for the optimization problem $d_{c,\gamma,\alpha}(\mu,\nu)^p$,
    defined in \eqref{eq:Renyi-Sinkhorn_metric}, is defined 
    for $\pi \in \Pi(\mu,\nu)$ and $\tau \ge 0$ by
    \begin{align*}
        \FL(\pi, \tau) 
        \coloneqq
        \langle c, \pi\rangle + \tau[\RR_\alpha(\pi \mid \mu\otimes\nu) - \gamma].
    \end{align*}
    The Rényi-OT premetric $d_{c, \gamma, \alpha}$
    is equivalent 
    to the dual Rényi-OT distance \eqref{eq:dualRényiDivSinkhornDist} in the following sense:
    by duality theory, 
    we obtain for $\hat\pi \coloneqq \argmin_{\pi \in \Pi_{\gamma}^{\alpha}(\mu, \nu)} \langle c, \pi \rangle$ that
    \begin{align*}
        d_{c,\gamma,\alpha}(\mu,\nu)^p
        = \langle c,\hat \pi \rangle 
        = \min_{\pi \in \Pi_\gamma^\alpha(\mu,\nu)}
        \langle c, \pi\rangle
        = \min_{\pi \in \Pi(\mu,\nu)}
        \max_{\tau\ge 0}
        \FL(\pi,\tau)
        = \max_{\tau\ge 0}
        \min_{\pi \in \Pi(\mu,\nu)}
        \FL(\pi,\tau).
    \end{align*}
    Moreover,
    for fixed $\mu, \nu \in \P_p(X)$, 
    $c \in \D_p$, $\gamma \in (0, \infty]$,
    and $\alpha \in (0, 1)$,
    there exists a Lagrange multiplier $\tau^*\ge 0$ 
    and $\pi^* \in \Pi(\mu,\nu)$ 
    such that $(\pi^*, \tau^*)$ is a saddle point of $\FL$, that is,
    \begin{align*}
        d_{c,\gamma,\alpha}(\mu,\nu)^p
        = \langle c,\hat \pi \rangle
        = \FL(\pi^*, \tau^*)
        & = \begin{cases}
            \langle c, \pi^*\rangle, & \text{if}\; \tau^* = 0, \\
            \langle c, \pi^*\rangle, & \text{if}\; \RR_\alpha(\pi^* \mid\mu\otimes\nu) = \gamma,
        \end{cases} \\
        & = \begin{cases}
            d_{c}^{\alpha, 0}(\mu, \nu), & \text{if}\; \tau^* = 0, \\
            d_c^{\alpha, \tau^*}(\mu, \nu), & \text{if}\; \RR_\alpha(\pi^* \mid\mu\otimes\nu) = \gamma,
        \end{cases}
    \end{align*}
    using the KKT conditions 
    $\tau^* \big(\RR_\alpha(\pi^* \mid\mu\otimes\nu) - \gamma\big) = 0$.
    Notably,
    whenever $\RR_\alpha(\pi_c\mid\mu\otimes\nu) < \gamma$,
    where $\pi_c \in \Pi(\mu,\nu)$ is the optimal unregularized transport plan,
    we are in the case in which $\tau^* = 0$ is necessary, 
    which translates into the case with $\eps = 0$.
    
    In other words, for any pair $(\mu, \nu) \in \P_p(X) \times \P_p(X)$ 
    and $\gamma \ge 0$,
    there exists $\eps \in [0, + \infty)$, 
    such that $\langle c, \pi^{\alpha, \eps}_{c}(\mu, \nu) \rangle = d_{c, \gamma, \alpha}(\mu, \nu)^p$.
\end{remark}

We now define the main object of this paper, using the objective from \eqref{eq:dualSinkhornProblem}.

\begin{definition}[Rényi regularized OT]
\label{def:RenyiRegOTProbelmPrimal}
    For $\alpha \in (0, 1]$, $\eps \in [0, \infty)$, and $c \in \D_p$ let
    \begin{equation} 
    \label{eq:OTeps,alpha}
        \OT_{\eps, \alpha} 
        \colon 
        \P_p(X) \times \P_p(X) \to [0, \infty), \qquad
        (\mu, \nu) 
        \mapsto 
        \min_{\pi \in \Pi(\mu, \nu)}
        \langle c, \pi \rangle + \eps\RR_\alpha(\pi \mid \mu\otimes\nu).
    \end{equation}
    The KL regularized OT problem is $\OT_{\eps} \coloneqq \OT_{\eps,1}$.
\end{definition}

\begin{theorem}[Rényi divergence regularized premetric] \label{theorem:DivergenceRegPreMetric}
    Let $\alpha \in (0,1)$ and $\eps \in [0, \infty)$.
    Then $\OT_{\eps, \alpha}$ is symmetric and fulfills the triangle inequality.
\end{theorem}

\begin{proof}
    The symmetry follows exactly like in the proof of \Cref{theorem:RenyiDistance}.

    To prove the triangle inequality, 
    take three marginals $\mu,\xi,\nu \in \P_p(X)$.
    Let $\pi_{12} \coloneqq \pi_c^{\alpha, \eps}(\mu, \xi) \in \Pi(\mu,\xi)$ and $\pi_{23} \coloneqq \pi_c^{\alpha, \eps}(\xi, \nu) \in \Pi(\xi,\nu)$ be the minimizers of $\OT_{\eps, \alpha}(\mu, \xi)$ and $\OT_{\eps, \alpha}(\xi, \nu)$, respectively.
    Using the same argument as in the proof of \Cref{theorem:RenyiDistance}
    we can glue $\pi_{12}$ and $\pi_{23}$ together to obtain ${\pi_{13}} \in \Pi(\mu,\nu)$.
    Again using the data processing inequality and that $\RR_\alpha(\pi_{12} \mid \mu\otimes\xi)\ge 0$, we obtain
    \begin{align*}
        \RR_\alpha({\pi_{13}} \mid \mu\otimes\nu) {\le R_{\alpha}(\pi_{12} \mid \mu \otimes \xi) 
        \le \RR_\alpha(\pi_{12} \mid \mu\otimes\xi) + \RR_\alpha(\pi_{23} \mid \xi\otimes\nu).}
    \end{align*}
    Taking the final inequality, $\langle c, {\pi_{13}} \rangle \le \langle c, \pi_{12} + \pi_{23} \rangle$, from the proof of \Cref{theorem:RenyiDistance} (with $p = 1$) into account yields 
    \begin{align*}
        \OT_{\eps, \alpha}(\mu,\nu)
        & \le \langle c, {\pi_{13}} \rangle + \eps\RR_\alpha({\pi_{13}} \mid \mu\otimes\nu) \\
        & \le \langle c, \pi_1 \rangle + \eps\RR_\alpha(\pi_{12} \mid \mu\otimes\xi)
        + \langle c, \pi_2 \rangle + \eps\RR_\alpha(\pi_{23} \mid \xi\otimes\nu) \\
        & = \OT_{\eps, \alpha}(\mu,\xi) + \OT_{\eps, \alpha}(\xi,\nu).\qedhere
    \end{align*}
\end{proof}

\begin{corollary} \label{lem:DivergenceRegMetric}
    For $\alpha \in (0,1)$ and $\eps \in [0, \infty)$,
    \begin{equation*}
        (\mu, \nu) 
        \mapsto \Mat 1_{\mu \ne \nu}(\mu, \nu)
        \OT_{\eps, \alpha}(\mu, \nu)
    \end{equation*}
    defines a metric on $\P_p(X) \times \P_p(X)$.
\end{corollary}

\begin{proof}
    This follows exactly as in the first part of the proof of \Cref{theorem:RenyiDistance}.
\end{proof}

We conclude this subsection with a lemma 
showing monotonicity properties 
of the regularized OT problem
$\OT_{\eps, \alpha}$ in $\alpha$ and $\eps$
for fixed marginals.

\begin{lemma}[Monotonicity of Rényi regularized OT]
\label{lem:RenyiMonotone}
    Let $\mu,\nu \in \P_p(X)$, $\alpha,\alpha' \in (0,1)$ and $\eps, \eps' > 0$ with $\alpha > \alpha'$
    and $\eps < \eps'$.
    Then, we have 
    \begin{equation*}
        \OT_{\eps',\alpha}(\mu,\nu)
       \ge \OT_{\eps, \alpha}(\mu,\nu)
       \ge \OT_{\eps, \alpha'}(\mu,\nu).
    \end{equation*}
\end{lemma}

\begin{proof}
    We show the inequalities separately.
    \begin{enumerate}
        \item 
        Thanks to \Cref{theorem:PiDivReg}, 
        there exist unique $\pi_c^{\alpha, \eps}, \pi_c^{\alpha',\eps} \in \Pi(\mu,\nu)$ 
        solving $\OT_{\eps, \alpha}(\mu,\nu)$ 
        and $\OT_{\eps, \alpha'}(\mu,\nu)$, respectively.
        Hence, the following inequality yields the assertion using 
        that $\RR_\alpha$ is non-decreasing in $\alpha \in (0,1)$.
        \begin{align*}
            \OT_{\eps, \alpha}(\mu,\nu) 
            & = \langle c, \pi_c^{\alpha, \eps}\rangle + \eps\RR_\alpha(\pi_c^{\alpha, \eps} \mid \mu\otimes\nu) \\
            &\ge \langle c, \pi_c^{\alpha, \eps}\rangle + \eps\RR_{\alpha'}(\pi_c^{\alpha, \eps} \mid \mu\otimes\nu) \\
            &\ge \langle c, \pi_c^{\alpha',\eps}\rangle + \eps\RR_{\alpha'}(\pi_c^{\alpha',\eps} \mid \mu\otimes\nu)
            = \OT_{\eps, \alpha'}(\mu,\nu). 
        \end{align*}

        \item 
        The second inequality follows by a similar technique:
        \begin{align*}
            \OT_{\eps',\alpha}(\mu,\nu)
            & = \langle c, \pi_c^{\alpha, \eps'}\rangle + \eps'\RR_\alpha(\pi_c^{\alpha, \eps'} \mid \mu\otimes\nu) \\
            &\ge \langle c, \pi_c^{\alpha, \eps'}\rangle + \eps\RR_\alpha(\pi_c^{\alpha, \eps'} \mid \mu\otimes\nu) 
           \ge \OT_{\eps, \alpha}(\mu,\nu).
        \end{align*}
    \end{enumerate}
\end{proof}

\subsection{Dual Formulation of Rényi Regularized OT} 
\label{sec:DualFormulationRenyiOT}

In the following, we assume that $X$ is compact, so the moment restriction on transport plans is vacuous.
We comment on the non-compact case in \Cref{rem:noncompact}.
We will work exclusively with $p = 1$.

For proving the predual formulation 
of the primal problem \eqref{eq:dualSinkhornProblem} in \cref{thm:dual_formulation},
we first derive a convex conjugate-style representation of the Rényi divergence 
with respect to its first component.
The proof of the next theorem first derives, in a similar fashion to~\cite[Thms.~2.1,~2.2]{BPDKB23}, a convex conjugate-type variational formulation.

\begin{theorem}[A preconjugate of the Rényi divergence on compact $X$] \label{theorem:RenyiPreconjugate}
    For $\mu, \theta \in \P(X)$ we have
    \begin{equation*}
        R_{\alpha}(\mu \mid \theta)
        = \sup_{\substack{h \in \C(X) \\ h < 0 \ \theta\text{-a.e.}}}
         \langle h, \mu \rangle
        - \ln\left(\langle (-h)^{\frac{\alpha}{\alpha - 1}}, \theta \rangle\right)
        - C_\alpha,
    \end{equation*}
    where $C_\alpha \coloneqq \frac{\alpha}{1 - \alpha}\bigl(\ln(\frac{\alpha}{1-\alpha}) - 1\bigr)$.
\end{theorem}

\begin{proof}
    Let $\mu \in \P(X)$.
    By~\cite[Eq.~(3)]{BPDKB23} we have 
    \begin{equation} \label{eq:renyi_var_form}
        \RR_\alpha(\mu \mid \theta) 
        = \sup_{g \in \Omega}
        \left\{\frac{\alpha}{\alpha - 1} \ln\left(\langle e^{(\alpha - 1)g}, \mu \rangle\right)
        - \ln\left(\langle e^{\alpha g}, \theta \rangle\right)\right\}
    \end{equation}
    for a set $\Omega$ with $\C(X) \subset \Omega \subset \F(X)$, 
    where $\F(X)$ denotes the set of measurable real-valued functions on $X$, 
    and we interpret $\infty - \infty = -\infty$
    and $-\infty + \infty = + \infty$.

    Consider
    \begin{equation*}
        \C(X) \subset \Omega
        \coloneqq \big\{(\alpha - 1)^{-1} \ln(-g) 
        : g \in \C(X), 
        \; g < 0 
        \; \theta\text{-a.e.}\big\}
        \subset \F(X).
    \end{equation*}

    Then \eqref{eq:renyi_var_form} becomes
    \begin{align}
        \RR_\alpha(\mu \mid \theta)
        & = \sup_{\substack{g \in \C(X) \\ g < 0 \ \theta\text{-a.e.}}}
        \left\{\frac{\alpha}{\alpha - 1} \ln\left(\langle e^{(\alpha - 1)(\alpha - 1)^{-1}\ln(-g)}, \mu \rangle\right)
        - \ln\left(\langle e^{\alpha(\alpha - 1)^{-1}\ln(-g)}, \theta \rangle\right)\right\} \notag \\
        & = \sup_{\substack{g \in \C(X) \\ g < 0 \ \theta\text{-a.e.}}}
        \left\{\frac{\alpha}{\alpha - 1} \ln\left(\langle -g, \mu \rangle\right)
        - \ln\left(\langle (-g)^{\frac{\alpha}{\alpha - 1}}, \theta \rangle\right)\right\}. 
        \label{eq:PreVarForm}
    \end{align}
    We observe that 
    \begin{align} \label{eq:lnInfProblem}
        \ln(c) = \inf_{z \in \R} 
        \{z - 1+ce^{-z}\} \quad \text{for any } c > 0.
    \end{align}
    Plugging \eqref{eq:lnInfProblem} into the first term of \eqref{eq:PreVarForm} yields 
     \begin{align*}
        \RR_\alpha(\mu \mid \theta)
        & = \sup_{\substack{g \in \C(X) \\ g < 0 \ \theta\text{-a.e.}}}
        \left\{\frac{\alpha}{\alpha - 1} \inf_{z \in \R} \left\{z - 1 + e^{-z}\langle -g, \mu \rangle \right\}
        - \ln\left(\langle (-g)^{\frac{\alpha}{\alpha - 1}}, \theta \rangle\right)\right\} \\
        & = \sup_{\substack{g \in \C(X) \\ g < 0 \ \theta\text{-a.e.}}}
        \sup_{z \in \R}
        \left\{ \frac{\alpha(1 - z)}{1 - \alpha} + \frac{\alpha}{1 - \alpha}e^{-z} \langle g, \mu \rangle
        - \ln\left(\langle (-g)^{\frac{\alpha}{\alpha - 1}}, \theta \rangle\right)
        \right\}.
    \end{align*}
    Since the suprema are interchangeable, 
    we can, 
    for every $z \in \R$, 
    substitute each $g \in \C(X)$ with $\theta$-a.e. fulfilling $g < 0$
    by the $\theta$-a.e. negative function 
    $h \coloneqq \frac{\alpha}{1 - \alpha}e^{-z}g \in \C(X)$
    to obtain
    \begin{align} 
        \RR_\alpha(\mu \mid \theta)
        & = \sup_{z \in \R} \sup_{\substack{h \in \C(X) \\ h < 0 \ \theta\text{-a.e.}}} 
        \left\{ \frac{\alpha(1 - z)}{1 - \alpha} + \langle h, \mu \rangle
        - \ln\left(\left\langle \left({\frac{1 - \alpha}{\alpha}}e^{z}(-h)\right)^{\frac{\alpha}{\alpha - 1}}, \theta\right\rangle\right)
        \right\} \notag\\
        & = \sup_{z \in \R} \sup_{\substack{h \in \C(X) \\ h < 0 \ \theta\text{-a.e.}}}
        \Biggl\{ \frac{\alpha(1 - z)}{1 - \alpha} + \langle h, \mu \rangle
        - \frac{\alpha}{\alpha - 1} \ln\left( \frac{1 - \alpha}{\alpha} e^{z} \right) \notag\\
        & \hspace{6cm}
        - \ln\left(\langle (-h)^{\frac{\alpha}{\alpha - 1}}, \theta \rangle\right)
        \Biggr\} \notag\\
        & = \sup_{z \in \R} \sup_{\substack{h \in \C(X) \\ h < 0 \ \theta\text{-a.e.}}}
        \Biggl\{ \frac{\alpha(1 - z)}{1 - \alpha} + \langle h, \mu \rangle
        - \frac{\alpha}{\alpha - 1}\ln\left(\frac{1 - \alpha}{\alpha}\right)\notag\\ 
        & \hspace{6cm} 
        - \frac{\alpha}{\alpha - 1} z
        - \ln\left(\langle (-h)^{\frac{\alpha}{\alpha - 1}}, \theta \rangle\right)
        \Biggr\}\notag\\
        & = \sup_{\substack{h \in \C(X) \\ h < 0 \ \theta\text{-a.e.}}}
         \langle h, \mu \rangle
        - \ln\left(\langle (-h)^{\frac{\alpha}{\alpha - 1}}, \theta \rangle\right)
        - C_\alpha.\label{eq:ConvConjFinale}
    \end{align}
\end{proof}
We can now derive the dual formulation of the Rényi regularized OT distance \eqref{eq:OTeps,alpha}.
For $f, g \in \C(X)$,
we define $f \oplus g \in \C(X \times X)$ 
via 
\begin{equation*}
    f \oplus g \colon X \times X \to \R, \qquad
    (x, y) \mapsto f(x) + g(y).
\end{equation*}

\begin{theorem}[Dual formulation of $\OT_{\eps, \alpha}$, primal-dual relation] \label{thm:dual_formulation}
    Let $X$ be compact, $\eps > 0$, $\alpha \in (0, 1)$ and $\mu,\nu \in \P(X)$.
    \begin{enumerate}
        \item 
        The (pre)dual formulation of the Rényi regularized OT distance \eqref{eq:OTeps,alpha} is
        \begin{equation} \label{eq:dualdualSinkhornProblem}
            \OT_{\eps, \alpha}(\mu, \nu)
            = \sup_{\substack{f, g \in \C(X) \\ f \oplus g < c \ \mu \otimes \nu\text{-a.e.}}} \langle f \oplus g, \mu \otimes \nu \rangle - \eps \ln(\langle \gamma_{\alpha, f \oplus g}, \mu \otimes \nu \rangle) + C_{\alpha, \eps},
        \end{equation}
        where
        \begin{equation*}
            \gamma_{\alpha, f \oplus g}
            \coloneqq ( c - f \oplus g)^{\frac{\alpha}{\alpha - 1}}
        \end{equation*}
        and $C_{\alpha, \eps} \coloneqq - \eps \frac{\alpha}{1 - \alpha} \ln(\eps) - \eps C_{\alpha}$ where $C_{\alpha}$ is defined in \Cref{theorem:RenyiPreconjugate}.

        \item 
        The dual potentials from 1. exist, 
        i.e. the supremum in \eqref{eq:dualdualSinkhornProblem} is a maximum.

        \item 
        If $(\hat f,\hat g)$ is an optimal dual solution of \eqref{eq:dualdualSinkhornProblem},
        then the unique primal solution $\pi_{c}^{\alpha, \eps}$ of \eqref{eq:dualSinkhornProblem} can be represented as 
        \begin{equation} \label{eq:primal-dual-pi}
            \pi_c^{\alpha, \eps}
            = \left(\frac{( c - \hat{f} \oplus \hat{g} )^{\frac{1}{\alpha - 1}}}{\langle (c - \hat{f} \oplus \hat{g})^{\frac{1}{\alpha - 1}}, \mu\otimes\nu \rangle}\right) \mu \otimes \nu.
        \end{equation}
    \end{enumerate}
\end{theorem}

\begin{proof}
    \begin{enumerate}
        \item 
        We will apply the following formulation of the Fenchel-Rockafellar theorem: 
        let $V, W$ be real Banach spaces, 
        $F \in \Gamma_0(V)$, $G \in \Gamma_0(W)$ 
        and $A \colon V \to W$ be linear and bounded. 
        If there is a $x \in \dom(F)$ such that $G$ is continuous at $A x$, 
        then 
        \begin{equation} \label{eq:FMR}
            \sup_{x \in V} - F(-x) - G(A x)
            = \inf_{w \in W^*} F^*(A^* w) + G^*(w).
        \end{equation}
        If a solution $\hat{x}$ of the sup-problem exists, 
        then for any solution $\hat{w}$ to the inf-problem 
        we have $\Lambda_{W}(A \hat{x}) \in \partial G^*(\hat{w})$ \cite[Thm.~II.4.1, Prop.~II.4.1]{ET99}.
        
        We want to identify the right hand side of \eqref{eq:FMR} with \eqref{eq:OTeps,alpha} 
        and thus choose $V = \C(X) \times\C(X)$, $W \coloneqq \C(X \times X)$, 
        which entails $V^* \cong \M(X) \times \M(X)$ and $W^* \cong \M(X \times X)$ 
        with the dual pairing $\langle f, \pi \rangle \coloneqq \int_{X \times X} f(x, y) \diff{\pi}(x, y)$ for $f \in W$ and $\pi \in W^*$.
        Furthermore, 
        we let
        \begin{equation} \label{eq:A}
            A \colon \C(X) \times \C(X) \to \C(X \times X), \qquad
            (f, g) \mapsto f \oplus g.
        \end{equation}
        This operator is linear and bounded, 
        since $\| f \oplus g \|_{\infty} \le \| f \|_{\infty} + \| g \|_{\infty}$.
        Then $A^* \colon \M(X \times X) \to \M(X) \times \M(X)$ is given by $A^*(\pi) = (P^{(1)}_{\#} \pi, P^{(2)}_{\#} \pi)$, 
        since for all $f, g \in \C(X)$ we have
        \begin{align*}
            \langle A(f,g), \pi\rangle
            & = \int_{X}\int_{X} f(P^{(1)}(x,y)) \diff \pi(x,y) + \int_{X}\int_{X} g(P^{(2)}(x,y)) \diff \pi(x,y) \\
            & = \int_{X} f(x) \diff[P_\#^{(1)}\pi](x) + \int_{X} g(y) \diff [P_\#^{(2)}\pi](y).
        \end{align*}
        Furthermore, we set
        \begin{equation*}
            F \colon \C(X) \times \C(X) \to \R, \qquad
            (f, g) \mapsto \langle f \oplus g, \mu \otimes \nu \rangle
        \end{equation*}
        which is clearly convex and continuous, 
        and set
        \begin{align*}
            G \colon \C(X \times X) 
            &\to \overline{\R}, \\
            h & \mapsto 
            \begin{cases}
                \eps \ln\left( \langle (c - h)^{\frac{\alpha}{\alpha - 1}}, \mu \otimes \nu \rangle\right) - C_{\alpha, \eps}, & \text{if } h < c \; \mu\otimes\nu\text{-a.e.}, \\
                \infty, & \text{else.}
            \end{cases}
        \end{align*}
        and 
        $E \coloneqq \{ h \in \C(X \times X): h < c \; \mu\otimes\nu-\text{a.e.}\}$.
        To show that $G$ is convex, 
        we only have to show that $\ln\left( \langle (c - \cdot)^{\frac{\alpha}{\alpha - 1}}, \mu \otimes \nu \rangle\right)$ is convex on $E$.
        This follows from the fact that mixtures of log-convex functions are log-convex again. 
        More precisely, 
        the result is obtained when choosing 
        $Y = X \times X$, 
        and 
        \begin{align}   \label{eq:Phi_gamma_alpha_c}
            \Phi(h, (x,y)) \coloneqq (c(x, y) - h(x, y))^{\frac{\alpha}{\alpha - 1}}
        \end{align}
        and $\theta = \mu \otimes \nu$ in \Cref{prop:LogConvexMixture},
        since $\Phi$ is (strictly) log-convex
        using that $z \mapsto z^{\frac{\alpha}{\alpha-1}}$ is (strictly) log-convex
        for $z > 0$.

        Now, 
        we show that $G$ is lower semicontinuous.
        As composition of continuous well-defined functions,
        $G$ is continuous on $E$, 
        so we only need to check lower semicontinuity
        on the boundary of $E$.
        The closure of $E$ is then
        \begin{align*}
            \overline{E} 
            = \{h \in \C(X\times X) : h - c \le 0 \ \mu\otimes\nu-\text{a.e.} \},
        \end{align*}
        and hence the the boundary of $E$ is given by 
        \begin{align*}
            \partial E 
            \coloneqq & \,\overline E \setminus E^\circ \\
            = & \{h \in \C(X \times X) : \exists (x_0, y_0) \in \supp(\mu \otimes \nu) : h(x_0, y_0) = c(x_0, y_0) \\
            & \hspace{40pt} \land (\mu\otimes\nu)(\{(x,y) \in X \times X : h(x,y) > c(x,y)\}) = 0\}.
        \end{align*}
        For verifying the lower semicontinuity, the crucial part of the boundary are those functions $h$ such that $h-c$ has zeros.
        Let $(h_n)_{n \in \N} \subset E$ be a sequence 
        converging to $h \in \partial E$, 
        realizing the limes inferior.
        The set $\{(x_0, y_0) \in \supp(\mu\otimes\nu) : h(x_0, y_0) = c(x_0, y_0)\}$
        has a subset, say $B_h$, of positive measure w.r.t. $\mu\otimes\nu$.
        Now,
        by the convergence $h_n \to h$,
        there exists
        for any $\varepsilon > 0$
        an $N \in \N$ such that
        \begin{align*}
            |c(x, y) - h_n(x,y)|
            \leq \underbrace{|c(x, y) - h(x,y)|}_{= 0} + |h(x, y) - h_n(x,y)| < \varepsilon
        \end{align*}
        for any $(x,y) \in B_h$ and $n \geq N$.
        Hence,
        by the monotonicity of $s \mapsto s^{\tfrac{\alpha}{\alpha-1}}$,
        we obtain
        \begin{align*}
            \langle (c - h_n)^{\tfrac{\alpha}{\alpha-1}}, \mu\otimes\nu \rangle
            \geq \int_{B_h} (c - h_n)^{\tfrac{\alpha}{\alpha-1}} \diff \mu\otimes\nu 
            &\geq \int_{B_h} \varepsilon^{\tfrac{\alpha}{\alpha-1}} \diff \mu\otimes\nu \\
            = (\mu\otimes\nu)(B_h) \cdot \varepsilon^{\tfrac{\alpha}{\alpha-1}}.
        \end{align*}
        By the continuity of the natural logarithm 
        and the monotonicity of $\liminf$,
        we obtain 
        \begin{align*}
            \liminf_{n \to \infty} G(h_n) = G(h) = \infty,
        \end{align*}
        which shows that $G$ is lower semicontinuous.
        
        Furthermore, $\dom(F) = \C(X) \times \C(X)$ and $G$ is continuous 
        on the interior of its domain, e.g. at $f \oplus g = - 1 < c$, 
        so that the constraint qualification is fulfilled.

        Before calculating $G^*$, we show that $G^*(\pi) = \infty$ for $\pi \in \M(X \times X) \setminus \M_+(X \times X)$.
        In this case, we can write $\pi = \pi_+ - \pi_-$ with $\pi_+, \pi_- \in \M_+(X \times X)$.
        By the definition of the conjugate, we have
        \begin{equation} \label{eq:G*}
            G^*(\pi)
            = \sup_{h \in \C(X \times X)} \langle h, \pi \rangle - \eps \ln\left( \langle (c - h)^{\frac{\alpha}{\alpha - 1}}, \mu \otimes \nu \rangle\right) + C_{\alpha, \eps} - \iota_{E}(h).
        \end{equation}
        Now, choose a sequence $(h_n)_{n \in \N} \subset E$ with $h_n|_{\supp(\pi_+)} \equiv 0$ and $h_n|_{\supp(\pi_-)} \le 0$, such that on a subset of $\supp(\pi_-)$ of positive measure, $h_n \to - \infty$ pointwise.
        Plugging in $h = h_n$ with $n \in \N$, up to constants, the objective in \eqref{eq:G*} becomes
        \begin{equation*}
            \underbrace{\langle h_n, \pi_- \rangle}_{\xrightarrow{n \to \infty} \infty} - \eps \ln\bigg( \langle \underbrace{(c - h_n)^{\frac{\alpha}{\alpha - 1}}}_{\searrow 0, n \to \infty}, \underbrace{\mu \otimes \nu}_{\ge 0} \rangle\bigg)
            \xrightarrow{n \to \infty} \infty - (- \infty)
            = \infty.
        \end{equation*}
        Hence, $\dom(G^*) \subset \M_+(X \times X)$.
        
        For $\pi \in \P(X \times X)$ we have, 
        using the substitution $g = \frac{1}{\eps} (h - c)$,
        \begin{align*}
            G^*(\pi)
            & = \sup_{h \in \C(X \times X)} \langle h, \pi \rangle - \eps \ln\left( \langle (c - h)^{\frac{\alpha}{\alpha - 1}}, \mu \otimes \nu \rangle\right) + C_{\alpha, \eps} - \iota_{E}(h) \\
            & = \sup_{\substack{g \in \C(X \times X) \\ g < 0 \ \mu\otimes\nu\text{-a.e.}}} \left\langle \eps g + c, \pi \right\rangle - \eps \ln\left( \left\langle \left(- \eps g\right)^{\frac{\alpha}{\alpha - 1}}, \mu \otimes \nu \right\rangle\right) + C_{\alpha, \eps} \\
            & = C_{\alpha, \eps} + \langle c, \pi \rangle
            + \eps \sup_{\substack{g \in \C(X \times X) \\ g < 0 \ \mu\otimes\nu\text{-a.e.}}} \langle g, \pi \rangle - \ln\left( \langle (- \eps g)^{\frac{\alpha}{\alpha - 1}}, \mu \otimes \nu \rangle\right) \\
            & = C_{\alpha, \eps} + \langle c, \pi \rangle
            + \eps \frac{\alpha}{1 - \alpha} \ln(\eps) \\
            & \qquad + \eps \sup_{\substack{g \in \C(X \times X) \\ g < 0 \ \mu\otimes\nu\text{-a.e.}}} \langle g, \pi \rangle - \ln\left( \langle (- g)^{\frac{\alpha}{\alpha - 1}}, \mu \otimes \nu \rangle\right) \\
            & = \langle c, \pi \rangle + \eps \overline{R}_{\alpha}(\pi \mid \mu \otimes \nu)
            + \underbrace{C_{\alpha, \eps} + \eps \frac{\alpha}{1 - \alpha} \ln(\eps) + \eps C_{\alpha}}_{= 0},
        \end{align*}
        using \Cref{theorem:RenyiPreconjugate} in the last line.
        Lastly, for $\xi, \theta \in \M(X)$ we have
        \begin{align*}
            F^*(\xi, \theta)
            & = \sup_{f, g \in \C(X)} \langle f, \xi - \mu \rangle + \langle g, \theta - \nu \rangle
            = \iota_{\{ (\mu, \nu) \}}(\xi, \theta).
        \end{align*}
        Hence, the Fenchel-Rockafellar theorem \eqref{eq:FMR} yields
        \begin{align*}
            & \sup_{\substack{f, g \in \C(X)^2 \\ f \oplus g < c \ \mu \otimes \nu-\text{a.e.}}} \langle f \oplus g, \mu \otimes \nu \rangle 
            - \eps \ln\left( \langle (c - f \oplus g)^{\frac{\alpha}{\alpha - 1}}, \mu \otimes \nu \rangle\right) + C_{\alpha, \eps} \\
            & = \sup_{x \in V} - F(-x) - G(A x)
            = \inf_{w \in W^*} F^*(A^* w) + G^*(w) \\
            & = \inf_{\pi \in \M_+(X \times X)} \iota_{(\mu, \nu)}\left((\pi_1)_\# \pi, (\pi_2)_{\#} \pi\right) + G^*(w) \\
            & = \inf_{\pi \in \Pi(\mu, \nu)} \langle c, \pi \rangle + \eps R_{\alpha}(\pi \mid \mu \otimes \nu).
        \end{align*}
        Here we used that if $\pi \in \M_+(X \times X)$ has marginals $\mu, \nu \in \P(X)$, then $\pi \in \P(X)$.

        \item 
        We first study, 
        for a constant $c > 0$, the auxiliary function 
        \begin{equation} \label{eq:help_function}
            \eta_{\alpha,c} : (-\infty, c) \to \R, 
            \quad 
            x \mapsto x\left(1 + \tfrac{\alpha}{1-\alpha}\tfrac{\ln(c - x)}{x}\right)
        \end{equation}
        which has a unique maximizer $x^*$
        and is strictly increasing on $(-\infty, x^*]$
        and strictly decreasing on $[x^*, c)$,
        moreover $\eta_{\alpha,c}$ is strictly concave.
        We now identify the function $\eta_{\alpha,c}$ 
        from \eqref{eq:help_function}
        as upper bound for the objective in the dual optimization problem \eqref{eq:dual_pre_formulation}.
        By definition of the supremum, 
        there is a sequence $((f_n, g_n))_{n \in \N}$
        with $f_n \oplus g_n < c$ $\mu\otimes\nu$-a.s. for all $n \in \N$,
        i.e. $(f_n, g_n) \in E$ for all $n \in \N$,
        such that
        \begin{align}
            \label{eq:lim_dual}
            &\lim_{n \to \infty} \langle f_n \oplus g_n, \mu \otimes \nu \rangle - \eps \ln\left( \langle \gamma_{\alpha, f_n \oplus g_n}, \mu \otimes \nu \rangle\right) \notag \\
            &\hspace{20pt}= \sup_{\substack{f, g \in \C(X), \\ f \oplus g < c \ \mu\otimes \nu \text{-a.e.}}} \langle f \oplus g, \mu \otimes \nu \rangle - \eps \ln\left( \langle \gamma_{\alpha, f \oplus g}, \mu \otimes \nu \rangle\right).
        \end{align}
        In analogy 
        to the proof of the lower semicontinuity of $G$,
        we towards contraction, 
        suppose that 
        \begin{align*}
            \{(x_0,y_0) \in \supp(\mu \otimes \nu) : \lim_{n\to \infty}(f_n \oplus g_n)(x_0,y_0) = c(x_0,y_0)\}
        \end{align*}
        and/or
        \begin{align*}
            \{(x_0,y_0) \in \supp(\mu \otimes \nu) 
            : \forall \eta > 0
            \;\exists n \in \N \;\text{s.t.}\; 
            |(f_n \oplus g_n)(x_0,y_0)| > \eta\}
        \end{align*} 
        are not $\mu\otimes\nu$ null sets
        (by definition of the support, there is a subset with positive mass),
        which means that the supremum is not attained by a maximizer in the set $E$.
        Then,
        for the first case 
        the right hand side in \eqref{eq:lim_dual} is infinite,
        by the lower semicontinuity of $G$; 
        however, by strong duality from \Cref{thm:dual_formulation}
        it must be equal to \eqref{eq:OTeps,alpha}, 
        which is always finite, 
        a contradiction.
        For the second case,
        we observe that the function 
        $\theta \colon (0, \infty) \to (0, \infty)$, $x \mapsto x^{\frac{\alpha}{\alpha - 1}}$ 
        is strictly convex for any $\alpha \in (0,1)$.
        The function $x \mapsto -\ln(x)$ is strictly decreasing. 
        Hence,
        Jensen's inequality yields for any $0 < \gamma \in \C(X \times X)$ that
        \begin{equation}
            \label{eq:rn_opt_plan}
            - \ln\left(\langle \theta \circ \gamma, \mu \otimes \nu \rangle\right)
            \le - \ln \left(\langle \gamma, \mu \otimes \nu \rangle^{\frac{\alpha}{\alpha - 1}}\right)
            = - \frac{\alpha}{\alpha - 1} \ln\left( \langle  \gamma, \mu \otimes \nu \rangle\right).
        \end{equation}
        Denote by $\Psi$ the objective of the dual problem in \eqref{eq:dualdualSinkhornProblem}.
        Hence,
        \begin{align*}
            \Psi(f,g) 
            &\leq \langle f \oplus g, \mu\otimes\nu\rangle + \tfrac{\alpha}{1-\alpha} \ln(\langle c -  f \oplus g, \mu\otimes\nu\rangle) \\
            &= \langle f \oplus g, \mu\otimes\nu\rangle\left(1 + \tfrac{\alpha}{1-\alpha}\frac{\ln(\langle c - f \oplus g, \mu\otimes\nu\rangle)}{\langle f \oplus g, \mu\otimes\nu\rangle}\right) \\
            &\leq \eta_{\alpha,\|c\|_\infty}(\langle f \oplus g, \mu\otimes\nu\rangle).
        \end{align*}
        using the monotonicity of the natural logarithm.
        Now,
        using \eqref{eq:help_function},
        if $f_n\oplus g_n$ are not bounded from below 
        on $\supp(\mu\otimes\nu)$,
        the right hand side is not bounded either,
        which contradicts the strong duality from \Cref{thm:dual_formulation}, again.

        \item 
        By \Cref{theorem:PiDivReg}, there exists a unique primal solution $\hat \pi \in \P(X\times X)$.
        Hence, 
        we know that the optimal dual solution $(\hat f,\hat g) \in \C(X)^2$ 
        fulfills $A(\hat f,\hat g) \in \partial G^*(\hat\pi)$, 
        that is,
        \begin{equation*}
            \hat{f} \oplus \hat{g} = c + \eps\phi,
        \end{equation*}
        where $\phi \in \partial\big[\overline \RR_\alpha(\cdot \mid \mu\otimes\nu)\big](\hat{\pi})$.
        We now want to find $\phi$.
        Define
        \begin{equation*}
            H \coloneqq [\overline{ \RR}_\alpha(\,\cdot\mid\mu\otimes\nu)]_*
        \end{equation*} 
        and
        \begin{equation*}
            \psi \colon \R \to \overline{\R}, \qquad x \mapsto \begin{cases}
                (-x)^{\frac{\alpha}{\alpha - 1}}, & \text{if } x<0, \\
                + \infty, & \text{else.}
            \end{cases},
        \end{equation*}
        which is differentiable on $\dom(\psi)$
        with $\psi'(x) = \frac{\alpha}{1 - \alpha}(-x)^{\frac{1}{\alpha - 1}}$ if $x<0$ 
        and $\partial \psi(x) = \emptyset$ otherwise.
        Notice that $H = \ln (\langle \psi \circ \cdot , \mu\otimes \nu\rangle)$.
        Since $\langle \psi\circ \phi, \mu\otimes\nu \rangle > 0$ 
        for $\phi \in \C(X \times X)$ with $\phi < 0$ $\mu\otimes\nu$-a.e.
        and $\ln \colon (0, \infty) \to \R$ is differentiable, 
        we have (see also \cite[Thm.~2.8.10]{Z2002})
        \begin{equation*}
            \partial H(\phi)
            = \begin{cases}
                \left\{ \frac{1}{\langle \psi\circ\phi, \mu\otimes\nu \rangle}(\psi' \circ \phi)\cdot (\mu\otimes\nu) \right\}, & \text{if } \phi < 0 \; \mu\otimes\nu\text{-a.e.}, \\
                \emptyset, & \textrm{else.}
            \end{cases}
        \end{equation*}
        One can show that
        $\partial H(\phi) = \emptyset$ holds 
        by utilizing a similar argumentation as above for the lower semicontinuity of $G$
        showing that $H(\varphi) = \infty$,
        for any $\varphi \in \C(X \times X)$
        with $\varphi \leq 0$ $\mu\otimes\nu$-a.s.
        and $\phi(x_0,y_0) = 0$ for some $(x_0,y_0) \in \supp(\mu\otimes\nu)$.
        Moreover, 
        $\overline \RR_\alpha(\,\cdot \mid \mu \otimes \nu)$ is strictly convex,
        using the Fenchel conjugate pair $H$ with $H^* = \overline{\RR}_\alpha$ 
        (see \Cref{theorem:RenyiPreconjugate})
        yields, due to lower semicontinuity of $G$,
        \begin{equation*}
            \phi \in \partial\big[\overline \RR_\alpha(\cdot \mid \mu\otimes\nu)\big](\hat{\pi})
            \iff \hat \pi \in \partial H(\phi),
        \end{equation*}
        since $\varphi \in \dom(H)$, i.e. $\varphi < 0$
        $\mu\otimes\nu$-a.e. \cite[Thm.~4.1]{ET99},
        similar to \cite{NS21}.
        \begin{align}
            \label{eq:dual_pre_formulation}
            \hat \pi
            = \frac{\alpha(1 - \alpha)^{- 1} \big(\frac{1}{\eps} (c - \hat f\oplus \hat g) \big)^{\frac{1}{\alpha - 1}}}{\langle \big(\frac{1}{\eps} (c - \hat f\oplus \hat g)\big)^{\frac{\alpha}{\alpha - 1}}, \mu\otimes\nu \rangle}\mu\otimes\nu,
        \end{align}
        by reinserting $\phi = \frac{1}{\eps}(\hat f\oplus \hat g - c)$.
        
        Lastly, the latter presentation of the solution $\pi_c^{\alpha, \eps}$ can be simplified as follows:
        since $\pi_c^{\alpha, \eps} \in \P(X\times X) $, 
        we have
        \begin{align*}
            1 = \langle 1, \pi_c^{\alpha, \eps} \rangle
            = \frac{\alpha(1 - \alpha)^{-1}}{\langle (\frac{1}{\eps}(c - \hat{f} \oplus \hat{g}))^{\frac{\alpha}{\alpha - 1}}, \mu\otimes\nu \rangle} \left\langle \left(\frac{1}{\eps}(c - \hat{f} \oplus \hat{g})\right)^{\frac{1}{\alpha - 1}}, \mu\otimes\nu \right\rangle \\
            \iff
            \frac{\langle (\frac{1}{\eps}(c - \hat{f} \oplus \hat{g}))^{\frac{\alpha}{\alpha - 1}}, \mu\otimes\nu \rangle}{\alpha(1 - \alpha)^{-1}} 
            = \left\langle \left(\frac{1}{\eps}(c - \hat{f} \oplus \hat{g})\right)^{\frac{1}{\alpha - 1}}, \mu\otimes\nu \right\rangle.
        \end{align*}
        and therewith 
        \begin{equation*}
            \pi_c^{\alpha, \eps} 
            = \frac{\alpha (1 - \alpha)^{-1}(\frac{1}{\eps}(c - \hat{f} \oplus \hat{g}))^{\frac{1}{\alpha - 1}}}{\langle (\frac{1}{\eps}(c - \hat{f} \oplus \hat{g}))^{\frac{1}{\alpha - 1}}, \mu\otimes\nu \rangle} \mu\otimes\nu
            = \frac{(c - \hat{f} \oplus \hat{g})^{\frac{1}{\alpha - 1}}}{\langle (c - \hat{f} \oplus \hat{g})^{\frac{1}{\alpha - 1}}, \mu\otimes\nu \rangle} \mu\otimes\nu.
            \qedhere
        \end{equation*}
    \end{enumerate}
\end{proof}

\begin{corollary}[Support of the optimal regularized transport plan] 
    Let $X$ be compact,
    $\eps > 0$, $\alpha \in (0,1)$
    and $\mu, \nu \in \P(X)$.
    Then,
    $\supp(\pi_c^{\alpha, \eps}) = \supp(\mu\otimes\nu)$.
\end{corollary}
\begin{proof}
    This is a direct conclusion from \Cref{thm:dual_formulation},
    by the positivity of the density of $\pi_c^{\alpha, \eps}$
    with respect to $\mu\otimes\nu$.
\end{proof}

\begin{remark}[Primal dual relation for other regularizations]
    The primal-dual relationship \eqref{eq:primal-dual-pi} is another aspect 
    where Rényi and KL regularization differ significantly: 
    for the latter one, the first-order optimality conditions can be solved
    for the optimal potentials 
    and so define the so-called Sinkhorn operator~\cite[Eq.~(11)]{FSVATP19}~\cite[Eq.~(26)]{NS21} 
    or Schrödinger map~\cite[Eq.~(2)]{CCL22}.
    This is not possible for the Rényi regularized setting 
    due to the renormalization term in the denominator.
    
    In the case $\alpha = 1$, the density $\frac{\diff \pi_{c}^{\alpha, \eps}}{\diff{\mu} \otimes \nu}$ instead is $\exp\left( - \frac{1}{\eps} \big( c - f \oplus g \big) \right)$~\cite[Eq.~(25)]{NS21}.

    In the case of discrete measures, the primal-dual relationship \eqref{eq:primal-dual-pi} for the optimal plan $\MP \in \R^{d \times d}$ becomes
    \begin{equation} \label{eq:primal-dual_discrete}
        p_{i, j}
        = \begin{cases} \displaystyle\frac{\big(m_{ij} + \nu_{i} + \xi_{j}\big)^{\frac{1}{\alpha - 1}}}{\sum_{k, \ell = 1}^{d} \big( m_{k\ell} + \nu_{k} + \xi_{\ell} \big)^{\frac{1}{\alpha - 1}}r_k c_\ell}r_i c_j, & \text{if } r_i c_j > 0, \\
        0, & \text{else.}
        \end{cases}
    \end{equation}
    where $\Vr, \Vc$ are the discrete marginals, $m_{i, j}$ are entries of the cost matrix, and $\Vnu, \Vxi \in \R^d$ are the Lagrange multipliers 
    corresponding to the affine constraints in $U(\Vr, \Vc)$, the discrete version of transport polytope.
    It is not difficult to show that these Lagrange multipliers are related to the optimal dual solution $(\Vf, \Vg)$ via $- [\Vf, \Vg]^{\tT} = [\Vnu, \Vxi]^{\tT}$.
    
    The corresponding formula for $\alpha = 1$ is 
    $p_{i, j} = \exp\left(- \frac{1}{\eps} (\nu_i + \xi_j + m_{i, j})\right)$~\cite[Lem.~2]{C13},
    for the $q$-Tsallis entropy it is 
    $p_{i, j} = \exp_q( - 1) \frac{1}{\exp_q(\nu_i + \xi_j + \frac{1}{\eps} m_{i, j})}$ 
   ~\cite[Thm.~4]{MNPN17}, where $\exp_q$ is the $q$-exponential
    and for the Gini entropy it is 
    $p_{i,j} = -\frac{1}{2}(\frac{1}{\eps}(m_{i,j} + \nu_i + \xi_j) - 1)$.
\end{remark}

\begin{remark}[Dual formulation on non-compact $X$] \label{rem:noncompact}
    Our approach for the predual formulation does not work 
    on non-compact spaces like $X = \R^d$,
    since then, the predual space of $\M(X)$ is $\C_0(X)$.

    Taking the same approach as above, the preconjugate would be finite only for negative functions $s$ in $\C_0(\R^d\times \R^d)$.
    In turn, any preconjugate of the function $F$ is also defined on $\C_0(\R^d\times \R^d)$, 
    which then is by \Cref{theorem:ConvConjRenyiDiv} a preconjugate
    of the Rényi divergence with the scaled and biased argument (by $\eps$ and $c$).
    Notice that $s - c \not\in \C_0(\R^d\times \R^d)$ in general.
    Hence, the dual problem would be equal to minus infinity at any point.
    A dual formulation on non-compact domain might hence be possible to obtain if either
    \begin{itemize}
        \item
        the cost function $c$ is also $\C_0(\R^d\times \R^d)$, 
        which is not a sensible assumption or
        
        \item
        the measures are compactly supported, 
        such that, very informally, the constraints of the convex conjugate of the Rényi divergence become $\C(\R^d\times \R^d)$ instead of $\C_0(\R^d\times \R^d)$, as in case where $X$ is compact.
    \end{itemize}
\end{remark}

The next proposition
is also well-known for KL regularized OT~\cite[Prop.~5.7]{NS21}\cite[Prop.~11]{FLA2021}
in the continuous setting,
as well as for the Tsallis regularized OT~\cite[Thm.~4]{MNPN17}
in the discrete setting.
The proof directly uses \Cref{thm:dual_formulation}.1,
i.e. \eqref{eq:Phi_gamma_alpha_c} and \Cref{prop:LogConvexMixture}.

\begin{proposition}[Uniqueness of dual potentials up to a constant]
\label{prop:UniConstants}
    Let $X$ be compact, $\eps > 0$
    and $\mu,\nu \in \P(X)$.
    The optimal dual potentials $\hat f,\hat g \in \C(X)$
    of \eqref{eq:dualSinkhornProblem} are unique on
    $\supp(\mu)$ and $\supp(\nu)$, respectively, up to additive constants.
    That is, for two optimal pairs $(\hat f,\hat g),(\tilde f,\tilde g) \in \C(X)^2$
    there exists a $\gamma \in \R$ such that $\hat f - \tilde f \equiv \gamma$ $\mu$-almost everywhere and $\hat g - \tilde g \equiv -\gamma$ $\nu$-almost everywhere.
\end{proposition}
\begin{proof}
    We define the objective in \Cref{thm:dual_formulation} of the dual optimization problem 
    \begin{align*}
        \Psi(f,g)
        \coloneqq 
        \langle f \oplus g, \mu \otimes \nu \rangle
        - \eps
        \ln(\langle \gamma_{\alpha,f\oplus g}, \mu\otimes\nu\rangle) + C_{\alpha, \eps}, \qquad f, g \in \C(X).
    \end{align*}
    The function $\Psi$ is strictly concave 
    since from the proof of \Cref{thm:dual_formulation} we know that $\Phi$ 
    is strictly log-convex, cf.~\Cref{prop:LogConvexMixture}.
    Moreover, we define 
    \begin{align*}
        \varphi(t) 
        \coloneqq 
        \Psi(f_t, g_t)
        \qquad\text{where}
        \quad
        f_t \coloneqq \hat f + t(\tilde f - \hat f),
        \qquad
        g_t \coloneqq \hat g + t(\tilde g - \hat g)
    \end{align*}
    and $(\hat f, \hat g),(\tilde f, \tilde g) \in \C(X)^2$
    are two pairs of optimal potentials of the considered problems
    by \Cref{thm:dual_formulation}.
    Notice that $\varphi(0) = \varphi(1) = \OT_{\eps, \alpha}(\mu,\nu)$
    by assumption, i.e. for all $t \in (0,1)$ holds
    \begin{align*}
        \varphi(t) \le \OT_{\eps, \alpha}(\mu,\nu)
    \end{align*}
    by optimality.
    For $t = \frac{1}{2}$ the convex combination becomes 
    \begin{equation*}
        f_{\frac{1}{2}}\oplus g_{\frac{1}{2}}
        = 
        \frac{1}{2}(\tilde f + \hat f) \oplus (\tilde g + \hat g).
    \end{equation*}
    We obtain 
    \begin{align*}
        \OT_{\eps, \alpha}(\mu,\nu) - C_{\alpha, \eps}
        & = \varphi(0) - C_{\alpha, \eps}
        \ge \varphi\left(\frac{1}{2}\right) - C_{\alpha, \eps}\\
        & = \frac{1}{2}
        \langle \tilde f \oplus \tilde g,\mu\otimes\nu\rangle
        + \frac{1}{2}
        \langle \hat f \oplus \hat g,\mu\otimes\nu\rangle
        - \eps\ln(\langle \gamma_{f_{\frac{1}{2}}\oplus g_{\frac{1}{2}}}^{\frac{\alpha}{\alpha-1}},\mu\otimes\nu\rangle) \\
        & > \frac{1}{2}\varphi(0) + \frac{1}{2}\varphi(1) - C_{\alpha, \eps}
        = \OT_{\eps, \alpha}(\mu,\nu) - C_{\alpha, \eps}
    \end{align*}
    using the strict concavity of the function $\varphi$,
    a contradiction.
    Hence, without loss of generality it holds that
    \begin{align*}
        f_{\frac{1}{2}}\oplus g_{\frac{1}{2}} & = \tilde f \oplus \tilde g
        &&\mu\otimes\nu\text{-a.s.} \\
        \Longleftrightarrow
        -\frac{1}{2}(\tilde f(x) - \hat f(x))
        & = \frac{1}{2}(\tilde g(y) - \hat g(y))
        &&\mu\otimes\nu\text{-a.s.},
    \end{align*}
    i.e. there exists $\gamma \in \R$ such that 
    $\tilde f - \hat f \equiv \gamma$ $\mu$-almost everywhere
    and $\tilde g - \hat g \equiv -\gamma$ $\nu$-almost everywhere.
\end{proof}

\subsection{Interpolation Properties of Rényi Regularized Optimal Transport}
\label{sec:InterpolatingRenyiOT}

The following theorem
generalizes the well-known convergence properties of KL regularized OT~\cite{FSVATP19,PC19,NS21}.
In short, $\OT_{\eps, \alpha}$ interpolates between unregularized OT and $\langle c, \mu \otimes \nu \rangle$ for $\eps \to \{ 0,\infty\}$ and is continuous $\alpha \in (0,1]$, meaning in particular that $\OT_{\eps, \alpha}$ interpolates between unregularized OT and KL regularized OT.

\begin{theorem}[Convergence in $\eps$ and {$\alpha \in (0, 1]$}]
    \label{theorem:ConvToOTKL}
    Let $X$ be compact, $\mu,\nu \in \P(X), \eps > 0$ and $\alpha \in (0, 1]$.
    Then, for fixed $\alpha \in (0,1)$, we have
    \begin{alignat}{2}
        & \OT_{\eps, \alpha}(\mu,\nu) \to \OT(\mu,\nu), 
        &&\quad \eps \to 0, \label{eq:convLamInfinity}\\
        & \OT_{\eps, \alpha}(\mu,\nu) \to \langle c,\mu\otimes\nu \rangle,
        &&\quad \eps \to \infty \label{eq:convLamZero}
    \end{alignat}
    and
    \begin{alignat*}{2}
        \pi_c^{\alpha, \eps}(\mu, \nu)
        & \rightharpoonup \argmin\big\{ \RR_\alpha(\pi \mid \mu\otimes\nu): 
        \pi \in \Pi(\mu,\nu), \ \langle c, \pi\rangle = \OT(\mu, \nu) \big\}, 
        &&\qquad \eps \to 0, \\
        \pi_c^{\alpha, \eps}(\mu, \nu)
        & \rightharpoonup \mu\otimes\nu,
        &&\qquad \eps \to \infty,
    \end{alignat*}
    where $\rightharpoonup$ denotes weak convergence.
    Moreover, for $\alpha \in (0, 1)$ and for fixed $\eps > 0$,
    we have
    \begin{alignat}{2}
        &\OT_{\eps, \alpha}(\mu,\nu) \to \OT(\mu,\nu), 
        && \qquad \alpha \searrow 0, \label{eq:convAlphaZero}\\
        & \OT_{\eps, \alpha}(\mu, \nu)
        \to \OT_{\eps, \alpha'}(\mu, \nu),
        && \qquad \alpha \nearrow \alpha' \in (0, 1], \label{eq:convAlphaOne}
    \end{alignat}
    and
    \begin{alignat}{2} 
        \label{eq:planconvergencealphato0}
        \pi_c^{\alpha, \eps}(\mu, \nu)
        & \rightharpoonup
        \pi \in \Pi(\mu, \nu)
        \quad
        \text{with}
        \quad
        \langle c, \pi\rangle = \OT(\mu,\nu) 
        && \qquad \alpha \searrow 0, \\
        \label{eq:planconvergencealphatoalpha'}
        \pi_c^{\alpha, \eps}(\mu,\nu)
        & \rightharpoonup \pi_c^{\alpha', \eps}(\mu,\nu),
        && \qquad \alpha \nearrow \alpha' \in (0,1].
    \end{alignat}
\end{theorem}

\begin{figure}[t!]
    \centering
    \includegraphics[width=.5\textwidth]{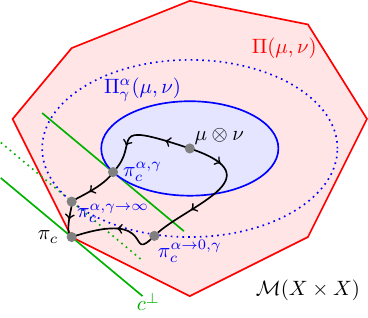}
    \caption{Transport polytope $\Pi(\mu,\nu)$ with the restricted transport polytope $\Pi_\gamma^\alpha(\mu,\nu)$.
    Convergence of $\pi_c^{\alpha,\gamma} \to \pi_c$ to the unregularized OT plan
    with $\alpha \to 0$ and $\gamma \to \infty$. 
    The set $\Pi_{\alpha}^{\gamma}$ is defined in \Cref{definition:RenyiBall}.
    (Plot inspired by~\cite{C13}.)}
\end{figure}

\begin{proof}
    We show each convergence result separately.
    \begin{enumerate}
        \item 
        For \eqref{eq:convLamInfinity}
        we follow the proof of~\cite[Prop.~4.1]{PC19}.
        Let $\eps_\ell \to 0$ for $\ell \to \infty$.
        Since $\Pi(\mu,\nu)$ is weakly compact 
        and by the uniqueness of the solution of \eqref{eq:dualSinkhornProblem}, 
        we have $\pi_c^{\alpha, \eps_\ell} \rightharpoonup \pi^*_c \in \Pi(\mu,\nu)$.
        Let $\pi_c \in \Pi(\mu,\nu)$ any OT plan realizing $\OT(\mu,\nu)$.
        We obtain 
        \begin{align}
            0
            & \le \langle c, \pi_c^{\alpha, \eps_\ell}\rangle - \langle c, \pi_c \rangle \notag \\
            & \le \eps_{\ell}\bigg(\RR_\alpha(\pi_c \mid \mu\otimes\nu) - \RR_\alpha(\pi_c^{\alpha, \eps_\ell} \mid \mu\otimes\nu)\bigg) \label{eq:ConLamInf} \\
            & \le \eps_{\ell}\RR_\alpha(\pi_c \mid \mu\otimes\nu) 
            \to 0, \qquad \ell \to \infty \notag
        \end{align}
        since $\RR_\alpha(\pi_c \mid \mu\otimes\nu)$ is independent of $\eps_\ell$.
        Together with the weak convergence $\pi_{c}^{\alpha, \eps_{\ell}} \rightharpoonup \pi_c^*$, 
        the weak lower semicontinuity of $c$ 
        we obtain
        \begin{equation*}
            \langle c, \pi_c \rangle
            = \lim_{\ell \to \infty} \langle c, \pi_c^{\alpha, \eps_\ell}\rangle 
            = \liminf_{\ell \to \infty} \langle c, \pi_c^{\alpha, \eps_\ell}\rangle 
           \ge \langle c, \pi_c^* \rangle
           \ge \langle c, \pi_c \rangle.
        \end{equation*}
        The optimality of $\pi_c$ for the unregularized problem yields $\langle c, \pi_c \rangle = \langle c, \pi_c^* \rangle$.
        
        Multiplying \eqref{eq:ConLamInf} by $\eps_\ell$ yields 
        \begin{align}
        \label{eq:boundRenyiConLam}
            \RR_\alpha(\pi_c^{\alpha, \eps_\ell}\mid \mu\otimes\nu) \le \RR_\alpha(\pi_c \mid \mu\otimes\nu)
            \quad\text{for all}
            \quad
            \ell \in \N.
        \end{align}
        In summary we have 
        \begin{align*}
            \langle c, \pi_c^{\alpha, \eps_\ell} \rangle
            + \eps_{\ell}
            \RR_\alpha(\pi_c^{\alpha, \eps_\ell} \mid \mu\otimes\nu)
            \to \OT(\mu,\nu), \qquad \ell \to \infty.
        \end{align*}
        Moreover, since $\RR_\alpha$ is weakly lower semicontinuous, we obtain 
        \begin{align*}
            \liminf_{\ell\to\infty} \RR_\alpha(\pi_c^{\alpha, \eps_\ell} \mid \mu\otimes\nu)
           \ge \RR_\alpha(\pi_c^* \mid \mu\otimes\nu).
        \end{align*}
        Together with taking the $\liminf_{\ell \to \infty}$ on both sides of \eqref{eq:boundRenyiConLam}, this yields 
        \begin{equation*}
            \RR_{\alpha}(\pi_c^* \mid \mu \otimes \nu)
            \le \liminf_{\ell \to \infty} \RR_{\alpha}(\pi_c^{\alpha, \eps_{\ell}} \mid \mu \otimes \nu)
            \le \RR_{\alpha}(\pi_c\mid \mu \otimes \nu).
        \end{equation*}
        Since this argument holds for any minimizer $\pi_c$ of $\OT(\mu, \nu)$, 
        we obtain that $\pi_c^*$ is a minimizer of 
        $\RR_{\alpha}(\,\cdot \mid \mu \otimes \nu)$ over the set of all such $\pi_c$.
        Finally,
        by the strict convexity of the Rényi divergence 
        on the transport polytope we conclude that
        \begin{equation*}
            \pi_c^{\alpha, \eps_\ell}
            \rightharpoonup \argmin\{ \RR_\alpha(\pi \mid \mu\otimes\nu): \pi \in \Pi(\mu,\nu), \langle c, \pi\rangle = \OT(\mu, \nu) \}, \qquad \ell \to \infty.
        \end{equation*}
        
        \item 
        The proof of \eqref{eq:convLamZero} 
        is similar to the one in~\cite[Prop.~5.3 (i)]{NS21}.
        For $\tilde\pi = \mu\otimes\nu \in \Pi(\mu,\nu)$,
        we have 
        \begin{align}
        \label{eq:RenyiConvLambdaToZero}
            \langle c,\tilde\pi\rangle 
            + \eps\RR_\alpha(\tilde\pi \mid \mu\otimes\nu)
            = \langle c,\mu\otimes\nu\rangle
        \end{align}
        for all $\eps > 0$ or equivalently 
        \begin{align}
        \label{eq:RenyiConvLambdaToZeroSec}
            \frac{1}{\eps} \langle c,\tilde\pi\rangle 
            + \RR_\alpha(\tilde\pi \mid \mu\otimes\nu)
            = \frac{1}{\eps}\langle c,\mu\otimes\nu\rangle.
        \end{align}
        From \eqref{eq:RenyiConvLambdaToZeroSec}, 
        we have 
        \begin{align*}
            \frac{1}{\eps} \OT_{\eps, \alpha}(\mu,\nu)
            & = 
            \min_{\pi\in\Pi(\mu,\nu)}
            \frac{1}{\eps} \langle c,\pi\rangle 
            + \RR_\alpha(\pi \mid \mu\otimes\nu)\\
            & = \frac{1}{\eps} \langle c,\pi_c^{\alpha, \eps}\rangle 
            + \RR_\alpha(\pi_c^{\alpha, \eps}\mid \mu\otimes\nu)
             \le 
            \frac{1}{\eps} \langle c,\mu\otimes\nu\rangle \\
            \Longrightarrow
            0 
            \leq 
            \RR_\alpha(\pi_c^{\alpha, \eps}\mid \mu\otimes\nu)
            & \le \frac{1}{\eps} \langle c, \mu\otimes\nu - \pi_c^{\alpha, \eps} \rangle
        \end{align*}
        for all $\eps > 0$.
        From this, we earn 
        \begin{align*}
            \limsup_{\eps \to \infty}
            \RR_\alpha(\pi_c^{\alpha, \eps}\mid \mu\otimes\nu)
             \le \lim_{\eps\to \infty} \frac{1}{\eps} \langle c, \mu\otimes\nu - \pi_c^{\alpha, \eps} \rangle
            = 0.
        \end{align*}
        As in the first case,
        we have the weak convergence of the OT plan 
        $\pi_c^{\alpha, \eps} \rightharpoonup \pi^*$
        for $\eps \to \infty$
        by the weak compactness of $\Pi(\mu,\nu)$.
        By the weak lower semicontinuity of the Rényi divergence we conclude 
        \begin{align*}
            0 \le \RR_\alpha(\pi^*\mid\mu\otimes\nu)
             \le \liminf_{\eps \to \infty}
            \RR_\alpha(\pi_c^{\alpha, \eps}\mid\mu\otimes\nu)
             \le \limsup_{\eps \to \infty}
            \RR_\alpha(\pi_c^{\alpha, \eps}\mid\mu\otimes\nu)
            \le 0
        \end{align*}
        so that
        $\pi_c^{\alpha, \eps} \rightharpoonup \pi^* = \mu\otimes\nu$.
        From \eqref{eq:RenyiConvLambdaToZero} we also have 
        \begin{equation*}
            \limsup_{\eps \to \infty}
            \OT_{\eps, \alpha}(\mu,\nu) 
            \le \langle c,\mu\otimes\nu \rangle.
        \end{equation*}
        Finally, 
        we obtain 
        \begin{align*}
            \langle c,\mu\otimes\nu\rangle
            & \ge 
            \limsup_{\eps \to \infty} 
            \OT_{\eps, \alpha}(\mu,\nu) \\
            & \ge 
            \liminf_{\eps \to \infty} \;
            \langle c, \pi_c^{\alpha, \eps}\rangle + \eps\underbrace{\RR_\alpha(\pi_c^{\alpha, \eps}\mid\mu\otimes\nu)}_{\ge 0}
            \ge 
            \liminf_{\eps \to \infty} \;
            \langle c, \pi_c^{\alpha, \eps}\rangle \\
            & \ge \langle c,\mu\otimes\nu\rangle,
        \end{align*}
        which yields the assertion.
        
        \item 
        Let $(\alpha_\ell)_{\ell \in \N} \subset (0,1)$,
        such that $\alpha_\ell \to 0$ for $\ell \to \infty$.
        Let $\pi_c^{\alpha_\ell,\eps}$ be the solution to $\OT_{\eps, \alpha_\ell}(\mu, \nu)$
        and $\pi_c \in \Pi(\mu,\nu)$ a solution of $\OT(\mu, \nu)$.
        As in the first case,
        we have the weak convergence of $\pi_c^{\alpha_\ell,\eps}$ 
        to some $\pi^* \in \Pi(\mu,\nu)$ for $\ell \to \infty$. 
        
        First,
        we show that 
        $\RR_{\alpha_\ell}(\pi_c^{\alpha_\ell, \eps} \mid \mu\otimes\nu) \to 0$ 
        for $\ell \to \infty$;
        second, we show that $\langle c, \pi_c^{\alpha_\ell, \varepsilon}\rangle \to \langle c, \pi_c\rangle$
        for $\ell \to \infty$.
        Combining both yields the assertion.
        
        By \Cref{thm:dual_formulation}, i.e. \eqref{eq:dual_pre_formulation},
        we have
        \begin{equation*}
            \frac{\diff \pi_c^{\alpha_\ell,\eps}}{\diff{\mu\otimes\nu}}
            = \frac{\alpha_\ell(1 - \alpha_\ell)^{-1}}{\left\langle \left|\frac{1}{\eps}\gamma_{\alpha_{\ell}, \hat{f} \oplus \hat{g}}\right|^{\frac{\alpha_\ell}{\alpha_\ell - 1}}, \mu \otimes\nu \right\rangle}
            \left|\frac{1}{\eps}\gamma_{\alpha_{\ell}, \hat{f} \oplus \hat{g}}\right|^{\frac{1}{\alpha_\ell - 1}},
        \end{equation*}
        where 
        $0 < \frac{1}{\eps} \gamma_{\alpha_{\ell}, \hat{f} \oplus \hat{g}} \in \C(X \times X)$.
        Utilizing \eqref{eq:rn_opt_plan}, 
        we can bound the Rényi divergence
        using the Radon-Nikodym derivative from \eqref{eq:rn_opt_plan}
        \begin{align}
            0 & \le \RR_{\alpha_\ell}(\pi_c^{\alpha_\ell, \eps} \mid \mu\otimes\nu) \notag\\
            & = \frac{1}{\alpha_\ell - 1}\ln\left(\int_{X \times X}\left(\frac{\diff \pi_c^{\alpha_\ell,\eps}}{\diff{\mu\otimes\nu}}\right)^{\alpha_\ell} \diff{\mu\otimes\nu}\right) \notag\\
            & = \frac{1}{\alpha_\ell - 1}
            \left\{\ln\left(\left(\frac{\alpha_\ell}{1 - \alpha_\ell}\right)^{\alpha_\ell}
            \int\limits_{X \times X} \frac{\frac{1}{\eps}\gamma_{\alpha_{\ell}, \hat{f} \oplus \hat{g}}^{\frac{\alpha_\ell}{\alpha_\ell - 1}}}{\left(\int_{X \times X} \frac{1}{\eps}\gamma_{\alpha_{\ell}, \hat{f} \oplus \hat{g}}^{\frac{\alpha_\ell}{\alpha_\ell - 1}} \diff{\mu\otimes\nu} \right)^{\alpha_\ell}}\diff {\mu\otimes\nu}\right)\right\} \notag\\
            & = \frac{1}{\alpha_\ell - 1}
            \left\{\alpha_\ell\ln\left(\frac{\alpha_\ell}{1 - \alpha_\ell}\right)
            + \ln\left(\left[\int_{X \times X} \frac{1}{\eps}\gamma_{\alpha_{\ell}, \hat{f} \oplus \hat{g}}^{\frac{\alpha_\ell}{\alpha_\ell - 1}}\diff{\mu\otimes\nu}\right]^{1 - \alpha_\ell}\right)\right\} \notag\\
            & = \frac{1}{\alpha_\ell - 1}
            \left\{\alpha_\ell(\ln(\alpha_\ell) - \ln(1 - \alpha_\ell))
            + (1 - \alpha_\ell)\ln\langle |\tfrac{1}{\eps}\gamma_{\alpha_{\ell}, \hat{f} \oplus \hat{g}}|^{\frac{\alpha_\ell}{\alpha_\ell - 1}},\mu\otimes\nu\rangle
            \right\} \notag\\
            & = \frac{1}{\alpha_\ell - 1}
            \Big[\alpha_\ell(\ln(\alpha_\ell) - \ln(1 - \alpha_\ell))\Big]
            - \ln\langle |\tfrac{1}{\eps}\gamma_{\alpha_{\ell}, \hat{f} \oplus \hat{g}}|^{\frac{\alpha_\ell}{\alpha_\ell - 1}},\mu\otimes\nu\rangle \notag\\
            & \le \underbrace{\frac{1}{\alpha_\ell - 1}
            \Big[\alpha_\ell(\ln(\alpha_\ell) - \ln(1 - \alpha_\ell))\Big]}_{\to 0, \ \ell \to \infty}
            + \underbrace{\frac{\alpha_\ell}{1 - \alpha_\ell}}_{\to 0, \ \ell \to \infty}
            \Big[\ln(\varepsilon^{-1}) 
            + \ln\underbrace{\langle \gamma_{\alpha_{\ell}, \hat{f} \oplus \hat{g}},\mu\otimes\nu\rangle}_{\le \tfrac{1}{\varepsilon}\OT(\mu, \nu) + \hat C_{\alpha, \varepsilon}})\Big]
            \label{eq:PreRenyialphaConvtoZero} \notag \\
            & \to 0 \quad \text{for } \ell \to \infty, 
        \end{align}
        where $0 \leq \hat C_{\alpha, \varepsilon} \to 0$
        for $\alpha \searrow 0$ for all $\varepsilon > 0$.
        We now prove the last bound 
        $\langle \gamma_{\alpha_{\ell}, \hat{f} \oplus \hat{g}},\mu\otimes\nu\rangle \le \frac{1}{\varepsilon}\OT(\mu, \nu) + \hat C_{\alpha, \varepsilon}$
        with the convergences properties of $0 \leq C_{\alpha, \varepsilon}$ in $\alpha \searrow 0$.
        By~\cite[Thm.~1.3]{V03}, the dual formulation of the Kantorovich problem reads as
        \begin{equation*}
            \OT(\mu,\nu) = \max_{\substack{(f,g) \in L^1(\mu)\times L^1(\nu) \\ f \oplus g \le c}}
            \langle f \oplus g,\mu \otimes \nu\rangle.
        \end{equation*}
        We rewrite the argument of the logarithm in \eqref{eq:PreRenyialphaConvtoZero}
        \begin{align*}
            \langle \gamma_{\alpha_{\ell}, \hat{f} \oplus \hat{g}},\mu\otimes\nu\rangle
            &= \tfrac{1}{\varepsilon}(\langle \hat f \oplus \hat g, \mu\otimes\nu\rangle - \OT_{\alpha_\ell,\varepsilon}(\mu, \nu) + C_{\alpha, \varepsilon})\\
            &\leq \tfrac{1}{\varepsilon}(\OT(\mu,\nu) - \OT_{0,\varepsilon}(\mu,\nu) + C_{\alpha, \varepsilon})\\
            &= \tfrac{1}{\varepsilon}(\OT(\mu,\nu) - \OT_{0,\varepsilon}(\mu,\nu)) - \tfrac{\alpha}{1-\alpha}\underbrace{\ln(\varepsilon)}_{\geq 1-\varepsilon^{-1}} - \tfrac{\alpha}{1-\alpha}(\ln(\tfrac{\alpha}{1-\alpha})-1)\\
            &\leq \tfrac{1}{\varepsilon}(\OT(\mu,\nu) - \OT_{0,\varepsilon}(\mu,\nu)) 
            + \underbrace{\tfrac{\alpha}{1-\alpha}
            (\varepsilon^{-1} - \ln(\tfrac{\alpha}{1-\alpha}))}_{=: \hat C_{\alpha, \varepsilon}}.
        \end{align*}
        Since $0 < \alpha/(1-\alpha) \searrow 0$ 
        for $\alpha \searrow 0$,
        and $0 < -\tfrac{\alpha}{1-\alpha}\ln(\tfrac{\alpha}{1-\alpha}) \searrow 0$
        for any $\tfrac{1}{2} > \alpha \searrow 0$
        we have $\hat C_{\alpha, \varepsilon} \searrow 0$.
        
        Let $\pi_c \in \Pi(\mu,\nu)$ be an optimal unregularized transport plan
        and $\pi_c^{\alpha_\ell,\eps}$ the OT plan of the regularized OT problem.
        Furthermore, 
        let $\kappa \in (0,1)$ and define $\tilde\pi_c \coloneqq (1-\kappa)\pi_c + \kappa\mu\otimes\nu$.
        We observe 
        \begin{align*}
            \langle c,\pi_c\rangle
            & \le \langle c,\pi_c^{\alpha_\ell,\eps}\rangle
            \le \langle c,\pi_c^{\alpha_\ell,\eps}\rangle 
            + \eps \RR_{\alpha_\ell}(\pi_c^{\alpha_\ell,\eps} \mid \mu\otimes\nu) \\
            & \le \langle c,\tilde\pi_c\rangle 
            + \eps \RR_{\alpha_\ell}(\tilde\pi_c \mid \mu\otimes\nu) 
            \to \langle c,\tilde\pi_c\rangle 
            + \eps \RR_0(\tilde\pi_c \mid \mu\otimes\nu), \qquad \ell\to\infty,
        \end{align*}
        i.e. by the lower semicontinuity and the weak convergence of $\pi_c^{\alpha_\ell,\eps}$
        we obtain 
        \begin{align*}
            \langle c,\pi_c\rangle
            \le \langle c,\pi_c^*\rangle 
            \le \langle c,\tilde\pi_c\rangle 
            + \eps \underbrace{\RR_0(\tilde\pi_c \mid \mu\otimes\nu)}_{= 0},
        \end{align*}
        since $\supp(\tilde\pi_c) = \supp(\mu\otimes\nu)$
        for all $\kappa > 0$.
        Let $\eps > 0$.
        We find $\kappa > 0$ such that 
        \begin{align*}
            \langle c,\tilde\pi_c\rangle 
            = (1-\kappa)\langle c,\pi_c\rangle + \kappa\langle c,\mu\otimes\nu\rangle 
            \le \langle c,\pi_c\rangle + \eps \\
            \Longleftrightarrow
            0 < \kappa \le \frac{\eps}{\langle c, \mu\otimes\nu - \pi_c\rangle}
        \end{align*}
        whenever $\langle c,\pi_c\rangle < \langle c,\mu\otimes\nu\rangle$,
        such that
        \begin{align*}
            \langle c,\pi_c\rangle
            \le \langle c,\pi_c^*\rangle 
            \le \langle c,\pi_c\rangle + \eps.
        \end{align*}
        If $\langle c,\pi_c\rangle = \langle c,\mu\otimes\nu\rangle$,
        then the statement is trivial.
        Together with the previous convergence result, this
        yields the assertion.

        \item 
        Let $\alpha' \in (0, 1]$.
        For any $\alpha \in (0, \alpha')$ let $\pi_c^{\alpha, \eps}, \pi_c^{\alpha', \eps} \in \Pi(\mu,\nu)$
        be the OT plans of the $\alpha$-Rényi and $\alpha'$-Rényi regularized 
        OT problem, respectively.
        By \Cref{lem:RenyiMonotone} 
        we have $\OT_{\eps, \alpha}(\mu,\nu) \le \OT_{\eps, \alpha'}(\mu,\nu)$.
        Now, let $(\alpha_\ell)_{\ell \in \N} \subset (0, \alpha')$ be a monotone sequence such that 
        $\alpha_\ell \nearrow \alpha'$ for $\ell \to \infty$.

        As in the first case, the weak compactness of $\Pi(\mu,\nu)$ implies that $\pi_c^{\alpha_\ell,\eps}$ converges weakly to some $\pi_c^{\alpha', \eps} \in \Pi(\mu,\nu)$.
        We obtain
        \begin{align*} 
            0 &
           \ge \liminf_{\ell \to \infty} \OT_{\eps, \alpha_\ell}(\mu,\nu) - \OT_{\eps, \alpha'}(\mu,\nu) \\
            & = \liminf_{\ell \to \infty} 
            \langle c,\pi_c^{\alpha_\ell,\eps}\rangle
            - \langle c,\pi_c^{\alpha', \eps}\rangle 
            + \eps \RR_{\alpha_\ell}(\pi_c^{\alpha_\ell,\eps} \mid \mu\otimes\nu)
            - \eps\RR_{\alpha'}(\pi_c^{\alpha', \eps} \mid \mu\otimes\nu) \\
            &\ge \liminf_{\ell \to \infty} 
            \left\{ \langle c,\pi_c^{\alpha_\ell,\eps}\rangle
            + \eps \RR_{\alpha_\ell}(\pi_c^{\alpha_\ell,\eps} \mid \mu\otimes\nu) \right\}
            - \langle c,\pi_c^{\alpha', \eps}\rangle - \eps\RR_{\alpha'}(\pi_c^{\alpha', \eps} \mid \mu\otimes\nu) \\
            &\ge \langle c, \pi_c^{\alpha', \eps}\rangle
            + \eps \RR_{\alpha'}(\pi_c^{\alpha', \eps} \mid \mu\otimes\nu)
            - \langle c, \pi_c^{\alpha', \eps} \rangle - \eps\RR_{\alpha'}(\pi_c^{\alpha', \eps} \mid \mu\otimes\nu) \\
            &\ge 0,
        \end{align*}
        where the last inequality is due to 
        \begin{equation*}
            \pi_{c}^{\alpha', \eps} 
            = 
            \argmin\left\{ \langle c, \pi \rangle + \eps \RR_{\alpha'}(\pi \mid \mu \otimes \nu) : \pi \in \Pi(\mu,\nu)\right\}.
        \end{equation*}
        Hence all inequalities are equalities and we obtain $\pi_c^{\alpha', \eps} = \pi_c^{\alpha', \eps}$ and thus \eqref{eq:planconvergencealphatoalpha'} as well as
        \begin{equation*}
            \liminf_{\ell \to \infty} \OT_{\eps, \alpha_{\ell}}(\mu, \nu)
            = \OT_{\eps, \alpha'}(\mu, \nu).
        \end{equation*}
        Furthermore, $\OT_{\eps, \alpha}(\mu,\nu) \le \OT_{\eps, \alpha'}(\mu,\nu)$ implies together with the previous line that
        \begin{equation*}
            \limsup_{\ell \to \infty} \OT_{\eps, \alpha_{\ell}}(\mu,\nu)
            \le \OT_{\eps, \alpha'}(\mu,\nu)
            = \liminf_{\ell \to \infty} \OT_{\eps, \alpha_{\ell}}(\mu, \nu),
        \end{equation*}
        that is, $\OT_{\eps, \alpha'}(\mu,\nu) = \lim_{\ell \to \infty} \OT_{\eps, \alpha_{\ell}}(\mu, \nu)$.
    \end{enumerate}
    All cases are validated, which yields the assertion.
\end{proof}

\section{A mirror descent algorithm for the primal problem}
 \label{sec:algo}

In this section,
we prove convergence of a mirror descent algorithm 
with a special step size to the solution of the primal problem \eqref{eq:dualSinkhornProblem}.
The proposed algorithm is summarized in \Cref{alg:MD}
and the step sizes are discussed in \Cref{prop:Polyak}, \Cref{cor:algMD} 
and in more detail in \Cref{remark:step_size}.

\textbf{Discretization.} \quad
First, let $X \subset \R^d$ be compact.
For $N \in \N$, we take a grid $(x_i)_{i = 1}^{N} \subset X$ 
and absolutely continuous probability measures $\mu,\nu \in \P(X)$.
We approximate $\mu$ by $\tilde{\mu}^{(N)} \coloneqq \sum_{k = 1}^{N} f(x_k) \delta_{x_k}$, where $f$ is the density of $\mu$ and then normalize the weights to obtain $\mu^{(N)} \coloneqq \sum_{k = 1}^{N} \frac{f(x_k)}{\sum_{j = 1}^{N} f(x_j)} \delta_{x_k} \eqqcolon \sum_{k = 1}^{N} a_k \delta_{x_k}$.
We apply the same procedure to $\nu$, obtaining $\nu^{(N)} \coloneqq \sum_{k = 1}^{N} b_k \delta_{x_k} \in \P(X)$.
Since both measures are discretized on the same grid, we simply consider their coefficients $\Vs \coloneqq (a_i)_{i=1}^{N}$ and $\Vt \coloneqq (b_i)_{i=1}^{N}$, respectively.
Furthermore, we denote the cost matrix by $\MM \coloneqq (c(x_i,x_j))_{i,j \in \{1,\ldots,N\}}$.

Hence, $\Vs,\Vt \in \Sigma_N$, 
$\MM \in \R_{+}^{N\times N}$ and the transport polytope becomes 
$U(\Vs,\Vt) \coloneqq \{\MP \in \Sigma_{N\times N} : \MP\1_N = \Vs, \MP^{\tT}\1_N = \Vt\}$,
where $\Sigma_{N\times N} \coloneqq \{ \MP \in \R_+^{N\times N}: \1_N^{\tT} \MP \1_N = 1 \}$ denotes the set of probability matrices 
and $\Sigma_N \coloneqq \{ \Vs \in \R_{\ge 0}^N: \1_N^{\tT} \Vs = 1 \}$ the probability simplex.

By $\odot$ and $\oslash$ we denote entrywise multiplication and division,
and by $\langle \cdot, \cdot \rangle_{\tF}$ resp. $\| \cdot \|_{\tF}$ the Frobenius inner product resp. norm on $\R^{N \times N}$.
The Euclidean inner product and norm on $\R^N$ 
are denoted by $\langle \cdot, \cdot\rangle$, 
respectively $\|\cdot\|_2$.
Furthermore, exponentiation of vectors and matrices is meant componentwise and their support, denoted by $\supp$, is the set of indices of non-zero entries.
Lastly, for matrices $\MP, \MQ \in \R^{N \times N}$, $\MP \ll \MQ$ means that $q_{i,j} = 0$ implies $p_{i, j} = 0$.

In this discretized setting, 
the Rényi and the Tsallis divergences become for $q, \alpha \in (0,1)$
\begin{equation} \label{eq:discreteRenyi}
    \RR_{\alpha}(\Vs \mid \Vt)
    = \frac{1}{\alpha - 1} \ln\left( \sum_{k = 1}^{N} s_k^{\alpha} t_k^{1 - \alpha} \right)
    \; \text{and} \;
    D_{f_q}(\Vs \mid \Vt)
    = \frac{1}{q - 1} \left( \sum_{k = 1}^{N} s_k^{q} t_k^{1 - q} - 1 \right),
\end{equation}
where $\ln(0) \coloneqq - \infty$,
and the Tsallis entropy becomes 
\begin{equation} 
    \TT_q(\Vs)
    = \frac{1}{q - 1} \left( 1 - \sum_{k = 1}^{N} s_k^{q} \right).
\end{equation}

We will use the mirror descent algorithm 
to solve the primal problem \eqref{eq:dualSinkhornProblem}, 
which in the discrete setting becomes
\begin{equation} \label{eq:discretePrimal}
    \argmin_{\MP \in U(\Vr, \Vc)}
    \quad
    \mathfrak{f}_{\MM, \Vr, \Vc}^{\alpha, \eps}(\MP),
\end{equation}
where
\begin{equation} \label{eq:mathfrakf}
    \mathfrak{f}_{\MM, \Vr, \Vc}^{\alpha, \eps}
    \colon U(\Vr, \Vc) \to [0, \infty), \qquad
    \MP \mapsto \langle \MM,\MP\rangle_F + \eps\RR_\alpha(\MP \mid \Vr\Vc^{\tT}).
\end{equation}

\begin{remark}[General mirror descent algorithm]
The update step of the general mirror descent algorithm~\cite{NY1983} 
to minimize a function $f \colon \R^n \to \R$
over a compact set $C \subset \R^n$ 
with entropy function $h \colon \R^n \to \R$
generating the Bregman divergence $D_h \colon \R^n \times \R^n \to [0, \infty)$, 
and step sizes $(\eta_k)_{k \in \N} \subset (0, \infty)$,
is the sequence starting at some $x^{(0)} \in C$ 
and then being defined iteratively by 
\begin{equation} \label{eq:MD}
    x^{(k)}
    = \argmin_{p \in C} D_h\left(p \mid
    (\nabla h)^{-1}\left( \nabla h(x^{(k - 1)}) - \eta_k \nabla f(x^{(k - 1)})\right)\right), \qquad k \in \N_{> 0}.
\end{equation}
\end{remark}
We choose $n = N^2$ and the Shannon entropy 
$h \colon [0, \infty)^{N \times N} \to \R$, $x \mapsto \sum_{i,j = 1}^{N} x_{i,j}\ln(x_{i,j}) - x_{i,j} + 1$ (and $h(x) = \infty$ else), 
which is $1$-strongly convex on $\Sigma_{N \times N}$ and generates the discrete KL divergence $D_h = \KL$,
where
\begin{equation*}
    \KL \colon \R_+^{N \times N} \times \R_+^{N \times N} \to [0, \infty), \qquad
    (\MP, \MQ) \mapsto \sum_{i, j = 1}^{N} p_{i, j} \ln\left(\frac{p_{i, j}}{q_{i, j}}\right)
    + q_{i, j} - p_{i, j},
\end{equation*}
with the conventions $0\cdot \ln(0) \coloneqq 0$ 
and $p \cdot \ln\left(\frac{p}{0}\right) \coloneqq \infty$
and $\KL(\MP, \MQ) = \infty$ else.

The gradient of $\mathfrak{f}_{\MM, \Vr, \Vc}^{\alpha, \eps}$ 
in \eqref{eq:mathfrakf} is given by 
\begin{equation*}
    \nabla \mathfrak{f}_{\MM, \Vr, \Vc}^{\alpha, \eps}(\MP)
    = \MM + \eps\frac{\alpha}{\alpha - 1}\frac{(\Vr\Vc^{\tT}\oslash\MP)^{1 - \alpha}}{\langle \MP^{\alpha}, (\Vr \Vc^{\tT})^{1 - \alpha} \rangle_F}, \qquad \MP \in U(\Vr, \Vc),
\end{equation*}
which is well-defined, since $\MP \in U(\Vr, \Vc)$ 
implies $\MP \ll \Vr \Vc^{\tT}$.

Hence for $f = \mathfrak{f}_{\MM, \Vr, \Vc}^{\alpha, \eps}$ 
and $C = U(\Vr, \Vc)$ the update \eqref{eq:MD} becomes
\begin{equation*}
    \MP^{(k)}
    = \argmin_{\MP \in U(\Vr, \Vc)} \KL\left(\MP \ \middle| \ \MP^{(k - 1)} \odot
    \exp\left( - \eta_k \MM - \eps \eta_k\frac{\alpha}{\alpha - 1}\frac{(\Vr\Vc^{\tT}\oslash\MP)^{1 - \alpha}}{\langle \MP^{\alpha}, (\Vr \Vc^{\tT})^{1 - \alpha} \rangle_F}\right)\right). 
\end{equation*}

For $\Vs,\Vt \in \Sigma_{N}$, the projection
\begin{equation*}
    \SK(\cdot, \Vs,\Vt) \colon \R_+^{N\times N} \to U(\Vs,\Vt), \qquad
    \MX \mapsto \argmin_{\MP \in U(\Vs,\Vt)} \KL(\MP \mid \MX),
\end{equation*}
can be computed using a variant of Sinkhorn algorithm,
where the marginals are $\Vs,\Vt \in \Sigma_N$ 
instead of $\1$'s~\cite{BCCNP15,LS2023}.

The following rather new result from \cite{YL22} generalizes the convergence of the mirror descent algorithm 
with a modified Polyak step size 
to the minimum of a only locally Lipschitz continuous function.
Importantly, to calculate this step size 
we do not need the optimal value of the objective
in contrast to the original Polyak step size~\cite{P69,P87}.

\begin{proposition}[{\cite[Thm.~4.1]{YL22}}]
\label{prop:Polyak}
    Let $f \in \Gamma_0(\R^n)$ 
    and $C \subset \R^n$ be a non-empty, closed, convex set.
    Define $K \coloneqq C \cap \intt(\dom(h))$.
    Assume that the following assumptions are fulfilled: 
    \begin{enumerate}
        \item
        The optimal value $f^* = \inf_{x \in C} f(x)$ is finite.
        
        \item
        The function $h$ is Legendre and 
        the set $C \cap \dom(f)$ contained in the closure of $\dom(h)$.
        
        \item
        The relative interior of $\dom(f)$ contains $K$.
        
        \item
        The function $h$ is strongly convex on $K$
        with respect to some norm $\|\cdot\|$, that is,
        \begin{equation*}
            D_h(x,y) \ge \frac{1}{2}\|x - y\|^2 \qquad \forall x,y \in K.
        \end{equation*}
        
        \item
        The subgradient of $f$ is bounded in any compact subset of $K$.
    \end{enumerate}
    Then, $\inf_{n \in \N} f(x_n) = f^*$,
    where $(x_n)_{n\in \N}$ is constructed as in~\cite[Alg.~3]{YL22}.
\end{proposition}
\begin{remark}
   The assumption 5 of \Cref{prop:Polyak}
    is equivalent to the local Lipschitz continuity of $f$.
\end{remark}

To guarantee the convergence of the mirror descent algorithm~\cite[Prop.~3.2, Thm.~4.1]{BT03}, 
we have to show local Lipschitz continuity of 
$\mathfrak{f}_{\MM, \Vr, \Vc}^{\alpha, \eps}$ on the restricted version of $U(\Vr,\Vc)$,
\begin{equation*}
    U(\Vr, \Vc)_{>0} \coloneqq \{\MP \in U(\Vr,\Vc) : p_{ij} > 0 \quad \text{for all} \quad (i,j) \in \supp(\Vr\Vc^{\tT})\}.
\end{equation*}
For any $\eps > 0$ we also define the convex set
\begin{equation*}
    U(\Vr, \Vc)_\eps \coloneqq \{\MP \in U(\Vr,\Vc) : 
    p_{ij}\ge \eps \quad \text{for all} \quad (i,j) \in \supp(\Vr\Vc^{\tT})\}.
\end{equation*}

\begin{lemma}
\label{lem:RenyiLipUeps}
    The functional \eqref{eq:mathfrakf} is Lipschitz continuous
    on $U(\Vr, \Vc)_\eps$ for all $\eps > 0$
    and locally Lipschitz on $U(\Vr,\Vc)_{>0}$,
    each with respect to $\| \cdot \|_{\tF}$.
\end{lemma}

\begin{proof}
    The linear part $U(\Vr, \Vc) \to \R$, $\MP \mapsto \langle \MM, \MP\rangle$ of 
    $\mathfrak{f}_{\MM, \Vr, \Vc}^{\alpha, \eps}$ is $\| \MM \|_{\tF}$-Lipschitz.
    We have 
    \begin{equation*}
        \mathfrak{f}_{\MM, \Vr, \Vc}^{\alpha, \eps}(\MP)
        = \langle \MM, \MP \rangle
        + \frac{\eps}{\alpha - 1}\ln(\II_{\alpha}(\MP \mid \Vr \Vc^{\tT})),
    \end{equation*}
    where $\II_{\alpha}(\;\cdot \mid \Vr \Vc^{\tT}) \colon U(\Vr, \Vc) \to [0, \infty)$, 
    $\MP \mapsto \sum_{i, j = 1}^{N} p_{i, j}^{\alpha} (r_i c_j)^{1 - \alpha}$ 
    is the $\alpha$-mutual information.
    
    Let $\MP \in U(\Vr, \Vc)$.
    Then we have $\MP \ll \Vr\Vc^{\tT}$, see \cite[p.~20]{PC19}, 
    so that $\II_{\alpha}(\;\cdot \mid \Vr \Vc^{\tT})$ is well-defined.
    We set $\beta \coloneqq \min\{ r_i c_j: i, j \in \{ 1, \ldots, N \}, \ r_i c_j > 0\} > 0$ to be the smallest nonzero element of $\Vr\Vc^{\tT}$.
    Then, by Jensen's inequality we have
    \begin{align*}
        \II_\alpha(\MP \mid \Vr\Vc^{\tT})
        = \sum_{i,j = 1}^{N} p_{i,j}^\alpha(r_ic_j)^{1 - \alpha}
        &\ge \sum_{i,j = 1}^{N} p_{i,j}r_ic_j \\
        & = \sum_{\substack{i, j = 1 \\ r_i c_j > 0}}^{N} p_{i,j}r_ic_j
        \ge \beta \sum_{\substack{i, j = 1 \\ r_i c_j > 0}}^{N} p_{i,j}
        = \beta \sum_{i, j = 1}^{N} p_{i, j}
        = \beta.
    \end{align*}
    For any $\beta > 0$, the function $\ln \colon [\beta, +\infty) \to \R$ is $\beta^{-1}$-Lipschitz
    and for any $\eps > 0$ and $\alpha \in (0, 1)$ the function $g_\alpha \colon [\eps, +\infty) \mapsto \R, x \mapsto x^\alpha$ is $\alpha \eps^{\alpha - 1}$-Lipschitz.
    Hence, the first assertion follows from the fact 
    that the sum and the composition of Lipschitz functions is Lipschitz.

    The local Lipschitz continuity on $U(\Vr,\Vc)_{>0}$
    follows from the fact that the objective is the composition of a Lipschitz function and the locally Lipschitz continuous function $\ln \colon (0,\infty) \to \R$.
\end{proof}

\begin{algorithm}[t]
\caption{Mirror descent algorithm to solve \eqref{eq:discretePrimal}---$\RR_\alpha$ROT.}
\KwData{marginals $\Vr, \Vc \in \Sigma_N$, distance matrix $\MM$, 
regularization parameter $\eps > 0$, order $\alpha \in (0,1)$.}
\KwResult{The solution $\MP^*$ of \eqref{eq:discretePrimal}}
$\MP^{(0)}\gets \Vr\Vc^{\tT}$
\Comment{Ensures $\MP^{(0)} \in U(\Vr, \Vc)$.} \\
\For{$k = 1,2,3,\ldots$}{
    Choose the step size $\eta_k$ according to \Cref{remark:step_size}.\\

    $\displaystyle \MP^{(k)}\gets \SK\left(\MP^{(k - 1)} \odot
    \exp\left( - \eta_k \MM - \eta_k \eps \frac{\alpha}{\alpha - 1}\frac{(\Vr\Vc^{\tT}\oslash\MP)^{1 - \alpha}}{\langle \MP^{\alpha}, (\Vr \Vc^{\tT})^{1 - \alpha} \rangle}\right), \Vr, \Vc\right)$ \\
}
\label{alg:MD}
\end{algorithm}

\begin{corollary}
\label{cor:algMD}
    \Cref{alg:MD}, with $(\eta_k)_{k \in \N}$ chosen as in~\cite[Alg.~3]{YL22},
    converges to the unique minimizer of \eqref{eq:discretePrimal}.
\end{corollary}
\begin{proof}
    To apply \Cref{prop:Polyak}, we choose $C = U(\Vr,\Vc) \subset \R^{N\times N}$, which is non-empty, closed and convex, $n \coloneqq N^2$, 
    $f = \mathfrak f_{\MM, \Vr,\Vc}^{\alpha, \eps}$
    and $h$ as the Shannon entropy,
    which is strongly convex with respect to the 1-norm
    and due to $\|x\|_1 \ge \|x\|_2$ for $x \in \R^n$,
    assumption 4 in \Cref{prop:Polyak} is fulfilled for both norms.
    
    Notice, that $\dom(f) \supset \R_{>0}^{N\times N}$,
    more precisely 
    \begin{align*}
        \dom(f) 
        & = \{\MX \in \R_{+}^{N\times N} : \supp(\MX) \cap \supp(\Vr\Vc^{\tT}) \neq \emptyset\},
    \end{align*}
    ensuring that the argument in the logarithm of the Rényi divergence is positive,
    hence we have $\overline{\dom(f)} = \R_+^{N\times N}$ such that 
    $\relint(\dom(f)) = (0, \infty)^{N\times N}$
    and $\intt(\dom(h)) = (0,\infty)^{N\times N}$.

    On the domain, $f$ is continuous, hence lower semicontinuous on $\R^{N\times N}$ using the extension $\overline{-\ln} \colon \R \to \overline{\R}$, $x \mapsto -\ln(x)$ if $x > 0$ and $+\infty$ otherwise,
    moreover convex and proper.
    
    We conclude with the well known properties of $h$ \cite{CS04}
    and the results from \Cref{lem:RenyiLipUeps}.
\end{proof}

\begin{remark}[Choice of step size] \label{remark:step_size}
    By \Cref{lem:RenyiLipUeps}, $\mathfrak{f}_{\MM,\Vr,\Vc}^{\eps, \alpha}$ 
    is locally Lipschitz continuous on 
    $K = U(\Vr,\Vc) \cap \intt(\dom h) = U(\Vr,\Vc)_{>0}$, so that
    for any compact subset of $K$ the restriction of $\nabla\mathfrak{f}_{\MM,\Vr,\Vc}^{\eps, \alpha}$ to that subset is bounded.
    By \Cref{cor:algMD} \Cref{alg:MD} converges when 
    using the generalized Polyak step size 
    \begin{equation*}
        \eta_k
        \coloneqq \frac{\mathfrak{f}_{\MM,\Vc,\Vr}^{\eps, \alpha}(\MP^{(k)}) - \hat{\mathfrak f}_k}{c\|\nabla \mathfrak{f}_{\MM,\Vc,\Vr}^{\eps, \alpha}(\MP^{(k)})\|_*^2},
    \end{equation*}
    where $\hat{\mathfrak{f}}_k$ is some minimal functional value 
    computed to the $k$-th iteration with some error $\delta_k$ and some fixed constant $c$
    and $\|\cdot\|_*$ denotes the dual norm of $\|\cdot\|$ for which assumption 5 in \Cref{prop:Polyak} is fulfilled,
    see~\cite[Alg.~3]{YL22} for more details.
    Notice that on $\R^n$ all norms are equivalent.
\end{remark}

\section{Numerical experiments} \label{sec:Num}

We investigate the regularized transport plans 
for synthetic marginals, illustrate the convergence results from \Cref{theorem:ConvToOTKL}, 
and then show the superiority over other regularizations 
in the real-world application of predicting voter migration.

Our numerical examples are performed in dimensions $d \in \{1,2\}$ 
with ground space $X = [0, 1]^d$, which is uniformly discretized.

\subsection{Validation of \texorpdfstring{\Cref{theorem:ConvToOTKL}}{Theorem 2.6}}

\begin{table}[b!]
\resizebox{\textwidth}{!}{
\begin{tabular}{c | c c c c}
    \toprule 
                & Rényi regularized OT                          & KL regularized OT                     & Tsallis regularized OT                & OT \\
    \midrule 
       distance & $\mathbf{\num{6.325e-2}}$                      & \num{7.588e-2}                        & \num{6.498e-2}                        & \num{6.316e-2} \\
       abs error $\pm$ std & $\mathbf{\num{2.618e-4} \pm \num{1.781e-3}}$ & \num{6.992e-4} $\pm$ \num{3.855e-3}   & \num{5.432e-4} $\pm$ \num{3.330e-3}   & -\\
       KL error & $\mathbf{\num{5.901e - 1}}$                     & \num{2.253}                           & \num{1.212} & -\\
       MS error & $\mathbf{\num{8.103e-3}}$                     & \num{3.837e-2}                        & \num{2.846e-2} & -\\
    \midrule
       distance & $\mathbf{\num{7.41e-3}}$ & \num{3.822e-2} & \num{8.0e-3} & \num{6.26e-3} \\
       abs error $\pm$ std & $\mathbf{\num{1.382e-4} \pm \num{6.767e-4}}$ & $\num{7.134e-4} \pm \num{2.448e-3}$ & $\num{4.827e-4} \pm \num{1.938e-3}$ & -\\
       KL error & \textbf{\num{3.594e - 1}} & \num{2.356e0} & \num{9.943e - 1} & -\\
       MS error & \textbf{\num{1.193e-3}} & \num{1.625e-2} & \num{9.968e-3} & -\\
    \bottomrule
\end{tabular}}
\caption{\label{tab:1}
Comparing the Rényi ($\RR_\alpha$), KL and Tsallis ($\TT_q$) 
regularized optimal transport plans 
to the unregularized transport plan
in terms of the absolute mean and standard deviation,
KL and mean squared error.
We use the (unscaled) squared distance matrix,
different from the ones used in~\cite{MNPN17} or \Cref{tab:2}.
\textit{Top row:} Gaussian distributed marginals
with $\alpha = 0.01, q = 1.6$ as good as possible;
\textit{Bottom row}: mixed Poisson-distributed marginals
with $\alpha = 0.01, q = 2$ as good as possible,
visualized in \Cref{fig:renyi/tsallis_ot_lam=10}.
In both cases $\eps = \num{e-1}$ and marginals in $\Sigma_{50}$.}
\end{table}

\begin{figure}[b!]
\resizebox{\linewidth}{!}{
\begin{tabular}{@{\hspace{1.4cm}} c @{\hspace{2.4cm}} c @{\hspace{2.1cm}} c @{\hspace{2.3cm}} c @{\hspace{1.2cm}} c}
    KL & R$_{\alpha}$OT & T$_{q}$OT & OT & |R$_{\alpha}$OT - OT| \\
    \multicolumn{5}{c}{\includegraphics[width=1.2\linewidth, clip=true, trim=0pt 10pt 0pt 10pt]{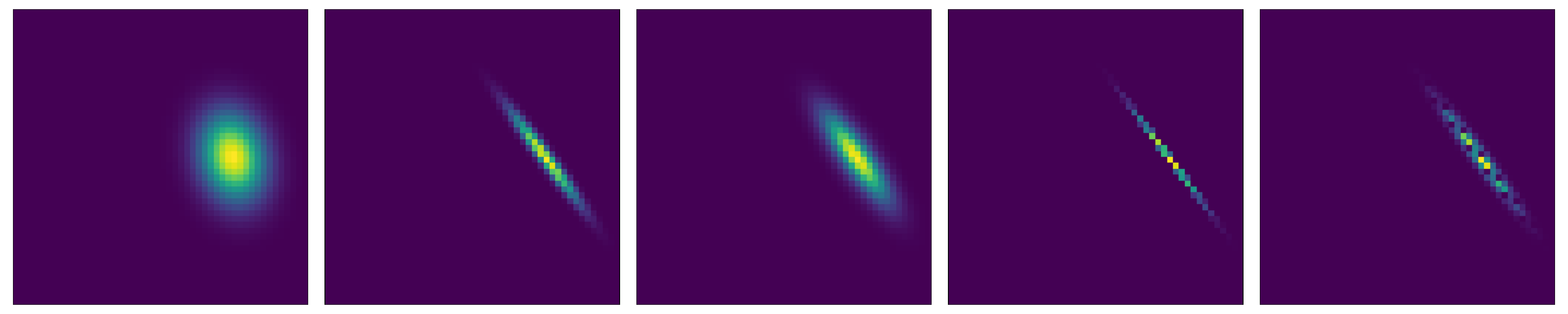}} \\[2pt]
    \multicolumn{5}{c}{\includegraphics[width=1.2\linewidth, clip=true, trim=0pt 10pt 0pt 10pt]{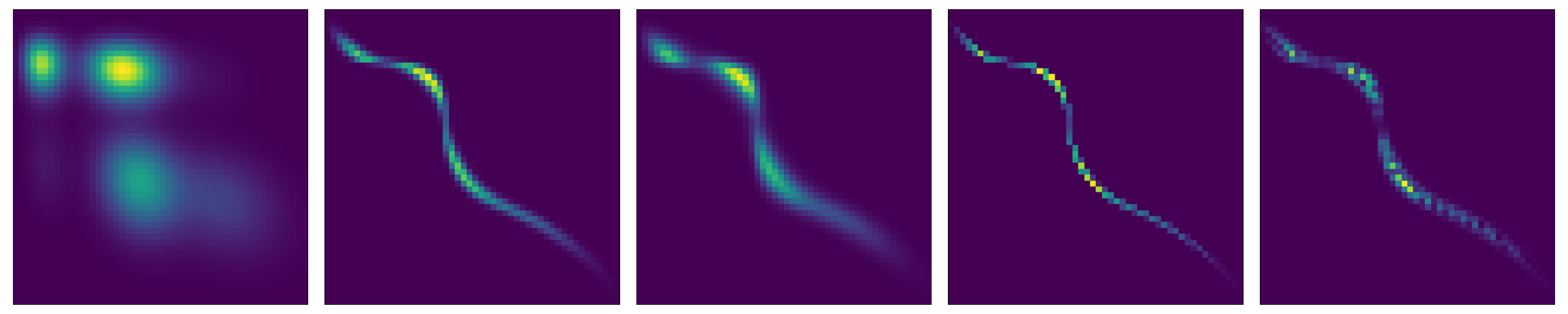}}
\end{tabular}}
\caption{
Transport plans belonging to two choices of marginals compared in \Cref{tab:1} 
and the absolute difference between the Rényi regularized 
and the unregularized plan (rightmost plot).
}
\label{fig:renyi/tsallis_ot_lam=10}
\end{figure}

We first numerically validate the convergence statements from \Cref{theorem:ConvToOTKL}.
In \Cref{fig:renyi_ot_comp_poisson_convergence} we observe significant improvements 
in comparison to the state of the art
KL regularized OT problem~\cite{C13}.
In \Cref{fig:renyi_ot_emd_comp_poisson_convergence}
a qualitative study is done 
through phase transitions 
on the difference of the optimal transport plans
between Rényi regularized 
and unregularized OT (left)
and between Rényi regularized 
and KL OT (right).
We observe a faster convergence in $\alpha$
than in $\varepsilon$ to the unregularized optimal transport plan (left),
and relatively large differences 
for the Rényi regularized and KL regularized optimal transport plans 
even for small regularization parameters.
For the comparison 
the setting from \Cref{fig:renyi_ot_comp_poisson_convergence}
is used.
Additionally, 
we compare Rényi and Tsallis regularization in \Cref{tab:1} and \Cref{tab:2}.

\begin{figure}[t!]
    \includegraphics[width=\linewidth, clip=true, trim=70pt 75pt 70pt 70pt]{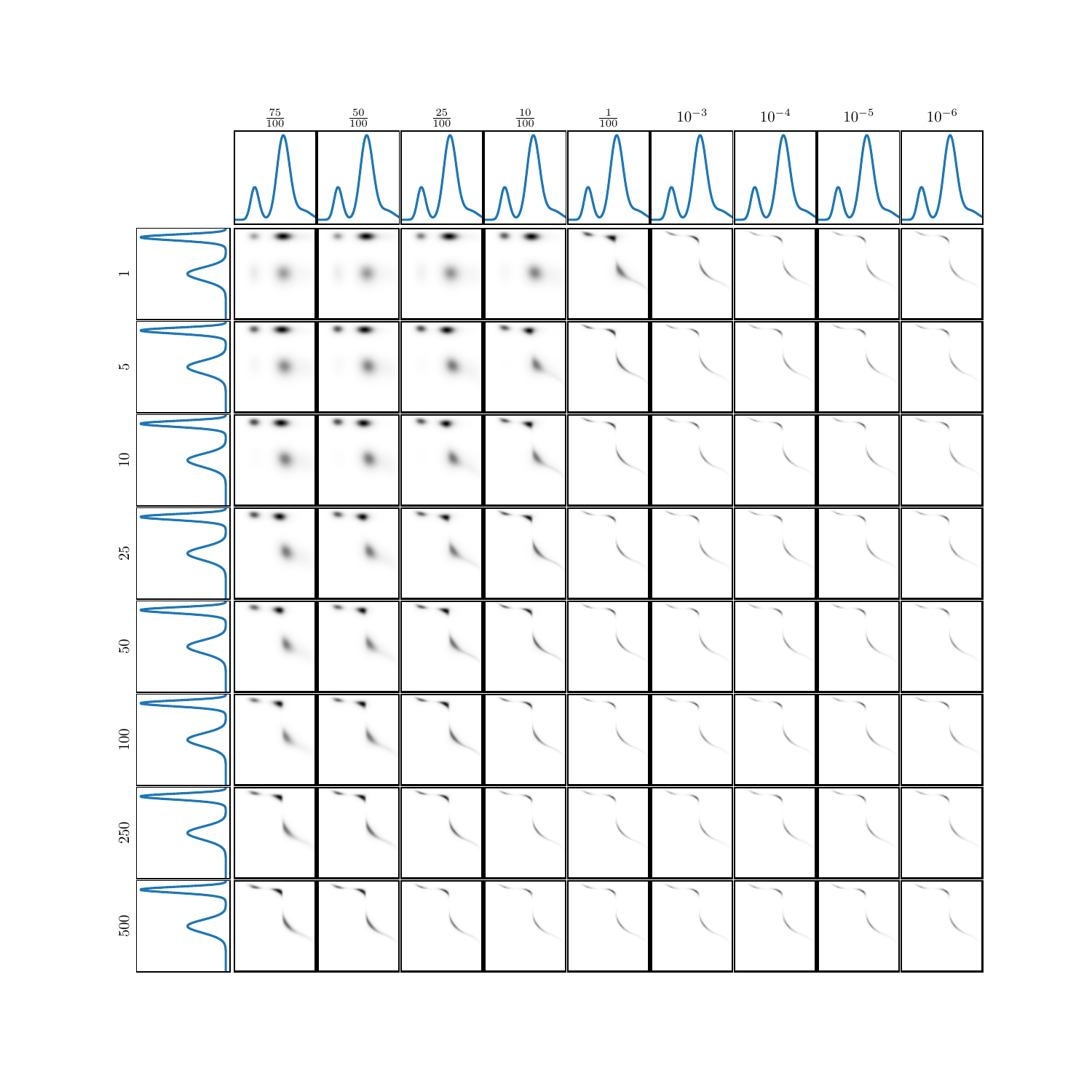}
    \vspace{-0.6cm}
    \caption{
    Summary of the convergence properties in $\alpha \to \{0,1\}$
    (columnwise)
    and $\eps^{-1} \to \{0, + \infty\}$ (rowwise) from \Cref{theorem:ConvToOTKL}.
    The unregularized OT plan is expected to be in the last columns
    and in the last row,
    where the trivial coupling is expected to be in the upper left corner
    (visualized by the open source code from~\cite{MNPN17}).}
    \label{fig:renyi_ot_comp_poisson_convergence}
\end{figure}

\begin{figure}[t!]
    \includegraphics[width=0.47\linewidth, clip=true, trim=35pt 25pt 47pt 40pt]{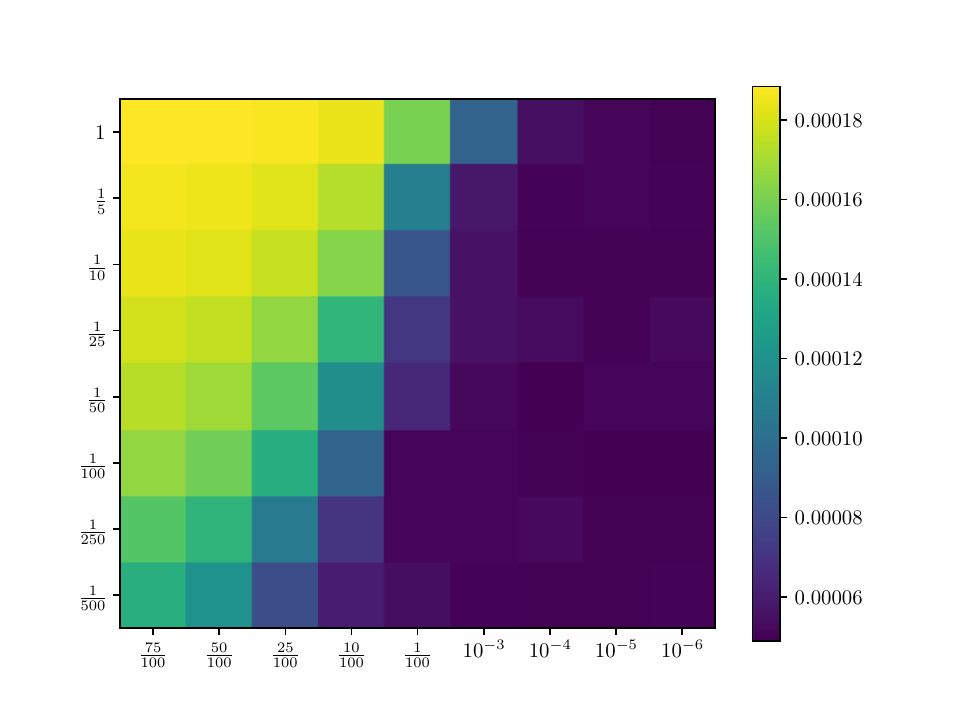}
    \hfill
    \includegraphics[width=0.47\linewidth, clip=true, trim=35pt 25pt 47pt 40pt]{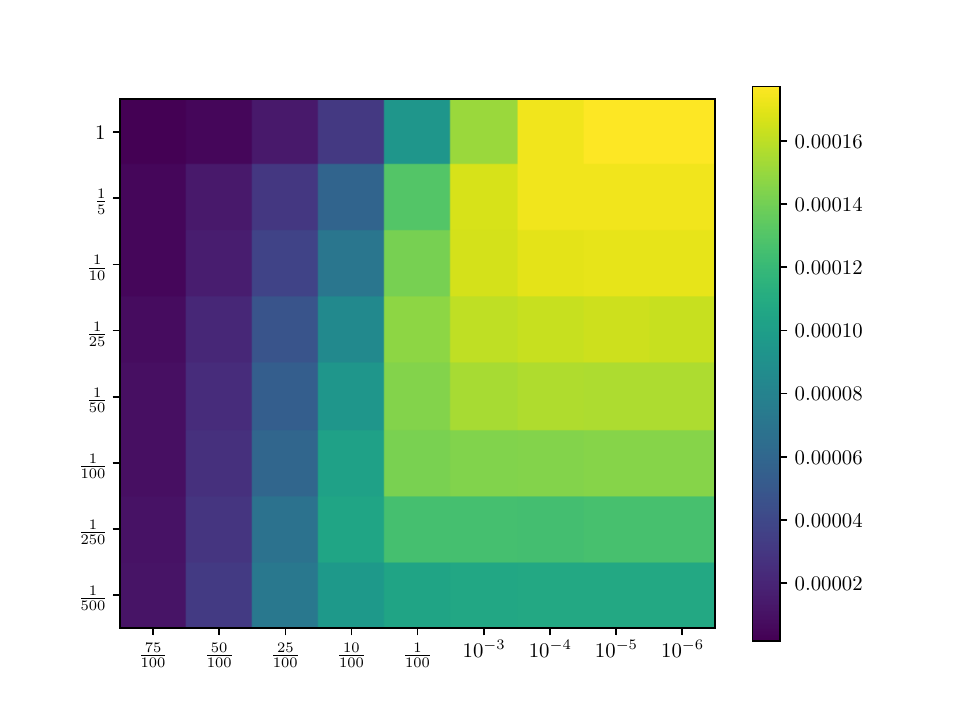}
    \caption{
    Absolute mean of the difference of the Rényi regularized transport plans to the unregularized OT plan (left) and to the KL regularized plans (right)
     with $\alpha$ on the $x$-axis (cf.~\Cref{fig:renyi_ot_comp_poisson_convergence})
     and 
     $\eps$ on the $y$-axis.}
    \label{fig:renyi_ot_emd_comp_poisson_convergence}
\end{figure}

In \Cref{tab:1}, 
respectively \Cref{fig:renyi/tsallis_ot_lam=10},
we use the (unscaled) squared Euclidean distance matrix $\MM$~\cite{FLA2021} for our OT-problem,
i.e. $m_{i,j} \coloneqq (\frac{i- j}{n - 1})^2$ for $0 \le i,j \le n - 1$, 
where $n$ is the length of the marginals.
However, we observe that the Tsallis-regularization from~\cite{MNPN17}
works better with the scaled version of $\MM$, 
i.e. $\eta_n \MM$, where $\eta_n \coloneqq \frac{(n - 1)^2}{n}$.
Our program is much more stable with the unscaled distance matrix; 
both problems have an objective in the form $\langle \MM, \cdot \rangle + \eps R(\cdot)$ 
and translate into each other as follows
\begin{align}
    \argmin_{\MP \in U(\Vr, \Vc)} 
    \quad 
    \langle \MM, \MP \rangle + \eps R(\MP) 
    & = \argmin_{\MP \in U(\Vr, \Vc)}
    \quad
    \langle \eta_d\MM, \MP \rangle + \eta_d \eps R(\MP) \\
    \text{ and }
    \qquad
    \argmin_{\MP \in U(\Vr, \Vc)} 
    \quad 
    \langle \eta_d\MM, \MP \rangle + \eps R(\MP) 
    & = \argmin_{\MP \in U(\Vr, \Vc)}
    \quad
    \langle \MM, \MP \rangle + \frac{\eps}{\eta_d}R(\MP),
\end{align}
where $R(\cdot)$ is a convex, proper, lower semicontinuous regularizer
and $\eps > 0$ the regularization parameter.
Given the scaled distance matrix,
we observe that the Tsallis regularized OT plan is not significantly tighter 
than the KL regularized OT plan,
whereas our regularization is much more tighter than both,
see \Cref{tab:2}. 
\begin{table}[t!]
\resizebox{\textwidth}{!}{
\begin{tabular}{c | c c c c}
    \toprule 
                                & Rényi regularized OT                          & Kullback-Leibler regularized OT       & Tsallis regularized OT                & OT \\
    \midrule 
       distance                 & \textbf{\num{3.03448}}                        & \num{3.07815}                         & \num{3.07420}                         & \num{3.03313} \\
       abs error $\pm$ std      & $\mathbf{\num{3.481e-4} \pm \num{2.496e-3}}$  & \num{4.596e-4} $\pm$ \num{3.038e-3}   & \num{4.596e-4} $\pm$ \num{3.038e-3}   & - \\
       KL error                 & \textbf{\num{7.059e-1}}                       & 9.199                                 & 9.199                                 & - \\
       MS error                 & \textbf{\num{1.588e-2}}                       & \num{2.360e-2}                        & \num{2.361e-2}                        & - \\
    \midrule
       distance                 & $\mathbf{0.32401}$                            & 0.34401                               & 0.34364                               & 0.30092 \\
       abs error $\pm$ std      & $\mathbf{\num{1.321e-4} \pm \num{6.453e-4}}$  & \num{4.069e-4} $\pm$ \num{1.759e-3}   & \num{4.070e-4} $\pm$ \num{1.759e-3}   & - \\
       KL error                 & $\mathbf{\num{1.640e-1}}$                     & \num{7.506e-1}                        & \num{7.507e-1}                        & - \\
       MS error                 & $\mathbf{\num{1.193e-3}}$                     & \num{1.625e-2}                        & \num{9.968e-3}                        & - \\
    \bottomrule
\end{tabular}}
\caption{\label{tab:2}
Comparing the Rényi ($\RR_\alpha$), KL and Tsallis ($\TT_q$) 
regularized optimal transport plans 
to the unregularized transport plan like in \Cref{tab:1}.
However, here we use the scaled, squared Euclidean distance matrix
as in~\cite{MNPN17,FLA2021}.
\textit{Top row}: Gaussian distributed marginals
with $\alpha = 0.25$ and $q = 1$;
\textit{Bottom row}: mixed Poisson-distributed marginals
with $\alpha = 0.25$ and $q = 1$ 
as good as possible.
In both cases $\eps = \num{e-1}$ and marginals in $\Sigma_{50}$.
}
\end{table}
%
%
Independent of the choice of the squared Euclidean distance matrix 
(scaled or unscaled),
we observe significant improvements of the regularization 
by the Rényi divergence 
instead of the Tsallis or KL~$=T_1 = \RR_1$ divergence, 
see \Cref{fig:renyi_ot_comp_poisson_I} for a qualitative comparison.

For a quantitative comparison concerning the convergence behavior,
we use the Poisson distributed marginals as in \Cref{tab:1} and \Cref{tab:2}, 
now in $\Sigma_{100}$ instead of $\Sigma_{50}$.
For fixed $\eps > 0$, we observe that the functional values of $\OT_{\eps, \alpha}$ convergence for $\alpha \to 0$.
Notably, choosing $\eps \ll 1$ additionally improves the convergence behavior 
of $\OT_{\eps, \alpha}$ in $\alpha$, 
and the unregularized OT value is already achieved for larger $\alpha \sim 10^{-3}$,
where $\alpha \sim \num{e-6}$ is necessary when choosing $\eps$ large,
we refer to \Cref{fig:comp_renyi_alpha}.
For a visualization of the corresponding optimal regularized transport plan,
we refer to \Cref{fig:renyi_ot_comp_poisson_convergence}.
Generally, we observe that the computation time decreases in $\alpha \in (0,1)$,
for fixed $\eps > 0$.
We claim, 
that choosing $\eps$ large, which is not a suitable choice 
for the KL regularized OT,
supports the computational expense of the Sinkhorn projection, 
and affects this behavior.
Notably,
we can also decrease $\eps$ for fixed $\alpha$,
observing the convergence to the unregularized OT
without any computational difficulties;
however, the computation time is now significantly higher,
cf.~\Cref{tab:runtimes}.
We highline 
that the computation 
for small $\alpha$ and $\varepsilon$
might be higher,
but leads us to sharper transport plans.
Additionally,
choosing $\alpha$ small 
and $\varepsilon$ large
enables similar transport plans,
better than those with small $\varepsilon$
and $\alpha$'s close to one
(i.e. close to KL regularized OT),
and especially with a smaller computation time,
cf~\Cref{tab:runtimes} (bold entry).
The second row and fourth column 
corresponds to \Cref{fig:renyi_ot_emd_comp_poisson_convergence}.
Choosing $\eps$ in the same range 
for the stabilized OT~\cite{S19}
leads to numerical instabilities,
while our algorithms still converge for these choices.

\begin{table}[t!]
   \begin{tabular}{c|c c c c c}
   \toprule
    \backslashbox{$\eps$}{$\alpha$} & \num{9.99e-1} & \num{9.9e-1}    & \num{9e-1} & \num{1e-6} & \num{1e-7} \\
    \midrule
    \num{1e1}         & <0.1 & <0.1 & 3 & \textbf{30} & \textbf{30} \\
    \num{1e0}         & <0.1 & <1 & 4 & 31 & 32 \\
    \num{1e-1}        & <0.1 & <1 & 5 & 34 & 35 \\
    \num{1e-6}        & 50 & 40 & 40 & 38 & 40  \\
    \num{1e-7}        & 66 & 48 & 55 & 57 & 60 \\
    \bottomrule
    \end{tabular}
    \caption{Rounded run times in seconds, of R$_\alpha$OT for different values of $\alpha \in (0, 1)$ and $\eps > 0$.}
    \label{tab:runtimes}
\end{table}

\begin{figure}[t!]
    \includegraphics[width=.49\linewidth, clip=true, trim=40pt 20pt 65pt 40pt]{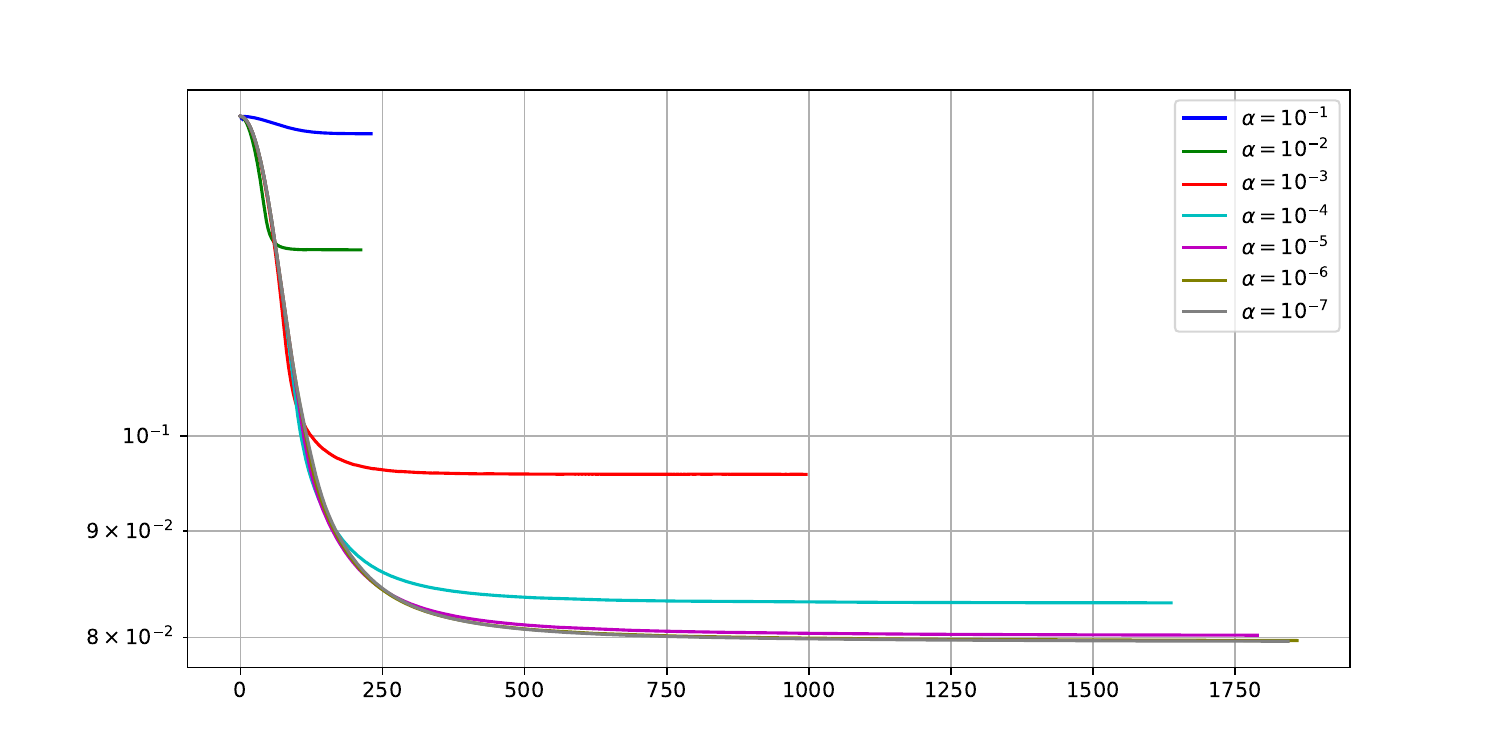}
    \includegraphics[width=.49\linewidth, clip=true, trim=40pt 20pt 65pt 40pt]{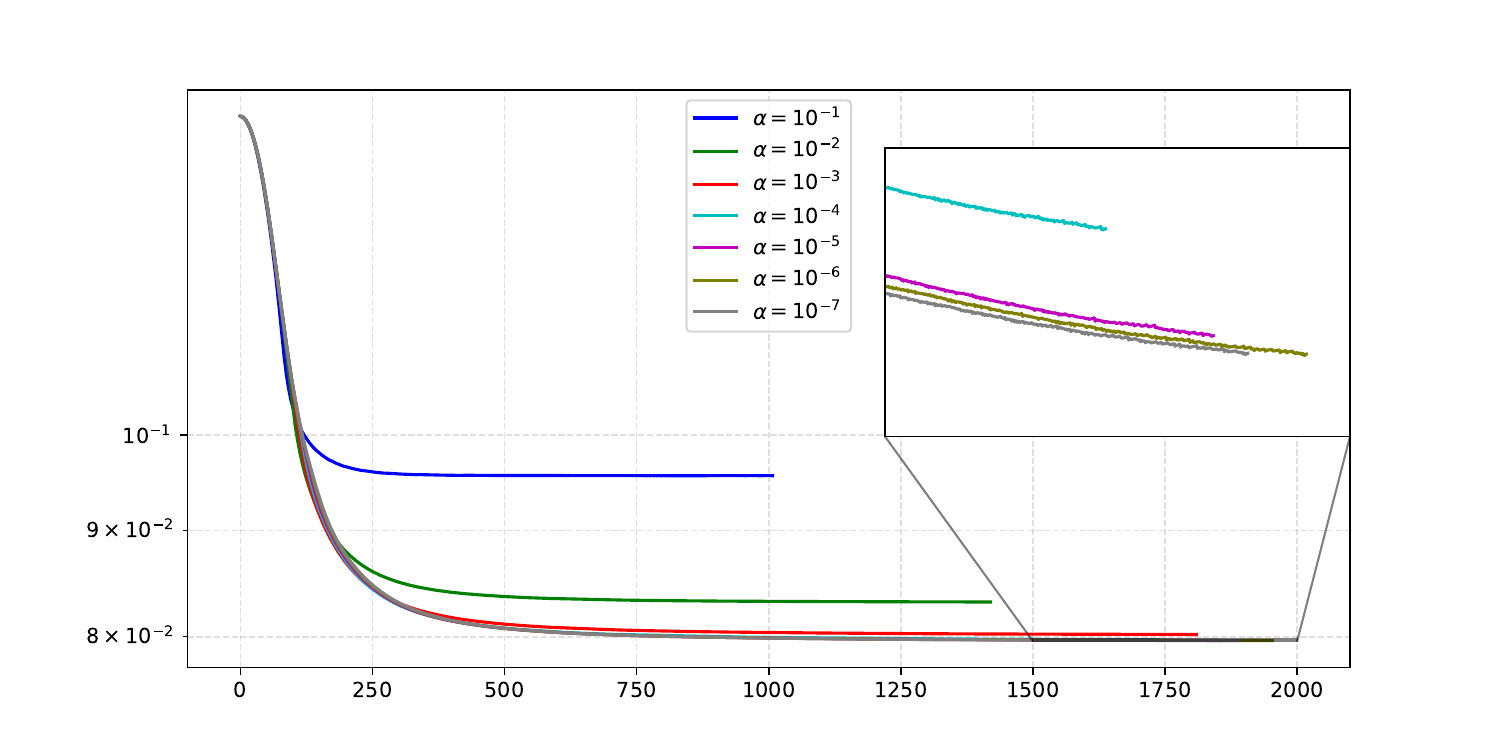}
    \caption{
    Convergence behavior of $\OT_{\eps, \alpha}(\mu,\nu)$:
    \textit{Left:} for fixed $\eps = \num{e1}$;
    \textit{Right:} for fixed $\eps = \num{e-1}$;
    using the mixed Poisson distributed marginals 
    as in \Cref{tab:1} and \Cref{tab:2}, here in $\Sigma_{100}$.
    The subplot is valued between $\num{7.95e-2}$ and $\num{7.97e-2}$.
    }
    \label{fig:comp_renyi_alpha}
\end{figure}

\begin{figure}[t]
    \begin{minipage}{0.5\linewidth}
    \resizebox{\linewidth}{!}{%
    \begin{tabular}{c c @{\hspace{87pt}} c @{\hspace{57pt}} c @{\hspace{57pt}} c @{\hspace{37pt}} c}
        &  & $1$ & $10$ & $25$ & $50$  \\
        & \multicolumn{5}{l}{\multirow{5}{*}{\includegraphics[width=1.75\textwidth, clip=true, trim=72pt 55pt 55pt 65pt]{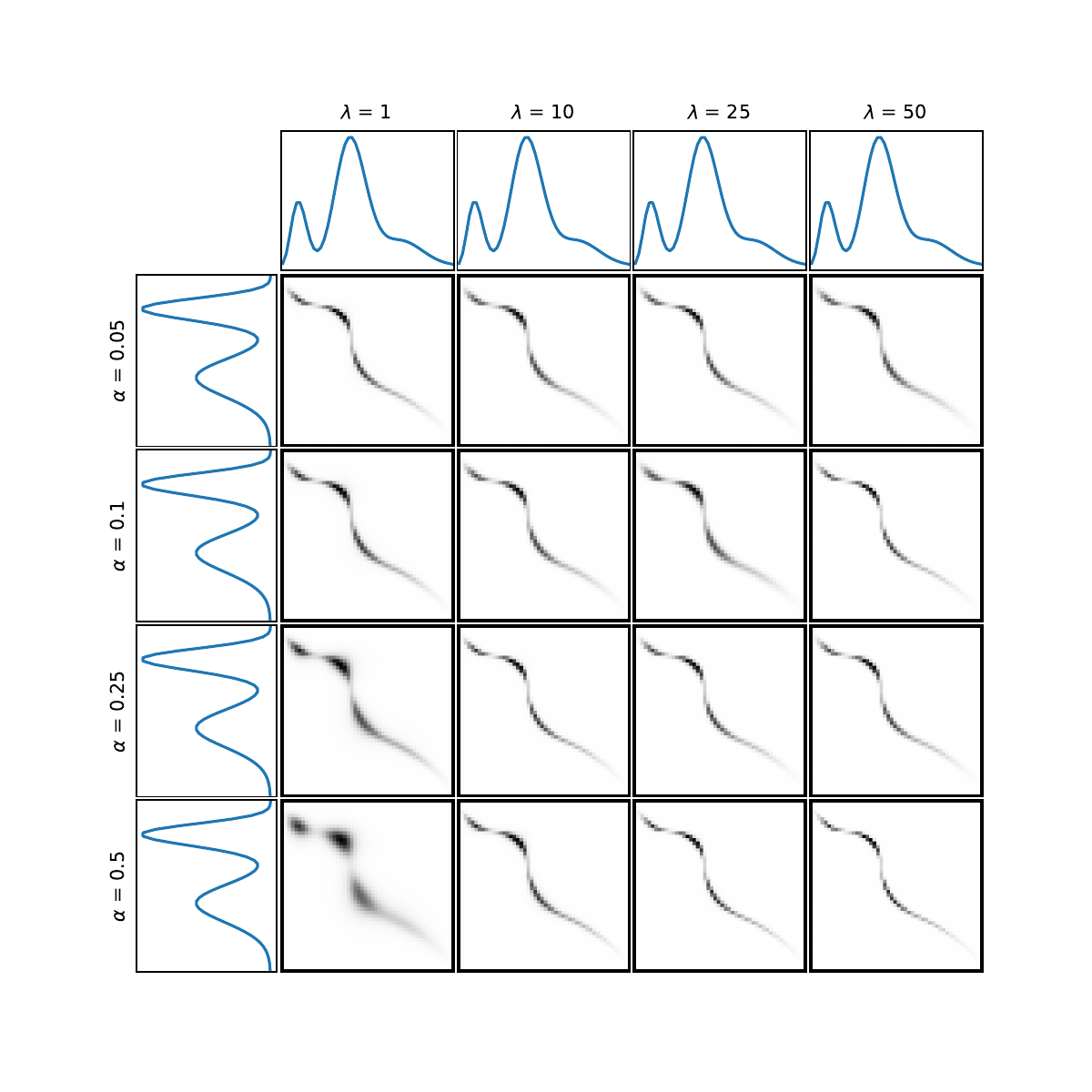}}} \\[70pt]
        \rotatebox[origin=c]{90}{$0.05$} \\[50pt]
        \rotatebox[origin=c]{90}{$0.1$} \\[50pt]
        \rotatebox[origin=c]{90}{$0.25$} \\[55pt]
        \rotatebox[origin=c]{90}{$0.5$} \\[30pt]
    \end{tabular}}
    \end{minipage}%
    \hfill
    \begin{minipage}{0.5\linewidth}
    \resizebox{\linewidth}{!}{%
    \begin{tabular}{c c @{\hspace{87pt}} c @{\hspace{57pt}} c @{\hspace{57pt}} c @{\hspace{37pt}} c}
        &  & $1$ & $10$ & $25$ & $50$ \\
        & \multicolumn{5}{l}{\multirow{5}{*}{\includegraphics[width=1.75\textwidth, clip=true, trim=72pt 55pt 55pt 68pt]{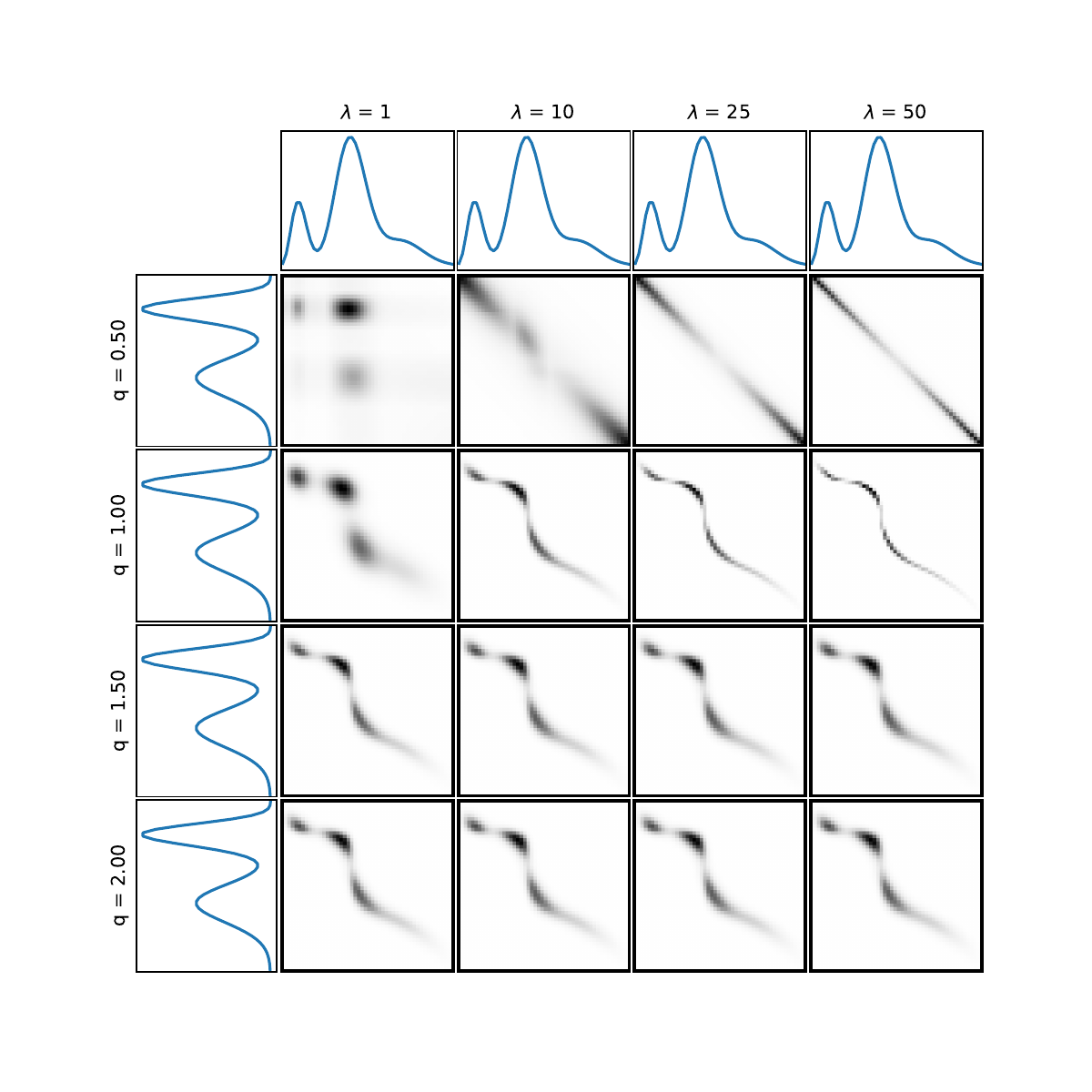}}} \\[70pt]
        \rotatebox[origin=c]{90}{$0.5$} \\[50pt]
        \rotatebox[origin=c]{90}{$1.0$} \\[50pt]
        \rotatebox[origin=c]{90}{$1.5$} \\[55pt]
        \rotatebox[origin=c]{90}{$2.0$} \\[30pt]
    \end{tabular}}
    \end{minipage} \\
    \caption{Similar to~\cite[Fig.~2]{MNPN17},
    we compare the Rényi and Tsallis regularized OT plans 
    for mixed Poisson marginals
    where the distance matrix $\MM$ is the squared Euclidean scaled by a factor of $\frac{(n - 1)^2}{n}$,
    and $n$ is the length of the marginals.
    In both cases,
    the regularization parameter,
    here $\varepsilon^{-1}$,
    is reported on the $x$-axis.
    The choice of $\alpha$ is reported on the $y$-axis (left hand side).
    We copy the choice of $q$ from~\cite[Fig.~2]{MNPN17} 
    ($y$-axis on the right hand side)
    and remark that for $q$ closer to zero, 
    the Tsallis regularized plan deteriorate further.
    }
    \label{fig:renyi_ot_comp_poisson_I}
\end{figure}

\begin{remark}[Simpler step sizes]
    Numerically, we also observe convergence for fixed 
     - and importantly constant in $k$ - step sizes,
    reducing the computational expense due to not having to compute the Polyak step size in each iteration.
    Theoretically,
    there is no guarantee for the convergence
    since the objective is not Lipschitz continuous in zero.
    Heuristically,
    since we stop the algorithm after finitely many iterations
    and since we can observe that the gradient is bounded and not out of memory,
    the convergence can be still observed,
    not only in a neighbourhood of the actual minimizer.
\end{remark}

\subsection{\texorpdfstring{R$_\alpha$}{R-alpha}OT for Measures with Disjoint Support}

Optimal transport also allows the comparison of measures 
with disjoint support, 
whereas such a comparison is not possible
using $f$-divergences, for instance.

We show that the Rényi-regularized OT allows
us to compare measures with disjoint support 
and contrast properties the corresponding transport plans.
Therefore,
we use two empirical measures%
---the first was also studied in \cite{FSVATP19}
in terms for optimal transport flows.
For applying Algorithm~\ref{alg:MD},
we have to take care of the dimensions: due to using KL projections,
the marginals have to lie in the same probability simplex.
Therefore, 
we add zeros and perturb those values with a small error
to avoid numerical instabilities.

\begin{figure}[t]
    \resizebox{\linewidth}{!}{%
    \begin{tabular}{c c c c}
        unregularized OT 
        & KL-unregularized OT
        & \multicolumn{2}{c}{Rényi-unregularized OT}\\
        \includegraphics[width=0.34\linewidth,clip=true, trim=10pt 10pt 40pt 10pt]{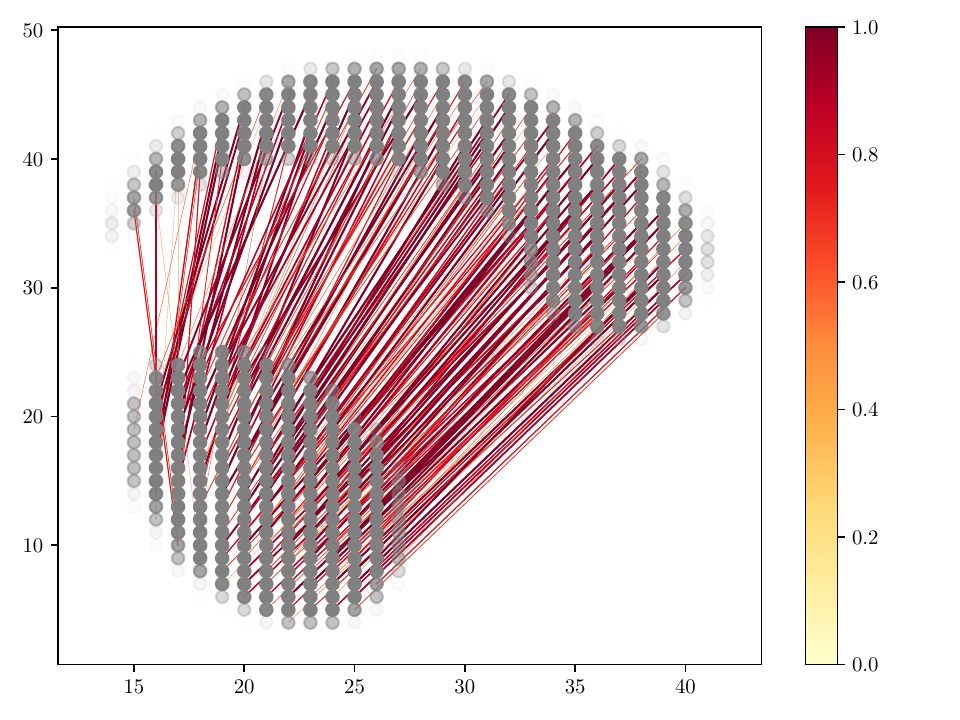}
        &\includegraphics[width=0.34\linewidth,clip=true, trim=10pt 10pt 40pt 10pt]{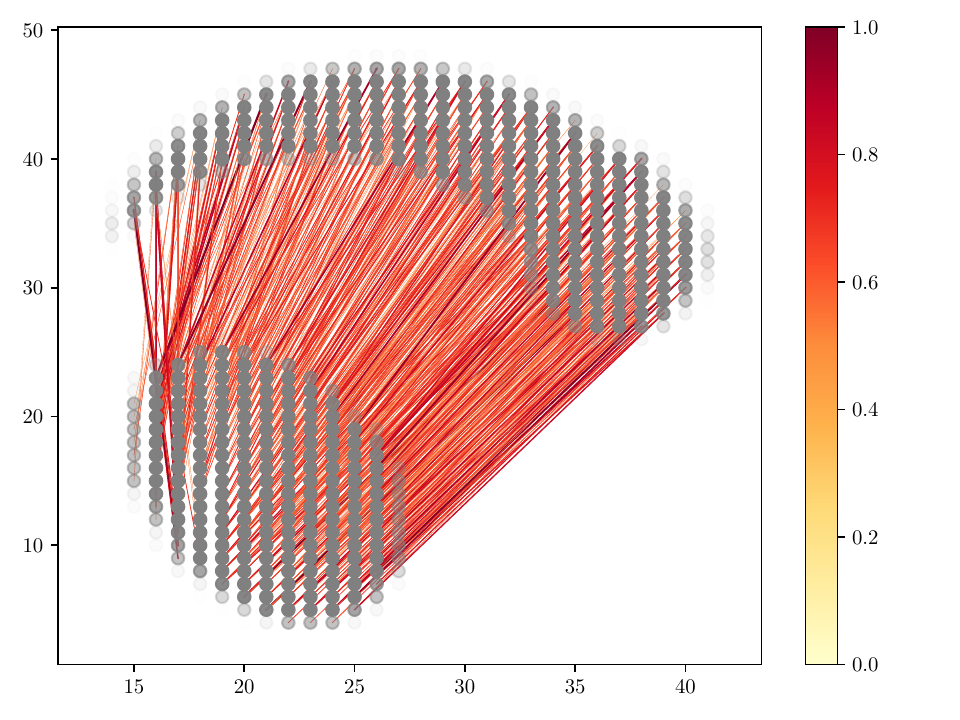}
        &\includegraphics[width=0.34\linewidth,clip=true, trim=10pt 10pt 40pt 10pt]{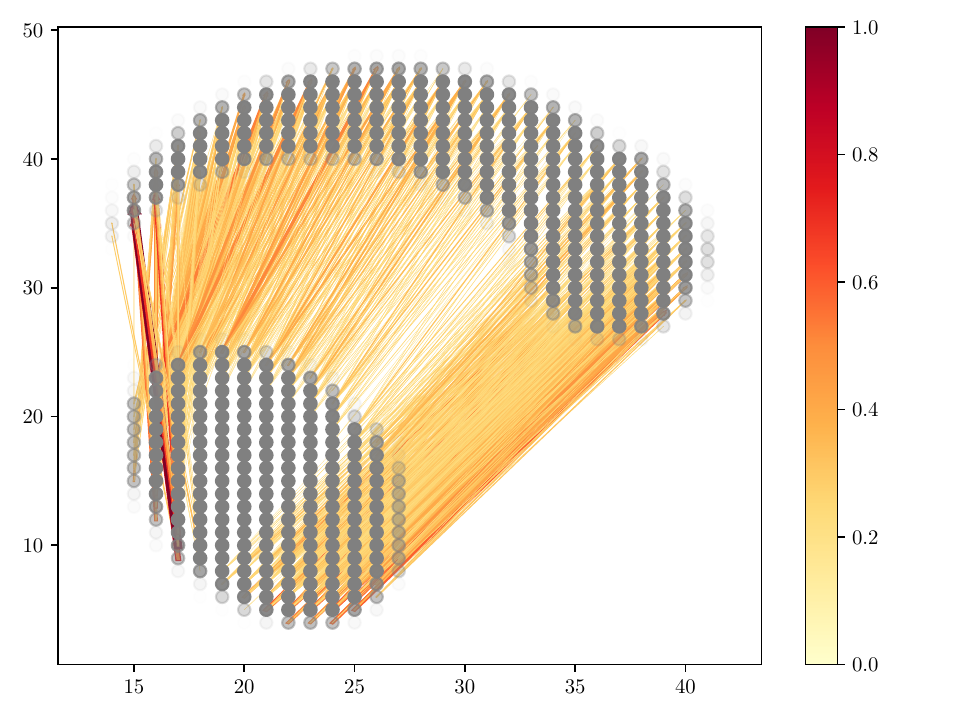}
        &\includegraphics[width=0.34\linewidth,clip=true, trim=10pt 10pt 40pt 10pt]{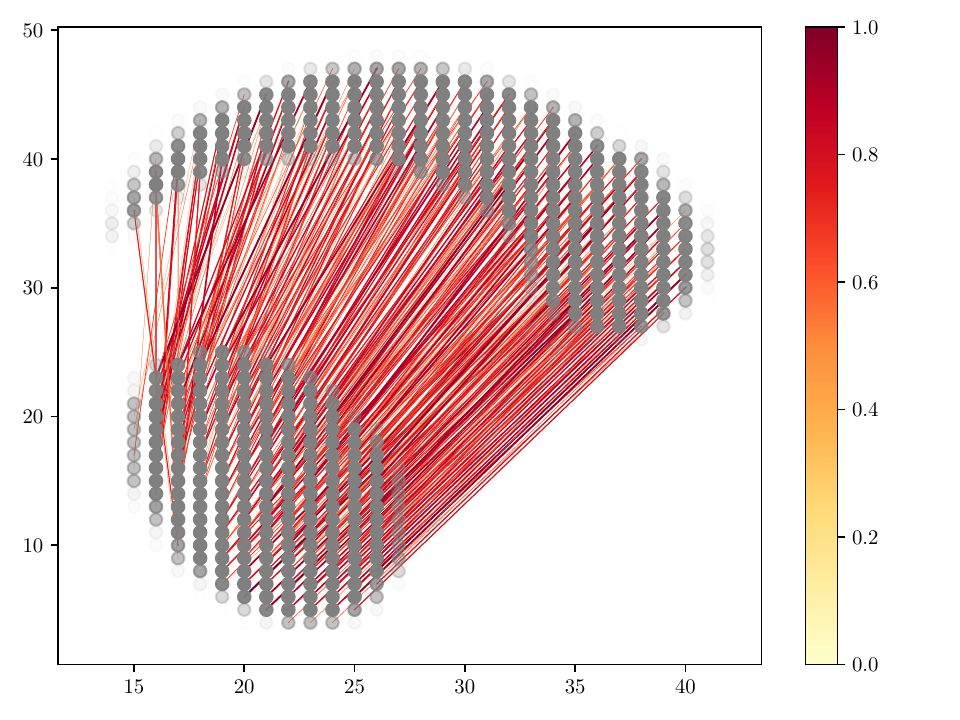} \\
        \includegraphics[width=0.34\linewidth,clip=true, trim=10pt 10pt 40pt 0pt]{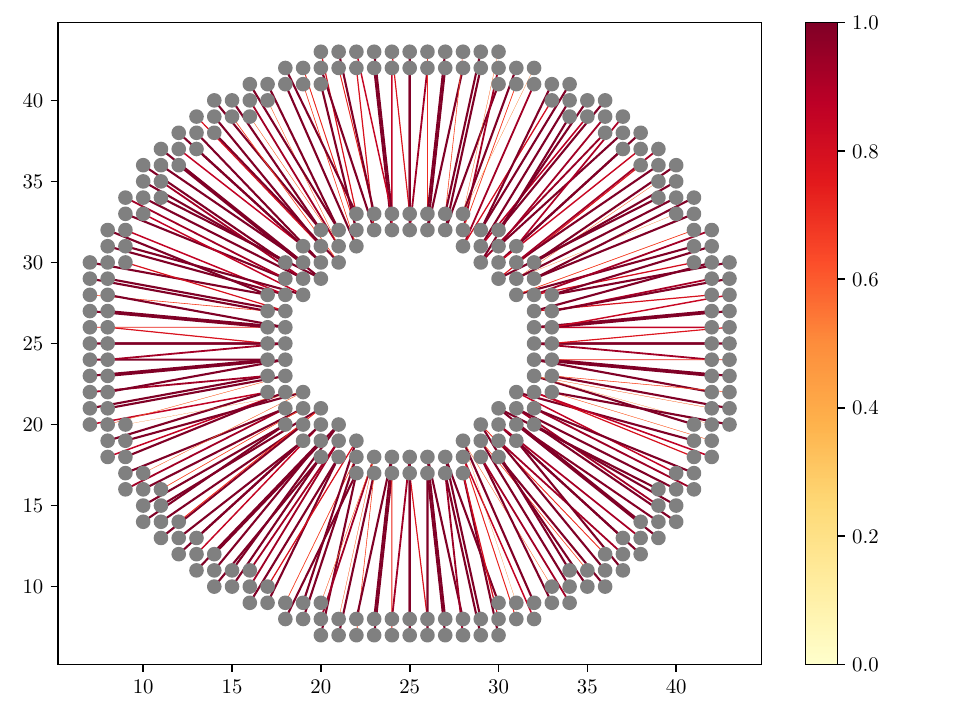}
        &\includegraphics[width=0.34\linewidth,clip=true, trim=10pt 10pt 40pt 0pt]{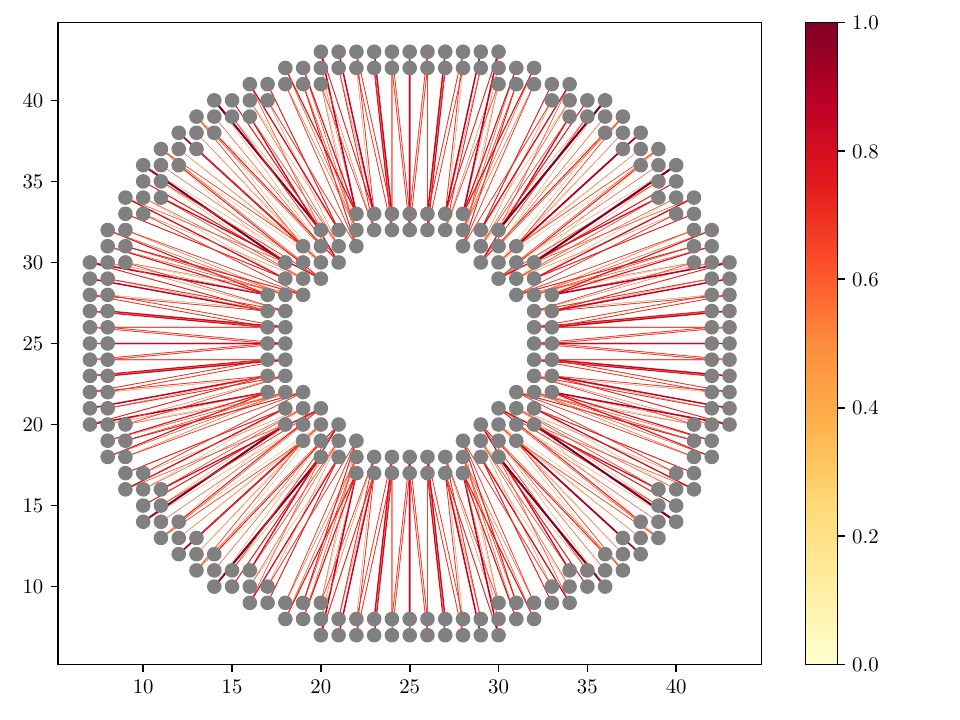}
        &\includegraphics[width=0.34\linewidth,clip=true, trim=10pt 10pt 40pt 0pt]{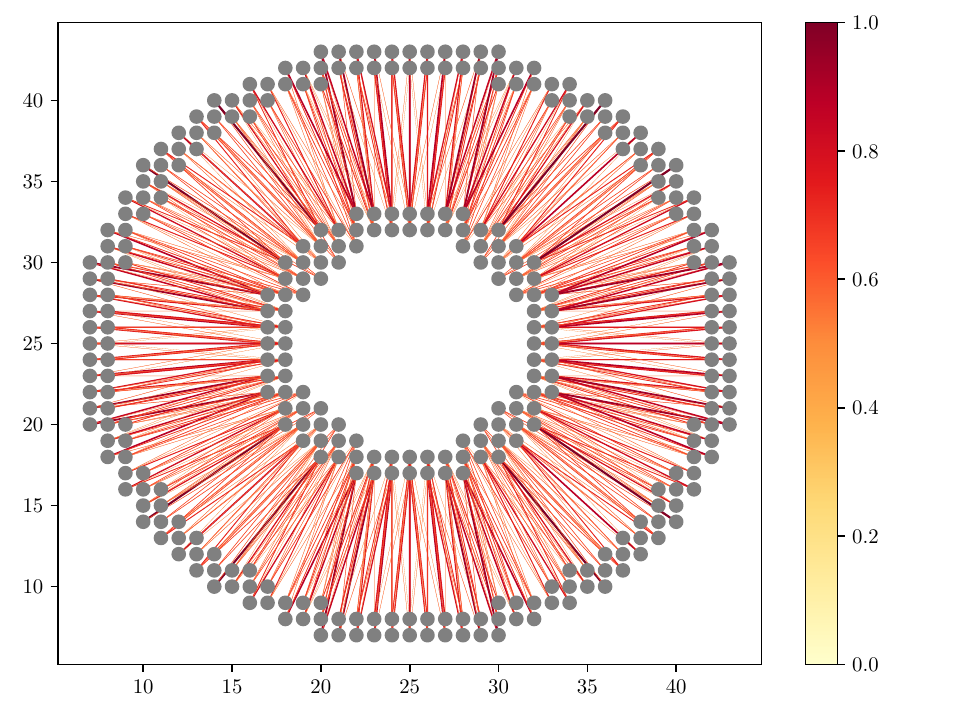}
        &\includegraphics[width=0.34\linewidth,clip=true, trim=10pt 10pt 40pt 0pt]{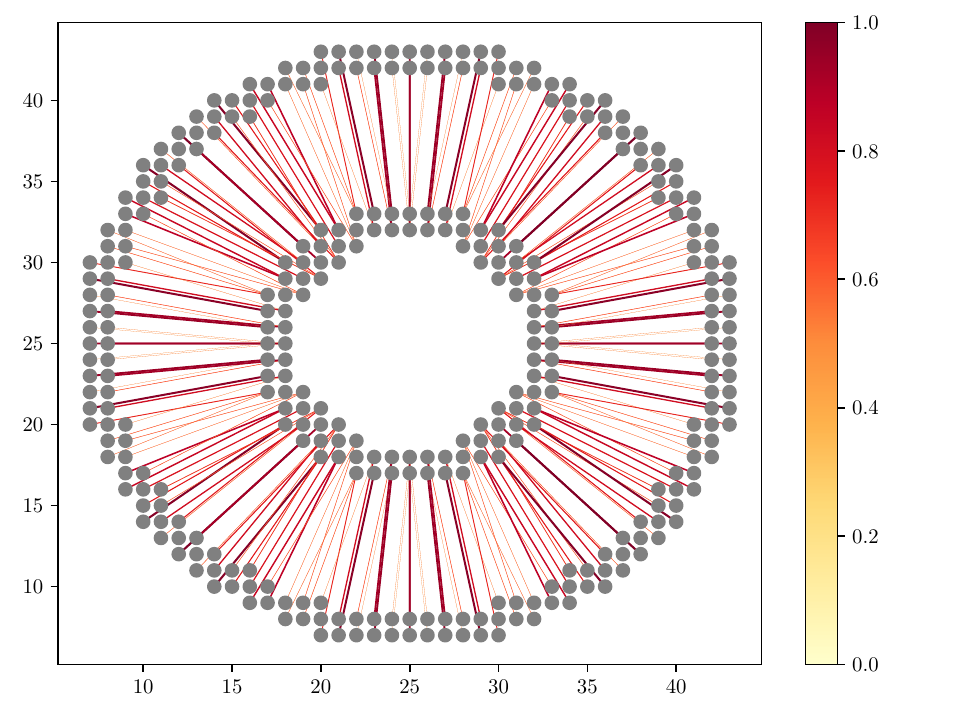}
    \end{tabular}}
    \caption{
    Comparison of the transport maps (from left to right):
    Unregularized, KL regularized OT
    with $\varepsilon=5\cdot 10^{-4}$,
    and R$_\alpha$ regularized OT 
    with $\varepsilon=1$ and $\alpha \in \{10^{-2}, 10^{-7}\}$.
    Only the trajectories with masses in the top 50\% are visualized 
    (rescaled to lie in $[0,1]$; for better discernibility of the differences).}
    \label{fig:2d-disjoint-measures}
\end{figure}

First, we notice 
that the computation is again not straightforward for the KL regularized OT problem for smaller regularization values.
Second,
the transport plan for (relatively) large regularization parameters
puts more weight on those trajectories with larger distances
(cf.~\Cref{fig:2d-disjoint-measures}, third---top and bottom).
This can be significantly reduced
by using smaller $\alpha$ parameters for the Rényi divergence
(cf.~Fig.~\ref{fig:2d-disjoint-measures}, fourth---top and bottom).
Those are quite close to the unregularized optimal transport plan
(cf.~Fig.~\ref{fig:2d-disjoint-measures}, first---top and bottom),
but allows splitting of mass over the support.

\subsection{Predicting Voter Migration with Rényi Regularized OT}

We want to use our approach to predict changed voter decisions, similar to the experiment in~\cite{MNPN17}.
Since the election of the Berlin State Parliament in September of 2021 had to be repeated in February of 2023,
it is interesting to compare both election results, because of the short interim.
In general, 
the elections are held only every five years.

We will focus on the six parties that achieved a result 
larger than $3\%$ in both elections,
and collect the others in 
"Others"
and NV (non-voters).
In total, we end up with eight parties.

The marginals are the publicly available results of both elections\footnotemark[3],
whereas the cost matrix is a priori unknown.
To construct a suitable cost matrix we will use the results 
of the "Federal Agency for Civic Education"~\footnotemark[1], which compared the parties according to their national party platform and gave similarity scores between all parties in percent, yielding for each party a vector $\mu_{\text{party}} \in [0,1]^8$ of these scores.
For the similarity score of a party to "Others"
we chose the mean of all other parties.
The similarity to "NV" is defined as 1,
since there is a $100\%$ match to any party
with the reason that the political interests of non-voters are unknown.
With these definitions, we assume that each voter's decision is 
rational and based on the political position of each party.
Using $\mu_{\text{party}}$, we define the distance matrices 
$\MM = (m_{i,j})_{1\le i,j\le 8} \in \{\MM^{\text{Eucl}}, \MM^{\text{sqEucl}}, \MM^{\text{Riesz}}, \MM^{\text{RBF}}, \MM^{\text{Res}}\}$, 
whose $(i, j)$-th entries are 
$\phi(\| \mu_{\text{party}_i} - \mu_{\text{party}_j} \|_2)$, 
where 
\begin{equation}
\label{eq:DistMatDis}
    \begin{gathered}
        \phi^{\text{Eucl}}(r) \coloneqq r, \qquad
        \phi^{\text{sqEucl}}(r) \coloneqq r^2, \qquad
        \phi^{\text{Riesz}}(r) \coloneqq \sqrt{2 r}, \\
        \phi^{\text{RBF}}(r) \coloneqq \sqrt{2(1 - e^{- \gamma r^2})}, \qquad
        \phi^{\text{Res}}(r) \coloneqq \sqrt{2\left(\frac{1}{\gamma} - \frac{1}{\sqrt{\gamma^2 + r^2}} \right)},
    \end{gathered}
\end{equation}
and $\gamma > 0$.

We expect the OT plan to be a diagonally dominant matrix
because most voters will have not changed their decision between both elections, 
see \Cref{fig:renyi_ot_vote}.
For different matrices,
we observe that the Rényi regularized plans 
outperform all the other ones (KL, Tsallis and unregularized OT),
see \Cref{tab:3}.

\begin{figure}[t!]
    \includegraphics[width=\linewidth, clip=true, trim=10pt 25pt 15pt 20pt]{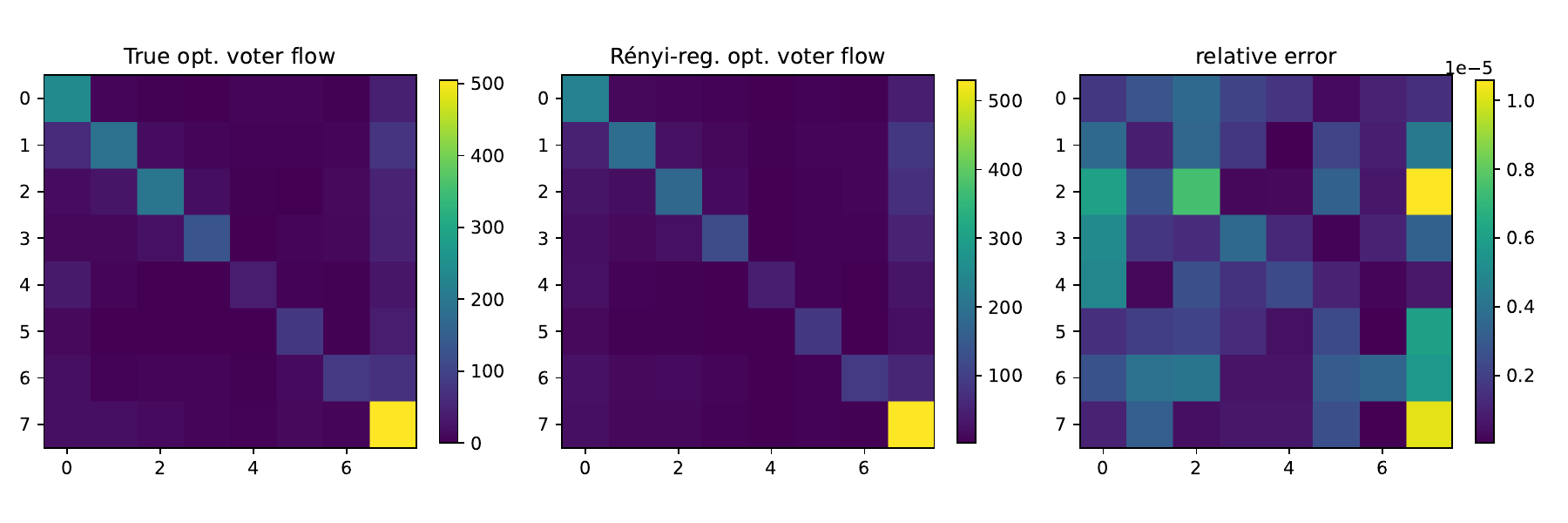}
    \includegraphics[width=\linewidth, clip=true, trim=10pt 25pt 15pt 20pt]{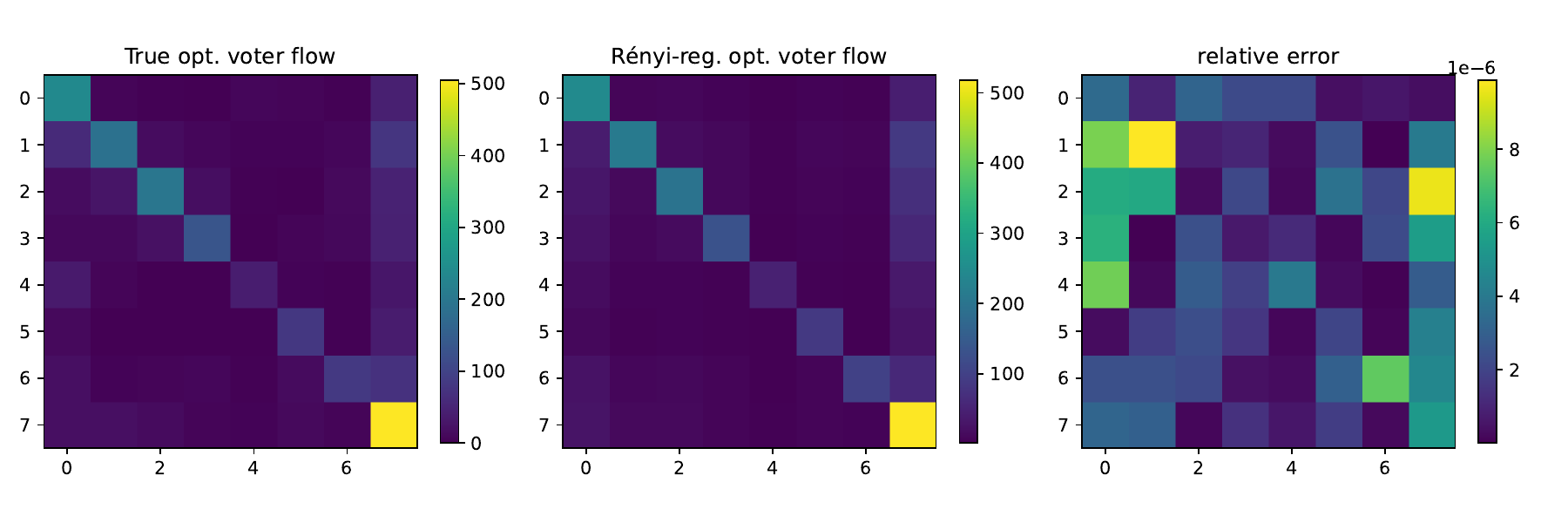}
    \includegraphics[width=\linewidth, clip=true, trim=10pt 25pt 15pt 20pt]{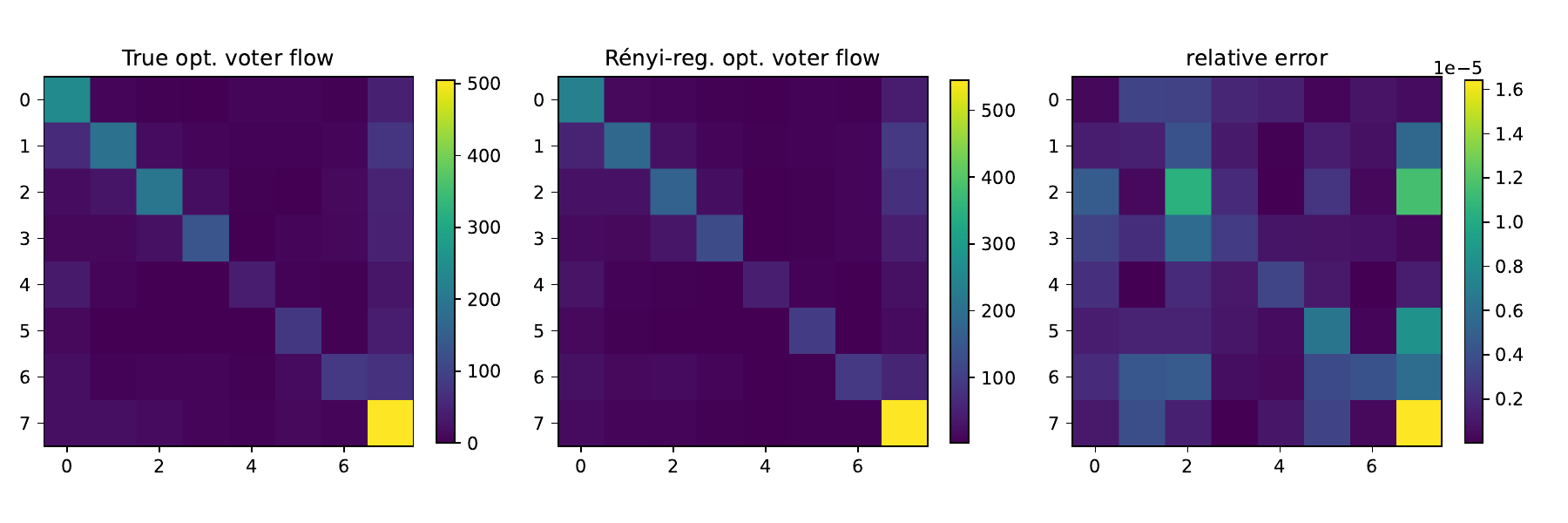}
    \includegraphics[width=\linewidth, clip=true, trim=10pt 25pt 15pt 20pt]{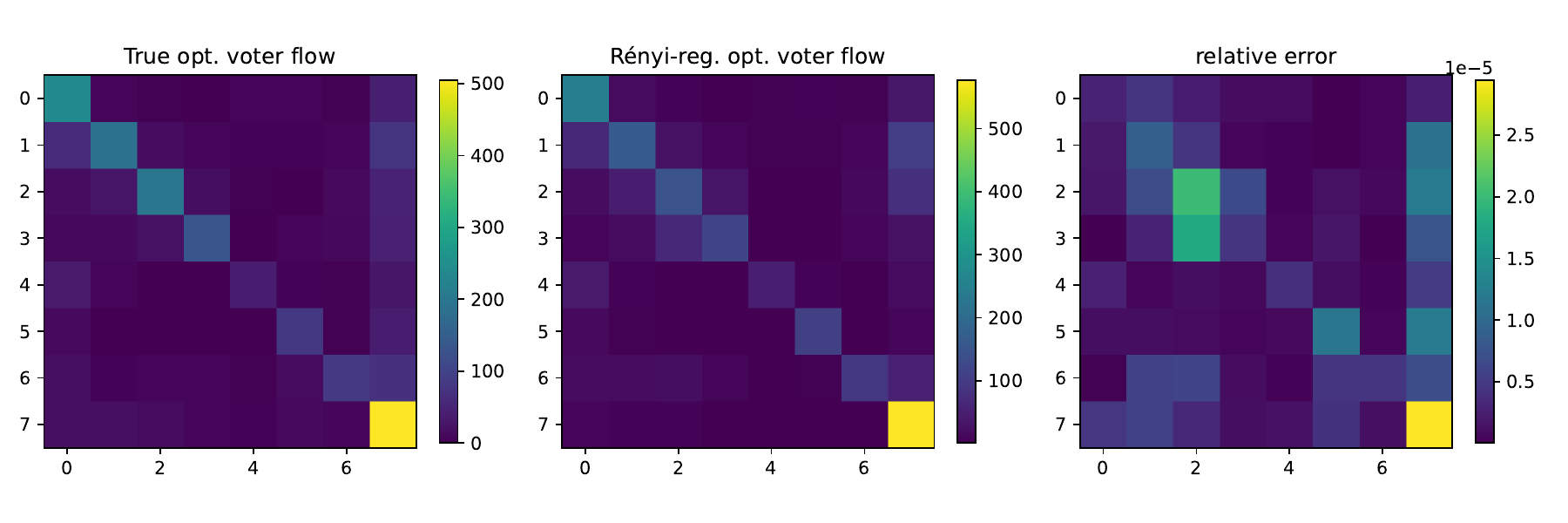}
    \vspace{-0.6cm}
    \caption{
    The true voter migration (left) evaluated by "Infratest dimap"\protect\footnotemark[2]. 
    The Rényi regularized OT plan $\pi_c^{\alpha, \eps}(\mu, \nu)$ with $\eps = 1$.
    $\mu, \nu$ are the Berlin elections results in 2021 and 2023 
    (source: "Federal Statistical Office"\protect\footnotemark[3]).
    Here, $\MM = \MM^{\text{sqEucl}}$ (bottom), with $\alpha = 0.2$
    $\MM = \MM^{\text{Eucl}}$ (middle bottom), with $\alpha = 0.3$,
    $\MM = \MM^{\text{RBF}}$ (middle top), with $\alpha = 0.3$
    and $\MM = \MM^{\text{Riesz}}$ (top), with $\alpha = 0.5$.
    }
    \label{fig:renyi_ot_vote}
\end{figure}

\footnotetext[1]{\url{https://interaktiv.morgenpost.de/parteien-bundestagswahl-2017/}}
\footnotetext[2]{\url{https://interaktiv.tagesspiegel.de/lab/waehlerwanderung-abgeordnetenhauswahl-berlin-2023/}}
\footnotetext[3]{\url{https://www.wahlen-berlin.de/wahlen/BE2021/AFSPRAES/ergebnisse.html} and \url{https://www.wahlen-berlin.de/wahlen/BE2023/AFSPRAES/agh/index.html}}

\begin{table}[t!]
\resizebox{\textwidth}{!}{
\begin{tabular}{c | c | c c c c}
    \toprule 
        distance matrix \eqref{eq:DistMatDis}         & regularizer                          & abs error $\pm$ std                       & KL error                        & mean squared error \\
        \midrule 
        $\MM^{\text{Riesz}}$                                  & $\KL$                        & \num{1.881e1} $\pm$ \num{2.049e1}         & \num{5.685e2}                   & \num{4.949e4} \\
                                                              & $\TT_{q = 1.2}$              & \num{8.627} $\pm$ \num{8.844}             & \num{2.481e2}                   & \num{9.769e3} \\
                                                              & $\OT$                        & \num{1.845e1} $\pm$ \num{2.358e1}         & \num{7.655e2}                   & \num{5.738e4} \\
        see \Cref{fig:renyi_ot_vote} (top)                    & $\text{R}_{\alpha = 0.5}$    & $\mathbf{\num{5.906} \pm \num{5.408}}$    & \textbf{\num{1.899e2}}          & \textbf{\num{4.104e3}}& \\
        \midrule 
        $\MM^{\text{RBF}}_{\gamma = 1}$                       & $\KL_{\eps = 1}$             & \num{2.297e1} $\pm$ \num{2.632e1}         & \num{7.526e2}                   & \num{7.813e4} \\
                                                              & $\TT_{q = 1.4}$              & \num{8.747} $\pm$ \num{1.488e1}           & \num{2.786e2}                   & \num{1.908e4} \\
                                                              & $\OT$                        & \num{1.845e1} $\pm$ \num{2.341e1}         & \num{7.555e2}                   & \num{5.738e4} \\
        see \Cref{fig:renyi_ot_vote} (top middle)             & $\text{R}_{\alpha = 0.3}$    & $\mathbf{\num{6.252} \pm \num{7.752}}$  & \textbf{\num{1.876e2}}          & \textbf{\num{6.348e3}}& \\
        \midrule 
        $\MM^{\text{Eucl}}$                                   & $\KL_{\eps = 1}$             & \num{2.4221e1} $\pm$ \num{2.848e1}        & \num{8.422e2}                   & \num{9.008e4} \\
        scalar: 10                                            & $\TT_{q = 10}^\dag$          & \num{9.409} $\pm$ \num{1.529e1}           & \num{3.173e2}                   & \num{2.063e4} \\
                                                              & $\OT$                        & \num{1.845e1} $\pm$ \num{2.358e1}         & \num{7.655e2}                   & \num{5.738e4} \\
        see \Cref{fig:renyi_ot_vote} (bottom middle)          & $\text{R}_{\alpha = 0.3}$    & $\mathbf{6.611 \pm 7.868}$                & \textbf{\num{2.128e2}}          & \textbf{\num{6.759e3}} \\
        \midrule 
        $\MM^{\text{Res}}_{\gamma = 2}$                       & $\KL_{\eps = 1}$             & $\num{3.223e1} \pm \num{4.172e1}$         & \num{1.276e3}                   & \num{1.779e5} \\
        scalar: 5                                             & $\TT_{q = 0.7}^\dag$         & $\num{8.187} \pm \num{1.110e1}$           & \num{2.501e2}                   & \num{1.218e4} \\
                                                              & $\OT$                        & $\num{1.845e1} \pm \num{2.358e1}$         & \num{7.655e2}                   & \num{5.738e4} \\
                                                              & $\text{R}_{\alpha = 0.1}$    & $\mathbf{\num{5.703} \pm \num{6.215}}$    & \textbf{\num{1.880e2}}          & \textbf{\num{4.553e3}}&\\
        \midrule 
        $\MM^{\text{sqEucl}}$                                 & $\KL_{\eps = 1}$             & $\num{2.476e1} \pm \num{2.906e1}$         & \num{9.154e2}                   & \num{9.329e4} \\
                                                              & $\TT_{q = 1.7}$              & $\num{1.407e1} \pm \num{1.844e1}$         & \num{5.388e2}                   & \num{3.444e3} \\
                                                              & $\OT$                        & \num{1.920e1}  $\pm$ \num{2.432e1}        & \num{8.554e2}                   & \num{6.146e4} \\
        see \Cref{fig:renyi_ot_vote} (bottom)                 & $\text{R}_{\alpha = 0.2}$    & $\mathbf{\num{9.732} \pm \num{1.295e1}}$  & \textbf{\num{2.979e2}}          & \textbf{\num{1.679e2}}& \\
    \bottomrule
\end{tabular}}
\caption{\label{tab:3}
Comparing the Rényi ($\RR_\alpha$), KL and Tsallis ($\TT_q$) 
regularized optimal transport plans 
to the unregularized transport plan like in \Cref{tab:1}.
We choose $\eps = 1$ and for each regularizer the value of $\alpha \in \{k\cdot10^{-1} : 1\le k \le 9\}$ of $q \in \{k\cdot10^{-1} : k \in \N\}$ with the best result.
The $\dag$ denotes that we used a distance matrix $\MM$ scaled by the scalar factor,
and modified the regularization parameter accordingly.}
\end{table}

\section{Conclusion} \label{sec:Conclusions}

In this paper we have introduced a method to regularize the OT problem
using the Rényi divergence of order $\alpha \in (0, 1)$, 
which yields a premetric on probability measures. 
We have derived the dual problem of the regularized version and showed strong duality.
In contrast to the well-known KL regularized OT problem,
we end up with hard constraints for the potentials in the dual formulation.
Through the dual formulation,
we can represent the regularized optimal plan 
using the dual potentials,
which are unique almost everywhere up to constants.
Finally,
we have shown the convergence behaviors of the Rényi regularized OT 
to the unregularized OT whenever $\alpha \to 0$ or $\eps \to 0$.
Moreover,
the convergence to the KL regularized OT shows up as a special case
in the convergence with respect to $\alpha$.
Numerically,
we have verified the convergence behaviors and obtained tight transport plans.
In the real-world example of predicting voter migration 
we have observed the superiority of the Rényi regularized OT plans
in comparison to other transport plans.

Directions for further work include investigating the convergence speed of $\OT_{\eps, \alpha}(\mu, \nu)$ 
for $\eps~\to~0$ (which is exponential for $\alpha = 1$).
Furthermore, 
the analogue of the Sinkhorn divergence 
for debiasing the Rényi regularized OT can be defined.
A similar proof as in~\cite{FSVATP19} might be suitable---a 
crucial point might be proving the Gateaux differentiability of $\OT_{\eps, \alpha}$,
where we need to prove the (uniform) convergence of the direct sum of the 
dual potentials. 
Such a Rényi-Sinkhorn divergence would interpolate between OT,
MMD and the Sinkhorn divergence as a direct conclusion of \Cref{theorem:ConvToOTKL}
generalizing the interpolation property of the Sinkhorn divergence.

\textbf{Acknowledgements}\quad
First and foremost the authors would like to thank their advisor Gabriele Steidl 
for her guidance. 
Furthermore, 
they thank the anonymous reviewer for their detailed comments which greatly improved the manuscript.
Lastly, they 
they acknowledge many fruitful discussions with Christian Wald.
V.S. gratefully acknowledge funding from the BMBF project 
"VI-Screen" with number 13N15754. 

\textbf{Statements and Declaration}\quad
The authors have no conflicts to disclose.
All authors contributed equally.
They read and approved the final manuscript.

\textbf{Data availability}\quad
All data is available as open access as disclosed.


\begin{appendices}

In this appendix we collects additional proofs and details as well as two interesting 
results cached up during our research.
After proving some simple results from the main text
we give a generalized statement of the (strict) convexity
of the mixture of (strictly) log-convex functions.
The proof is like in~\cite[Prop.~16.D.4]{MOA11}, 
but valid for a more general setting.
In the next section,
an extension of \Cref{theorem:ConvConjRenyiDiv} is given on a non-compact set.
Last but not least,
a simple subgradient method to solve the dual problem is given.
However, 
in contrast to KL regularized OT, 
there is no apparent relationship between the primal and the dual algorithm 
and our dual algorithm 
converges much slower than the primal algorithm 
in terms of convergence speed, presented in the main text.

\section{Numerical implementation details}

The code is available on GitHub at 
\url{https://github.com/JJEWBresch/renyi_regularized_OT}.
The algorithm from \cref{sec:algo} 
is implemented
in Python~3.11.4 using Numpy~1.25.0, Scipy~1.11.1 and POT~0.9.3.
The experiments are performed on an off-the-shelf MacBookPro 2020
with Intel Core i5 Chip (4‑Core CPU, 1.3~GHz) and 8~GB RAM.
For the mirror descent algorithm,
we used a variant of the Sinkhorn projection 
which is implemented and used in~\cite{MNPN17}.
In all experiments,
the Sinkhorn projection is performed until the residuum 
of the marginals in the 2-norm is smaller than $10^{-4}$.
The proposed mirror descent (\Cref{alg:MD}) terminates 
whenever the residuum of the iterates in the Euclidean norm is smaller than $\num{e-6}$.

\section{Additional proofs}
\label{sec:app_add_proofs}

We now prove \Cref{lemma:restricted_polytope_compact}.
\begin{proof}
    The Rényi divergence $\RR_\alpha$ is jointly convex by \Cref{prop:RenyiDivProps}, 
    hence $\RR_\alpha(\,\cdot \mid \mu\otimes\nu)$ is convex 
    and thus so is $\Pi_\gamma^\alpha(\mu,\nu)$.
    
    Since $\RR_\alpha(\;\cdot\mid \mu\otimes\nu)$ is weakly lower semicontinuous by \Cref{prop:RenyiDivProps}, 
    the sublevel set $\{ \pi \in \P(X \times X): \RR_{\alpha}(\pi \mid \mu \otimes \nu) \le \gamma \}$ is (weakly) closed.
    Since $\Pi(\mu, \nu)$ is weakly compact by~\cite[Thm.~4.1]{V08}, so is $\Pi_{\alpha}^{\gamma}(\mu, \nu)$.
\end{proof}

Now, we prove \Cref{theorem:RenyiDistance}.
\begin{proof}
    We verify the three axioms of a metric.
    \begin{enumerate}
        \item
        Since $d_{c, \gamma, \alpha}(\mu, \nu)\ge 0$, we have $\tilde d_{c, \gamma, \alpha}(\mu, \nu)\ge 0$.
        If $\mu = \nu$, then $\tilde{d}_{c, \gamma, \alpha}(\mu, \nu) = 0$ by definition.
        Let $\tilde d_{c, \gamma, \alpha}(\mu, \nu) = 0$. 
        Then either $\mu = \nu$ or $d_{c, \gamma, \alpha}(\mu, \nu) = 0$. 
        In the latter case, we also have $\mu = \nu$, by the same argument as for the Wasserstein-$p$ distance, see \cite[p.~94]{V08}.
        
        \item
        Consider the transposition map $m \colon X \times X \to X \times X$, $(x, y) \mapsto (y, x)$.
        Since $c$ is symmetric, we have $\langle c, m_{\#} \pi \rangle = \langle c \circ m, \pi \rangle = \langle c, \pi \rangle$ and
        \begin{equation} \label{eq:RenyiSymm}
            \RR_\alpha(m_{\#}\pi \mid \nu\otimes\mu)
            = \RR_\alpha(\pi \mid \mu\otimes\nu)
            \le \gamma. 
        \end{equation}
        Since $m_{\#}(m_{\#} \pi) = (m \circ m)_{\#} \pi = \pi$, \eqref{eq:RenyiSymm} implies $\pi \in \Pi_{\gamma}^{\alpha}(\mu, \nu)$ if and only if $m_{\#}\pi \in \Pi_{\gamma}^{\alpha}(\nu,\mu)$.
        Hence
        \begin{align*}
            d_{c, \gamma, \alpha}(\mu, \nu)^p
            = \min_{\pi \in \Pi_{\gamma}^{\alpha}(\mu, \nu)} \langle c, \pi \rangle
            & = \min_{\pi \in \Pi_{\gamma}^{\alpha}(\nu, \mu)} \langle c, m_{\#} \pi \rangle \\
            & = \min_{\pi \in \Pi_{\gamma}^{\alpha}(\nu, \mu)} \langle c, \pi \rangle
            = d_{c, \gamma, \alpha}(\nu, \mu)^p.
        \end{align*}

        \item 
        We consider the three marginals $\mu, \xi, \nu \in \P_p(X)$.
        Let $\pi_{12} \in \Pi_\gamma^\alpha(\mu,\xi)$ and $\pi_{23} \in \Pi_\gamma^\alpha(\xi,\nu)$
        be the respective solutions of \eqref{eq:Renyi-Sinkhorn_metric}, 
        i.e., $d_{c,\gamma,\alpha}(\mu,\xi)^p = \langle c, \pi_{12} \rangle$ and 
        $d_{c,\gamma,\alpha}(\xi,\nu)^p = \langle c, \pi_{23} \rangle$.

        Thanks to the disintegration theorem for measures on Polish spaces~\cite[pp.~78-80]{DM82}
        there exists families $((\pi_1)_{y}\big)_{y \in X}, ((\pi_2)_{y})_{y \in X} \subset \P(X)$ both uniquely determined $\xi$-almost everywhere,
        such that for all $A, B, C \in \B$ we have
        \begin{equation*}
            \pi_{12}(A \times B) = \int_B (\pi_1)_{y}(A) \diff{\xi}(y) 
            \quad\text{and}\quad
            \pi_{23}(B \times C) = \int_{B} (\pi_2)_{y}(C) \diff{\xi}(y).
        \end{equation*}
        Thanks to the Gluing lemma~\cite[pp.~11-12]{V08}, 
        there exists a $\pi_{123} \in \P(X^3)$ with 
        \begin{equation} \label{eq:t_gluing}
            \pi_{123}(A \times B \times C)
            = \int_{B} (\pi_1)_{y}(A) (\pi_2)_{y}(C) \diff{\xi}(y)
        \end{equation}
        for all $A, B, C \in \B$ and 
        \begin{align} \label{eq:GluingLemma}
            P_{\#}^{(1,2)}\pi_{123} = \pi_{12} 
            \quad\text{and}\quad
            P_{\#}^{(2,3)}\pi_{123} = \pi_{23}.
        \end{align}
        Then, $\pi_{13} \coloneqq P_{\#}^{(1,3)}\pi_{123}$ fulfills $\pi_{13} \in \Pi(\mu,\nu)$. 
        We now show that $\pi_{13} \in \Pi_{\gamma}^{\alpha}(\mu, \nu)$ as well.

        The Rényi divergence satisfies the following chain rule:
        \begin{align*}
            R_{\alpha}(\pi_{123} \mid \mu \otimes \pi_{23})
            & = \frac{1}{\alpha - 1}\ln\left(\int_{X^3} \left( \frac{\diff \pi_{123}}{\diff (\mu \otimes \pi_{23})}(x, y, z)\right)^{\alpha} \diff{\mu}(x) \diff{\pi_{23}}(y, z) \right) \\
            & = \frac{1}{\alpha - 1}\ln\left(\int_{X^3} \left( \frac{\diff \pi_{12}}{\diff (\mu \otimes \xi)}(x, y)\right)^{\alpha} \diff{\mu}(x) \diff{\pi_{23}}(y, z) \right) \\
            & = R_{\alpha}(\pi_{12} \mid \mu \otimes \xi).
        \end{align*}
        Since $(P^{(1, 3)})_{\#} (\mu \otimes \pi_{2 3}) = \mu \otimes \nu$, 
        the data processing inequality \cref{prop:RenyiDivProps}.5 with the deterministic Markov kernel $\kappa_{P^{(1,3)}}$ yields
        \begin{align*}
            \RR_\alpha(\pi_{13} \mid \mu\otimes{ \nu})
            & = R_{\alpha}\left(\kappa_{P^{(1,3)}} \circ \pi_{123} \mid \kappa_{P^{(1,3)}} \circ (\mu \otimes \pi_{23}) \right) \\
            & \le R_{\alpha}(\pi_{123} \mid \mu \otimes \pi_{23})
            = R_{\alpha}(\pi_{12} \mid \mu \otimes \xi)
            \le \gamma,
        \end{align*}
        such that $\pi_{13} \in \Pi_\gamma^\alpha(\mu, \nu)$.

        Using that $\pi_{13} \in \Pi_\gamma^\alpha(\mu,\nu)$ might be not an optimal solution for the Rényi-regularized OT problem between $\mu$ and $\nu$, 
        the triangle inequality for $c \in \D_p$,
        the Minkowski inequality (for all $p \geq 1$),
        and that the marginals of $\pi_{123}$ are known from \eqref{eq:GluingLemma}, we obtain
        \begin{align*}
            d_{c,\gamma,\alpha}(\mu,\nu)
            & \le \left(\int_{X^2} d(x,z)^p \diff \pi_{13}(x,z)\right)^{\frac{1}{p}} 
            {= \left(\int_{X^3} d(x,z)^p \diff \pi_{123}(x,y,z)\right)^{\frac{1}{p}}} \\
            & \le \left(\int_{X^2} \left( d(x,y) + d(y,z) \right)^p\diff \pi_{123}(x,z)\right)^{\frac{1}{p}} \\
            & \le \left(\int_{X^3} d(x,y)^p \diff \pi_{13}(x,y,z) \right)^{\frac{1}{p}}
            + \left(\int_{X^3} d(y,z)^p\diff \pi_{123}(x,y,z) \right)^{\frac{1}{p}}\\
            & = \left( \int_{X^2} d(x,y)^p\diff {\pi_{12}}(x,y) \right)^{\frac{1}{p}}
            + \left( \int_{X^2} d(y,z)^p\diff {\pi_{23}}(y,z) \right)^{\frac{1}{p}} \\
            & = d_{c,\gamma,\alpha}(\mu,\xi) + d_{c,\gamma, \alpha}(\xi,\nu).
        \end{align*}
    \end{enumerate}
\end{proof}

We show the following generalization of the weak continuity statement from \Cref{prop:RenyiDivProps}.

\begin{lemma}[Joint one-sided weak lower semicontinuity of $\RR_{\alpha}$] \label{lem:JointlyContRenyiDiv}
    Let
    $\nu \in \P(X)$.
    Moreover, let $(\alpha_n)_{n \in \N} \subset (0,1)$ with $\alpha_n \nearrow \alpha \in (0,1]$
    and let $(\mu_n)_{n \in \N} \subset \P(X)$ converge weakly
    to $\mu$ for $n \to \infty$.
    Then,
    \begin{equation*}
        \liminf_{n \to \infty} \RR_{\alpha_n}(\mu_n \mid \nu) 
       \ge \RR_\alpha(\mu \mid \nu).
    \end{equation*}
\end{lemma}

\begin{proof}
    Without loss of generality and renaming we can select an increasing subsequence $(\alpha_n)_{n \in \N} \subset (0, \alpha]$.
    By the monotonicity of the Rényi divergence with respect to $\alpha$
    from \Cref{lem:RenyiMonotone} we have for all $m \in \N$ 
    \begin{equation*}
        \liminf_{n \to \infty} \RR_{\alpha_n}(\mu_n \mid \nu)
       \ge \liminf_{n \to \infty} \RR_{\alpha_m}(\mu_n \mid \nu)
        = \RR_{\alpha_m}(\mu \mid \nu),
    \end{equation*}
    where the last equality holds by the weak lower semicontinuity of the Rényi divergence.
    Now, we can take the limit $m \to \infty$ and by the continuity 
    of the Rényi divergence in its order $\alpha$ (see \Cref{prop:RenyiDivProps})
    we obtain the assertion.
\end{proof}

\section{Log-Convex functions}

In the following, let $E$ be a convex subset of a vector space.

\begin{definition}[Log-convex function]
    A function $f \colon E \to (0, \infty)$ is 
    (strictly) log-convex if the composition $\ln \circ f \colon E \to \R$ is (strictly) convex.
\end{definition}

A function $f$ is (strictly) log-convex if and only if 
$f(t x + (1 - t) y) \overset{(<)}{\le} f(x)^t f(y)^{1 - t}$ 
for all $x, y \in E$ and all $t \in (0, 1)$.

\begin{proposition}[Mixture of log-convex functions is log-convex] \label{prop:LogConvexMixture}
    Let $(Y, \Sigma, \theta)$ be a measure space.
    Consider a measurable function $\Phi \colon E \times Y \to (0, \infty)$.
    If $\Phi(\cdot, y) \colon E \to (0, \infty)$ is 
    (strictly) log-convex for every $y \in Y$, 
    then the mixture
    \begin{equation*}
        M \colon E \to \R, \qquad
        h \mapsto \int_{Y} \Phi(h, y) \diff{\theta}(y)
    \end{equation*}
    is (strictly) log-convex, too.
\end{proposition}

\begin{proof}
    Let $y \in Y$.
    Since $\Phi(\cdot, y)$ is (strictly) log-convex, we have for all $t \in (0, 1)$ and all $h_1, h_2 \in E$ that 
    \begin{equation} \label{eq:log-convex1}
        \Phi(t h_1 + (1 - t) h_2, y)
        \overset{(<)}{\le} \Phi(h_1, y)^t \Phi(h_2, y)^{1 - t}.
    \end{equation}
    Hence Hölder's inequality with conjugate exponents $p = \frac{1}{t}$ and $q = \frac{1}{1 - t}$ implies
    \begin{align*}
        M(t h_1 + (1 - t) h_2)
        & = \int_{Y} \Phi(t h_1 + (1 - t) h_2, y) \diff{\theta}(y) \\
        & \overset{(<)}{\le} \int_{Y} \Phi(h_1, y)^t \Phi(h_2, y)^{1 - t} \diff{\theta}(y) \\
        & \le \left(\int_{Y} \Phi(h_1, y) \diff{\theta}(y)\right)^{t} \left(\int_{Y} \Phi(h_1, y) \diff{\theta}(y)\right)^{1 - t} \\
        & = M(h_1)^t M(h_2)^{1 - t}.
    \end{align*}
\end{proof}

\section{Convex conjugate style formulation of the Rényi divergence for measures on \texorpdfstring{$\R^d$}{Rd}}

For a Polish space $X$ let $\C_c(X) \subset \C_0(X) \subset \C_b(X) \subset \C(X)$, be the subsets of continuous functions that are compactly supported resp. vanishing at infinity resp. bounded.
If $X$ is compact, then $\C_c(X) = \C(X)$.

The next lemma states a variational formulation of the Rényi divergence 
for probability measures on the non-compact space $X = \R^d$, generalizing \Cref{theorem:RenyiPreconjugate}.
Notice that the variables of the variational formulation are now $\C_0(\R^d)$.

\begin{lemma}[Variational formulation of $\RR_{\alpha}$ on $\C_0(\R^d)$]
\label{lem:renyi_var_form_C0}
    Let $\mu, \nu \in \P(\R^d)$ and $\alpha \in (0,1)$.
    Then we have 
    \begin{equation*}
        \RR_\alpha(\mu \mid \nu) 
        = \sup_{\phi \in \C_0(\R^d)} \left\{
        \frac{\alpha}{\alpha - 1} \ln\left(\langle e^{(\alpha - 1)\phi}, \mu \rangle\right) 
        - \ln\left(\langle e^{\alpha\phi}, \nu \rangle\right)
        \right\}.
    \end{equation*}
\end{lemma}

\begin{proof}
    We show both inequalities similar to~\cite[Lem.~6]{NSSR24}:
    \begin{enumerate}
        \item 
        Since $\C_0(\R^d) \subset \C_b(\R^d)$
        and $\Lip_b(\R^d) \subset \C_b(\R^d)$ consists of measurable functions, 
        we have with \eqref{eq:renyi_var_form}
        \begin{equation*}
            \RR_\alpha(\mu \mid \nu)
           \ge \sup_{\phi \in \C_0(\R^d)} \left\{
            \frac{\alpha}{\alpha - 1} \ln\left(\langle e^{(\alpha - 1)\phi}, \mu \rangle\right) 
            - \ln\left(\langle e^{\alpha\phi}, \nu \rangle\right)
            \right\}.    
        \end{equation*}
        
        \item 
        Let $\phi \in \C_b(\R^d)$.
        We find a family of $\C_0(\R^d)$-functions $(\phi_k)_{k \in \N}$
        such that $\lim_{k \to \infty} \phi_k(x) = \phi(x)$ for all $x \in \R^d$.
        Furthermore, we have the following inequality
        \begin{equation*}
            \sup_{X} \phi\ge \max_{X} \phi_k\ge \inf_{X} \phi_k\ge \min(0,\inf_{X} \phi),
        \end{equation*}
        which implies that 
        \begin{align*}
        |\phi_k| = 
            \begin{cases}
                \phi_k(x) \\
                -\phi_k(x)
            \end{cases}
        & \le 
            \begin{cases}
                \max_{X}\phi_k(x) \\
                -\min_{X}\phi_k(x)
            \end{cases} \\
        & \le 
            \begin{cases}
                \sup_{X}\phi(x) \\
                -\min(0,\inf_{X}\phi(x)) \\
            \end{cases} \\
        & = \begin{cases}
                \sup_{X}\phi(x) & \text{if } \phi_k(x)\ge 0,\\
                \max(0, \sup_{X} -\phi(x)) & \text{else}, 
            \end{cases}
        \le \sup_{X} |\phi| \eqqcolon \gamma
        \end{align*}
        for any $k \in \N$.
        We have 
        \begin{enumerate}
            \item 
            $|e^{\alpha \phi_k}| \le e^{\alpha |\phi_k|} \le e^{\alpha \gamma} \le e^{\gamma}$ and 
            \item 
            $|e^{(\alpha - 1) \phi_k}| \le e^{(1 - \alpha) |\phi_k|} \le e^{(1 - \alpha) \gamma} \le e^{\gamma}$ for all $\alpha \in (0,1)$ and $k \in \N$,
        \end{enumerate}
        where $e^\gamma$ is in $L^1(\mu)$, respectively $L^1(\nu)$.
        Hence, the dominated convergence theorem 
        with the continuity of $\ln \colon (0, \infty) \to \R$ 
        states for any $\phi \in C_b(\R^d)$ that 
        \begin{align*}
            & \frac{\alpha}{\alpha - 1} \ln\left(\langle e^{(\alpha - 1)\phi}, \mu \rangle\right) 
            - \ln\left(\langle e^{\alpha\phi}, \nu \rangle \right) \\
            & \quad = \lim_{k \to \infty} 
            \frac{\alpha}{\alpha - 1} \ln\left(\langle e^{(\alpha - 1)\phi_k}, \mu \rangle\right)
             - \ln\left(\langle e^{\alpha\phi_k}, \nu \rangle\right) \\
            & \quad \le \sup_{\psi \in \C_0(\R^d)}
            \frac{\alpha}{\alpha - 1} \ln\left(\langle e^{(\alpha - 1)\psi}, \mu \rangle\right) 
            - \ln\left(\langle e^{\alpha\psi}, \nu \rangle\right).
        \end{align*}
        Taking the supremum over all $\phi \in \C_b(\R^d)$ on the left hand side 
        yields the assertion.
    \end{enumerate}
\end{proof}

\begin{theorem}[Preconjugate of the Rényi divergence] \label{theorem:ConvConjRenyiDiv}
    The convex preconjugate with respect to the first component of the Rényi divergence
    of order $\alpha \in (0,1)$,
    where $\theta \in \P(\R^d)$, is
    \begin{equation*}
        \big[\RR_{\alpha}(\;\cdot \mid \theta)\big]_*(g)
        = \begin{cases}
            \ln\left( \int_{\R^d} \big(-g(x)\big)^{\frac{\alpha}{\alpha - 1}} \diff{\theta}(x) \right) + C_{\alpha}, & \text{if } g \in \C_0(\R^d), g < 0, \\
            +\infty & \, \text{else.}
            \end{cases}
    \end{equation*}
    where $C_\alpha \coloneqq -\frac{\alpha}{1 - \alpha}\left(1 + \ln\left(\frac{1 - \alpha}{\alpha}\right)\right)$.
\end{theorem}

\begin{proof}
    The first part of the proof follows the idea from~\cite[Thm.~2.2]{BPDKB23}
    and gives the proof for the supremum over bounded continuous functions.
    The second part shows that the supremum over the subspace of continuous vanishing functions
    is sufficient.
    
    \begin{enumerate}
        \item 
        Proceed exactly as in \Cref{theorem:RenyiPreconjugate}, replacing $X$ by $\R^d$ and $\C(X)$ by $\C_b(\R^d)$.

        \item 
        Notably, the supremum over $\C_0(\R^d) \subset \C_b(\R^d)$ is smaller, in general.
        We show an approximation possibility via a construction 
        in $\C_0(\R^d)$.
    
        Let $g^* \in \C_b(\R^d)$ and $g^* < 0$ the function realizing the Rényi divergence,
        i.e. the variational formulation.
        The main idea is to construct $\tilde g \in \C_0(\R^d)$ also realizing the Rényi divergence.
        Since $\tilde g \in \C_0(\R^d)$ is now necessary we have to handle the infinite distance points.
        Let $B_R(0) \coloneqq \{x \in \R^d : \|x\|_2 < \infty\} \subset \R^d$ 
        the $d$-dimensional ball with radius $R>0$ around zero. 
        Since $\mu, \theta \in \P(\R^d)$, we have
        \begin{equation*}
            \mu(\R^d \setminus B_R(0)) \to 0
            \quad \text{for} \quad R \to \infty
        \end{equation*}
        especially we can realize the convergence by a function $\eps(R) \in \mathcal O(1/R)$,
        i.e. $\mu(\R^d \setminus B_R(0)) \le \eps(R)$.
        Without loss of generality the same bound holds at least for $\theta$.
    
        Now, since we still take the supremum over $\C_0(\R^d) \ni \tilde g$
        with 
        \begin{equation*}
            \tilde g |_{B_R(0)} = g^*|_{B_R(0)} \quad \text{for any} \quad R > 0,
        \end{equation*}
        we find 
        \begin{equation*}
            \tilde g(x)|_{\R^d\setminus B_R(0)} \in \tilde \eps(\|x\|),
        \end{equation*}
        $\tilde \eps$ monotone decreasing
        where $\tilde \eps(\|x\|) \le \tilde \eps(R) \in \mathcal O(1/\sqrt{R})$.
        Hence, we found a negative vanishing function, 
        approximating the bounded function---realizing the variational formulation---in the sense, 
        that 
        \begin{enumerate}
            \item 
            $\tilde g$ is exactly $g^*$ whenever the mass of $\mu$ and $\theta$ is not vanishing,
            \item 
            otherwise, $\tilde g$ vanishes slower,
            i.e. $|\tilde g|^{\frac{\alpha}{\alpha - 1}}$ increases slower,
            than $\mu$ and at least $\theta$ vanish
            - the crucial second part in the variational formulation.
        \end{enumerate}
    \end{enumerate}
    Since $\ln\left(\int_{X} |\cdot|^\frac{\alpha}{\alpha - 1} \diff{\theta}\right)$ 
    is proper, lower semicontinuous and convex in $\C_0(\R^d)$, 
    we can directly read off a preconjugate of the Rényi divergence with respect to its first component from \eqref{eq:ConvConjFinale}.
\end{proof}

\section{Subgradient descent on the dual problem}
In this subsection we show how a subgradient method can be used to solve the dual problem in the discrete setting, which reads
\begin{equation} \label{eq:dual_discrete}
    \max_{\substack{\Vq \in \R^{2d} \\ L^*\Vq < \MM}} 
    \left\langle \Vq, \begin{bmatrix} \Vr \\ \Vc \end{bmatrix}\right\rangle 
    - \eps 
    \ln\left(\langle (\MM - L^*\Vq)^\frac{\alpha}{\alpha - 1}, \Vr\Vc^{\tT}\rangle_{\tF}\right)
    + C_{\alpha, \eps},
\end{equation}
where
\begin{equation*}
    L \colon \R^{d\times d} \to \R^{2d}, \qquad
    \MP \mapsto \begin{bmatrix}
        \MP\1 \\
        \MP^{\tT}\1
    \end{bmatrix}
\end{equation*}
is the discrete counterpart of $A^*$, where $A$ is defined in \eqref{eq:A}.
We denote the objective function in \eqref{eq:dual_discrete} by $\Phi \colon \{ \Vq \in \R^{2 d}: L^* \Vq < \MM \} \to \R$.

We can use the subgradient method~\cite{S12} with step sizes $(a^{(k)})_{k \in \N}$ and
with some special constraint for the initial value 
$\Vq^{(0)} \in \R^{2d}$
and the subgradient - say $\Vg^{(k)}$ at $\Vq^{(k)}$ -
to guarantee that 
\begin{itemize}
    \item
    in each iteration $k$, the variable $\Vq^{(k+1)}$ is feasible, i.e.
    $L^*\Vq^{(k+1)}
    = L^*\left[\Vq^{(k)} + a^{(k)}\Vg^{(k)}\right] < \MM$,
    which additionally guarantees the differentiability 
    of the objective function in \eqref{eq:dualdualSinkhornProblem} 
    and 
    
    \item
    the functional value decreases in each iteration 
    step $k$.
\end{itemize}

\begin{remark}[Step size choice]
    Note that the domain of our dual formulation is open,
    hence projected (sub)gradient ascent can not be applied directly.
    We employ an Armijo line search type heuristic by scaling a fixed initial step size 
    until the domain constraint is fulfilled.
    However, since the problem is strictly concave,
    this choice of step size should solve the problem.
    In the numerical experiments 
    the proposed method shows great potential.
    However, it is much more slower 
    in terms of convergence speed than the proposed primal method.
\end{remark}

Notice that \eqref{eq:primal-dual_discrete} already guarantees that the entries of the transport plan are positive whenever $r_i c_j > 0$.
By \eqref{eq:primal-dual_discrete}, the entries of the OT can be computed 
by defining $\tilde{p}_{i,j} = (m_{i, j} - q_{i} - q_{d + j})^{\frac{1}{\alpha - 1}}r_ic_j$ 
and then normalizing the resulting matrix to have a total entry sum of one.
We summarize this in \Cref{algo:two}.

\begin{algorithm}[H]
    \caption{Subgradient ascent method to solve problem \eqref{eq:dualdualSinkhornProblem}---D-$\RR_\alpha$ROT.} 
    \label{algo:two}
    \SetAlgoLined
    \KwData{marginals $\Vr, \Vc \in \Sigma_N$, distance matrix $\MM$, regularization parameter $\eps > 0$,
    and order $\alpha \in (0,1)$.}
    \KwResult{Dual solution $\Vq$ to \eqref{eq:dual_discrete}.}
    $\Vq^{(0)}\gets -\1_{2d}$
    \Comment{Ensures $L^*\Vq^{(0)} < \MM$.} \\
    \For{$k = 1,2,3,\ldots$}{
        $\MY^{(k)}\gets L^*\Vq^{(k)} - \MM$ \\
        $\Vg^{(k)}\gets [\Vr, \Vc]^{\tT} + \partial \RR_\alpha(\;\cdot\; \mid \Vr\Vc^{\tT})^*[\MY^{(k)}]$ \\
        $a^{(k)}\gets 10, t\gets 0.5$\\ 
        \While{$\max(L^*(\Vq^{(k)} + a^{(k)}\Vg^{(k)}) - \MM)\ge 0$ 
        or 
        $\Phi(\Vq^{(k)} + a^{(k)}\Vg^{(k)}) < \Phi(\Vq^{(k)})$}{
            $a^{(k)}\gets a^{(k)}\cdot t$
            \Comment{Ensures feasibility of iterates
            and increase of functional value.}
        }
        $\Vq^{(k+1)}\gets \Vq^{(k)} + a^{(k)}\Vg^{(k)}$
    }
\end{algorithm}

\end{appendices}

\bibliography{Bibliography}

\end{document}